\documentclass{amsart}[12pt]
\usepackage[curve,matrix,arrow,frame,tips]{xy}
\usepackage{enumerate}
\usepackage{amssymb}
\usepackage{stmaryrd}
\usepackage[pdfborder={0 0 0}, pdftitle={The universal property of bordism of commuting involutions}, pdfauthor={Markus Hausmann, Stefan Schwede},pdfstartview={FitH}]{hyperref}

\DeclareMathOperator{\colim}{colim}
\DeclareMathOperator{\el}{el}
\DeclareMathOperator{\epi}{epi}
\DeclareMathOperator{\Hom}{Hom}
\DeclareMathOperator{\Id}{Id}
\DeclareMathOperator{\op}{op}
\DeclareMathOperator{\tr}{tr}
\DeclareMathOperator{\res}{res}
 
\newcommand{\mF}{{\mathbb F}}
\newcommand{\mR}{{\mathbb R}}
\newcommand{\mS}{{\mathbb S}}
\newcommand{\mZ}{{\mathbb Z}}
\newcommand{\cG}{{\mathcal G}}
\newcommand{\cN}{\mathcal N}
\newcommand{\SH}{{\mathcal{SH}}}
\newcommand{\GH}{{\mathcal {GH}}}
\newcommand{\elRO}{{\el^{R O}_2}}
\newcommand{\bmO}{\mathbf{mO}}
\newcommand{\bBO}{\mathbf{BO}}
\newcommand{\bMO}{\mathbf{MO}}
\newcommand{\iso}{\cong}
\newcommand{\sm}{\wedge}
\newcommand{\tensor}{\otimes}
\newcommand{\xra}{\xrightarrow}
\newcommand{\xla}{\xleftarrow}
\newcommand{\upi}{{\underline \pi}}
\newcommand{\td}[1]{\langle #1\rangle}
\renewcommand{\to}{\longrightarrow}
\newcommand{\dbr}[1]{\llbracket #1\rrbracket}
\newcommand{\lr}[1]{(\!( #1 )\!)}

\numberwithin{equation}{section}
\newtheorem{theorem}[equation]{Theorem}
\newtheorem{cor}[equation]{Corollary}
\newtheorem{prop}[equation]{Proposition}

\theoremstyle{definition}
\newtheorem{defn}[equation]{Definition}
\newtheorem{rk}[equation]{Remark}
\newtheorem{eg}[equation]{Example}
\newtheorem{construction}[equation]{Construction}

\begin{document}

\title[The universal property of bordism of commuting involutions]{The universal property of bordism\\of commuting involutions}
\author{Markus Hausmann}
\author{Stefan Schwede}
\email{hausmann@math.uni-bonn.de, schwede@math.uni-bonn.de}
\address{Mathematisches Institut, Universit{\"a}t Bonn, Germany}

\begin{abstract}
  We propose a formalism to capture the structure of the equivariant bordism
  rings of smooth manifolds with commuting involutions.
  We introduce the concept of an {\em oriented $\elRO$-algebra},
  an algebraic structure featuring representation graded rings
  for all elementary abelian 2-groups, connected by restriction homomorphisms,
  a pre-Euler class, and an inverse Thom class;
  this data is subject to one exactness property.
  Besides equivariant bordism, oriented global ring spectra also give rise to
  oriented $\elRO$-algebras, so examples abound.
  Inverting the inverse Thom classes yields a global 2-torsion group law.
  In this sense, our oriented $\elRO$-algebras are delocalized generalizations of 
  global 2-torsion group laws.

  Our main result shows that equivariant bordism for elementary abelian 2-groups
  is an initial oriented $\elRO$-algebra.
  Several other interesting equivariant homology theories can also be characterized,
  on elementary abelian 2-groups, by similar universal properties.
  We prove
  that stable equivariant bordism
  is an initial $\elRO$-algebra with an invertible orientation;
  that Bredon homology with constant mod 2 coefficients
  is an initial $\elRO$-algebra with an additive orientation;
  and that Borel equivariant homology with mod 2 coefficients
  is an initial $\elRO$-algebra with an orientation that is both additive and invertible.\medskip
\end{abstract}

\maketitle

\tableofcontents

\section*{Introduction}

The purpose of this paper is to unveil the universal property of the equivariant bordism rings
for elementary abelian 2-groups. More precisely, we will show that
equivariant bordism of manifolds with commuting involutions
is an initial example of a specific kind of algebraic structure that we call
an {\em oriented $\elRO$-algebra}. The superscript `$R O$' refers to the
feature that gradings by real representations are hard wired into our theory.
Our formalism also allows to characterize several other interesting equivariant
homology theories, on elementary abelian 2-groups,
by similar universal properties in the category of oriented $\elRO$-algebras;
examples include stable equivariant bordism,
Bredon homology with constant mod 2 coefficients, and Borel cohomology with mod 2 coefficients.
In short, stable equivariant bordism is initial among
oriented $\elRO$-algebras equipped with an invertible orientation;
Bredon homology is initial among
oriented $\elRO$-algebras equipped with an additive orientation;
and Borel cohomology is initial among
oriented $\elRO$-algebras equipped with an orientation that simultaneously invertible and additive.

The key new insight that we wish to promote in this paper is that the
algebraic structure of an orientable $\elRO$-algebra arises very naturally
in the study of equivariant homology theories for elementary abelian 2-groups,
and that it greatly clarifies and conceptualizes the `true' nature of several prominent
equivariant theories. Indeed, for all the theories mentioned in the previous paragraph,
there is a vast literature of explicit calculations,
in particular for the group with two elements.
Examples include \cite{alexander, beem-2mod4, beem-C_4,beem-rowlett, boardman:revisited, carlisle:cobordism, conner-floyd, firsching, sinha-C_2 bordism, stong:unoriented_finite, stong:Z_2^k-actions}
for equivariant bordism,
\cite{broecker-hook, firsching, sinha-C_2 bordism} for stable equivariant bordism,
and \cite{basu-dey, caruso:operations, dugger:atiyah-hirzebruch, hausmann-schwede:Bredon, holler-kriz, hu-kriz, kriz-lu, kriz:notes_equivariant, sikora:EMspectra, sikora-yan} for Bredon homology and cohomology.
Many of the answers are in terms of lists of generators and relations.
While often quite explicit, the present authors do not find many
of these calculations conceptually enlightening,
and contemplating the answers left us wondering what it actually means to `calculate'
or understand such kinds of mathematical structures.
Furthermore, none of the previous results we are aware of relate
to Quillen's algebraic interpretation \cite{quillen:formal_groups}
of the bordism ring $\mathcal{N}_*$ as the $2$-torsion Lazard ring,
carrying the universal formal group law with trivial $2$-series. 

Our approach shifts the emphasis away from explicit calculations of equivariant
bordism or cohomology rings for individual groups of equivariance;
we focus on the question: `What do the equivariant rings represent?'
A lesson we learned is that for getting interesting answers, one might want to:
\begin{itemize}
\item consider all elementary abelian 2-groups simultaneously,
  keeping track of the contravariant functoriality in group homomorphisms, and
\item remember and exploit the $R O$-grading by real representations.
\end{itemize}
We implement these two aspects in the structure an {\em oriented $\elRO$-algebra}.
Central to this notion are two distinguished classes of elements,
the `pre-Euler classes' $a_\lambda$ and the `inverse Thom classes' $t_\lambda$,
indexed by all nontrivial characters $\lambda$ of elementary abelian $2$-groups.
The structure is subject to one key regularity property:
multiplication by the pre-Euler classes of nontrivial characters must be injective,
with cokernel given by the value at the kernel of the given character,
while the classes $t_\lambda$ must restrict to the unit at the trivial group.
In the case of equivariant bordism, $a_\lambda$ equals the bordism class
of the inclusion $\{0\}\to S^\lambda$ of the origin into
the representation sphere associated to $\lambda$,
while $t_\lambda$ is given by the bordism class of the identity map $S^\lambda\to S^\lambda$.
In the case of Bredon homology, the classes $a_\lambda$ and $t_\lambda$ generate
the entire representation-graded ring,
subject to certain explicit polynomial relations \cite{hausmann-schwede:Bredon}.
The precise definition of an oriented $\elRO$-algebra will be given in 
Definition \ref{def:oriented R-algebra}.
Notably, the category of oriented $\elRO$-algebras contains the category
of $2$-torsion formal group laws as a reflective subcategory
via the `complete $t$-invertible' objects, see Section \ref{sec:global-2-torsion group}.
Hence, in very loose terms, an oriented $\elRO$-algebra can be viewed
as a delocalized and decompleted version of a $2$-torsion formal group law.
\smallskip

After this attempt to provide motivation for our formalism,
we now list our main results.
The first of these, to be proved as Theorem \ref{thm:main} below,
makes the title of this paper precise:\smallskip

{\bf Theorem A.} {\em Equivariant bordism is an initial oriented $\elRO$-algebra.}\smallskip

We emphasize that in contrast to all previous results that we are aware of,
Theorem A characterizes geometric equivariant bordism in absolute terms rather than
relative to the non-equivariant bordism ring $\cN_*$.
In particular, no a priori connection to formal group laws,
nor any $\cN_*$-algebra structure, are directly built into the definition of an
oriented $\elRO$-algebra. Still, the regularity properties of the pre-Euler classes
together with the functoriality in elementary abelian $2$-groups
indirectly induce nontrivial structure also over the trivial group,
namely that of a $2$-torsion formal group law.
As such, Theorem A can be viewed as a refinement
of Quillen's theorem \cite{quillen:formal_groups};
however, our Theorem A does not provide an independent proof of Quillen's result,
as we rely on it as central input.
\smallskip

We now turn to our next result. Stable equivariant bordism is a certain localization of equivariant
bordism first considered by Br{\"o}cker and Hook \cite[\S 2]{broecker-hook}.
We call an orientation of an $\elRO$-algebra {\em invertible}
if the inverse Thom classes for all characters of elementary abelian 2-groups
act invertibly, see Definition \ref{def:invertible}.
Stable equivariant bordism is, essentially by design, the oriented $\elRO$-algebra
obtained from equivariant bordism by inverting all those inverse Thom classes.
The following result is thus an immediate corollary of Theorem A:\smallskip

{\bf Theorem B.} {\em Stable equivariant bordism is an initial invertibly
  oriented $\elRO$-algebra.}\smallskip

Theorem B will be proved as Corollary \ref{cor:stable bordism} below.
As we explain at the end of Section \ref{sec:global-2-torsion group},
combining Theorem B with the main result
of Br{\"o}cker and Hook \cite{broecker-hook}
yields an independent proof of a recent result of the first author,
namely that the global 2-torsion group law carried by the global Thom spectrum $\bMO$ is initial,
see \cite[Theorem D]{hausmann:group_law}.
In fact, this earlier result of the first author was a key motivation for the present project,
and some of our arguments are inspired by arguments from \cite{hausmann:group_law}.\smallskip

The next interesting example is Bredon homology with coefficients
in the constant Mackey functor with value $\mF_2$; by definition,
the $\mZ$-graded part of this theory is a copy of $\mF_2$ concentrated in dimension 0.
These Bredon homology groups have attracted a certain amount of attention,
and we listed specific references with many explicit calculations above.
In this paper we show an algebraic universal property for the coefficients of Bredon homology.
The oriented $\elRO$-algebra of Bredon homology has a unique orientation,
and this orientation is {\em additive} in the sense of Definition \ref{def:additive}.
We use the adjective `additive' because the formal group law
arising from an additive orientation is the additive one, see
Example \ref{eg:additive is additive}.
\smallskip

{\bf Theorem C.} {\em Bredon homology with constant $\mF_2$-coefficients is an initial
  additively oriented $\elRO$-algebra.}\smallskip

Theorem C is proved as Theorem \ref{thm:Bredon initial} below.
In our paper \cite{hausmann-schwede:Bredon}, we provide an explicit presentation
of the representation-graded Bredon homology rings with constant $\mF_2$-coefficients.\smallskip

The last example whose universal property we discuss is the 
mod 2 global Borel spectrum $b(H\mF_2)$, representing
Borel equivariant cohomology with $\mF_2$-coefficients.
The associated $\elRO$-algebra has a unique orientation, and this unique orientation
is both additive and invertible.
We prove in Theorem \ref{thm:Borel initial} that
mod 2 global Borel cohomology is characterized by such an orientation:\smallskip

{\bf Theorem D.} {\em Borel cohomology with $\mF_2$-coefficients is initial among
  oriented $\elRO$-algebras whose orientation is both additive and invertible.}\medskip

In Section \ref{sec:classical} we explain
how our theory can be used to reprove various
previous results on equivariant bordism,
in particular the Conner--Floyd exact sequence \cite{conner-floyd},
Alexander's description \cite{alexander} of an explicit basis over $\mathcal{N}_*$
for the bordism ring of involutions,
and Firsching's pullback square \cite{firsching}.
In all cases, the results in fact hold for a large class of oriented $\elRO$-algebras,
and not just for the initial one given by equivariant bordism.\smallskip

{\bf Acknowledgments.}
The authors are members of the Hausdorff Center for Mathematics
at the University of Bonn (DFG GZ 2047/1, project ID 390685813).
A substantial part of the work for this paper was done while the second author
spent the summer term 2023 on sabbatical at Stockholm University, with financial support
from the Knut and Alice Wallenberg Foundation;
the second author would like to thank SU for the hospitality
and stimulating atmosphere during this visit.

\section{Oriented \texorpdfstring{$\elRO$}{el}-algebras}

In this section we introduce our main characters, the {\em oriented $\elRO$-algebras}.
One of the key points of this paper is the insight that to unveil the universal property of
equivariant bordism for elementary abelian 2-groups, one needs to take the grading by
representations into account.
More precisely, the proper book keeping needs the `effective' part
of the $R O$-grading,
i.e., the part that keeps track of the integer graded bordism groups of representation spheres.

We start by fixing notation and terminology to deal with the $R O$-grading.
For every finite group $G$, we define a submonoid $I_G$ of the real representation
ring $R O(G)$ by
\[ I_G \ = \ \{ m\in R O(G)\ : \ m = k- V \text{ for some $k\in\mZ$ and some $G$-representation $V$}\} \ .\]
So the elements of $I_G$ are formal differences of integers and actual representations.
These submonoids are closed under restriction along homomorphisms
$\alpha:K\to G$ between finite groups, i.e., $\alpha^*:R O(G)\to R O(K)$
sends $I_G$ to $I_K$. We emphasize that unless the group $G$ is trivial,
the abelian monoid $I_G$ is {\em not} a group.

\begin{construction}[The category $\elRO$]
  We introduce a category $\elRO$ of effective representations of elementary abelian 2-groups.
  Objects of $\elRO$ are pairs $(A,m)$, where $A$ is an elementary abelian 2-group and $m\in I_A$.
  Morphisms from $(B,n)$ to $(A,m)$ in the category $\elRO$
  are group homomorphisms $\alpha:B\to A$ such that $\alpha^*(m)=n$.
  
  We give the category $\elRO$ a symmetric monoidal structure as follows.
  On objects, we set
  \[ (A,m)\times (B,n)\ = \ (A\times B, m\oplus n)\ , \]
  where
  \[ m\oplus n \ = \ p_1^*(m) + p_2^*(n) \]
  is the external sum, with $p_1:A\times B\to A$ and $p_2:A\times B\to B$ the projections to the two factors.
  On morphisms, the symmetric monoidal structure is given by product of homomorphisms.
  The symmetry and associativity isomorphisms are the symmetry and associativity
  isomorphism for the product of groups.
  The unit object is $(1,0)$, the pair consisting of the trivial group $1$
  and its zero representation. A symmetric monoidal structure on a category is at the same time
  a symmetric monoidal structure on its opposite.
\end{construction}

\begin{defn}
  An {\em $\elRO$-algebra} is a lax symmetric monoidal functor
  \[ X \ : \ (\elRO)^{\op}\ \to \ \mF_2\text{-mod} \]
  to the category of $\mF_2$-vector spaces under tensor product.
  A {\em morphism} of  $\elRO$-algebras is a monoidal transformation.
\end{defn}

Since symmetry is part of the definition of an $\elRO$-algebra, a more precise
name for these objects would have been {\em commutative} $\elRO$-algebras.
Since we will not consider non-commutative $\elRO$-algebras
(i.e., non-symmetric lax monoidal functors), we have opted for the simpler name.

Expanding the definition reveals that an
 $\elRO$-algebra $X$ consists of
\begin{itemize}
\item $\mF_2$-vector spaces $X(A,m)$
for every elementary abelian 2-group $A$ and every $m\in I_A$,
\item a restriction homomorphism
\[ \alpha^* \ : \ X(A,m)\ \to \ X(B,\alpha^*(m))\ , \]
for every group homomorphism $\alpha:B\to A$ and every $m\in I_A$,
\item
  associative and commutative pairings
\[ \times \ : \ X(A,m)\tensor X(B,n)\ \to \ X(A\times B, m\oplus n) \]
with a multiplicative unit $1\in X(1,0)$.
\end{itemize}
Moreover, the restriction homomorphisms must be contravariantly functorial
and strictly compatible with the multiplicative structure.
A morphism $f:X\to Y$ of  $\elRO$-algebras consists of
additive maps $f(A,m):X(A,m)\to Y(A,m)$ such that $f(1,0)(1)=1$
and 
\[ f(A\times B,m\oplus n)(x\times y)\ = \ f(A,m)(x)\times f(B,n)(y)\]
for all $x\in X(A,m)$ and $y\in X(B,n)$.

The multiplicative structure could alternatively be given in `internal form'
\[ \cdot \ : \ X(A,m)\tensor X(A,n)\ \to \ X(A,m+ n)\ , \]
as the composite
\[  X(A,m)\tensor X(A,n)\  \xra{\ \times \ } \  X(A\times A,m\oplus n)\
  \xra{\Delta^*}  X(A,m+ n)\ , \]
where $\Delta:A\to A\times A$ is the diagonal.
For fixed $A$ and varying $m$,
this makes the collection of groups $X(A,m)$ into a
commutative $I_A$-graded $\mF_2$-algebra.
We will often employ the notation $X(A,\star)$ for this $I_A$-graded $\mF_2$-algebra.
The abelian groups $X(A,k)$ for $k\in \mZ$
form a commutative $\mZ$-graded $\mF_2$-algebra,
the {\em integer graded subring} of this $I_A$-graded $\mF_2$-algebra $X(A,\star)$,
which we denote by $X(A)_*$.
The external multiplication can be recovered from the internal multiplication  as the composite
\[ 
  X(A,m)\tensor X(B,n)\
  \xra{p_1^*\tensor p_2^*} \
  X(A\times B,p_1^*(m))\tensor X(A\times B,p_2^*(n))\
  \xra{\ \cdot\ }  X(A\times B,m\oplus n)\ .  \]

\begin{rk}
  The cautious reader might wonder why we consider functors on {\em isomorphism classes}
  of representations, and why we abuse notation and use the same symbol for a representation
  and for its class in $R O(G)$. In other words: why does our theory not
  take automorphisms of representations into account?
  The reason is that the automorphisms of real representations,
  and even all equivariant homotopy self-equivalences of linear spheres,
  are invisible to the theories we consider.
  We explain this in detail in Theorem \ref{thm:oriented consequences}, but here is already a brief sketch.
  All theories relevant to us arise from
  global ring spectra $E$ in which $\tr_1^C(1)=0$, i.e., the transfer from the trivial
  group to the two element group $C$ vanishes in $\pi_0^C(E)$. This property has strong consequences
  for the coefficients of the theory: for every compact Lie group $G$,
  the unit ring map $\pi_0^G(\mS)\to\pi_0^G(E)$ from the $G$-equivariant 0-stem sends all
  units to 1. Consequently, all $G$-equivariant homotopy self-equivalences of
  spheres of real $G$-representations induce the identity
  in the $G$-homology theory represented by $E$.
\end{rk}

In the following we write $C=\{\pm 1\}$ for the group with two elements.
A {\em character} of a group $G$ is a homomorphism $\lambda:G\to C$.
We shall consistently confuse such characters with the associated 1-dimensional
$G$-representation on $\mR$ with $G$-action $g\cdot x=\lambda(g)\cdot x$,
and with the class of this representation in the representation ring $R O(G)$.
We write $\sigma$ for the sign representation of $C$ on $\mR$, corresponding
to the identity character of $C$.

\begin{defn}[Orientable $\elRO$-algebras]\label{def:oriented R-algebra}
  An {\em orientable $\elRO$-algebra} is a pair $(X,a)$ consisting of an $\elRO$-algebra $X$
  and a class $a\in X(C,-\sigma)$ with the following property:
  for every elementary abelian 2-group $A$, every $m\in I_A$ and every
  nontrivial $A$-character $\lambda$, the sequence
  \[ 0 \ \to \ X(A,m)\ \xra{a_\lambda\cdot-}\ X(A,m -\lambda)\
    \xra{\res^A_K}\ X(K,\res^A_K(m)-1) \ \to \ 0 \]
  is exact, where $K$ is the kernel of $\lambda$, and $a_\lambda=\lambda^*(a)$.
  We refer to the class $a$ as the {\em pre-Euler class} of the orientable $\elRO$-algebra.
  A {\em morphism} of orientable $\elRO$-algebras is a morphism of $\elRO$-algebras
  that takes the pre-Euler class to the pre-Euler class.
\end{defn}

An easy but useful fact is that a morphism
of orientable $\elRO$-algebras is already an isomorphism
if its restriction to all the integer graded subrings is:

\begin{prop}\label{prop:Z-graded iso suffices}
  Let $f:X\to Y$ be a morphism of orientable $\elRO$-algebras.
  If all the homomorphisms of $\mZ$-graded rings  $f(A)_*:X(A)_*\to Y(A)_*$ are bijective,
  then $f$ is an isomorphism of $\elRO$-algebras.
\end{prop}
\begin{proof}
  We show that for all elementary abelian 2-groups $A$, all $k\in\mZ$ and all $A$-representations
  $W$ with $W^A=0$, the map $f(A,k-W):X(A,k-W)\to Y(A,k-W)$ is an isomorphism.
  We argue by induction on the dimension of $W$.
  For $W=0$ the claim is true by hypothesis.
  If $W$ is nonzero, we choose a nontrivial $A$-character $\lambda$ such that
  $W=V\oplus\lambda$. Then the outer two maps in the following commutative diagram
  are isomorphisms by induction, because $V$ has smaller dimension than $W$:
  \[ \xymatrix{
      0 \ar[r] & X(A,k-V)\ar[d]_{f(A,k-V)} \ar[r]^-{a_\lambda\cdot-}&
       X(A,k -W)\ar[d]^{f(A,k-W)}\ar[r]^-{\res^A_K} & X(K,k- 1-V|_K) \ar[d]^{f(K,k-1- V|_K)} \ar[r]& 0 \\
       0 \ar[r] & Y(A,k-V) \ar[r]_-{a_\lambda\cdot-}&
       Y(A,k -W)\ar[r]_-{\res^A_K} & Y(K,k- 1-V|_K)  \ar[r]& 0 
     } \]
   Since both rows are exact, the middle map is an isomorphism.
\end{proof}

For an orientable $\elRO$-algebra $(X,a)$ we already employed the notation
$a_\lambda=\lambda^*(a)$  for $A$-characters $\lambda$.
Then $a_\lambda=0$ when $\lambda$ is the trivial character;
and multiplication by $a_\lambda$ is injective for all nontrivial characters $\lambda$.
If $V$ is an $A$-representation, we define
\[ a_V \ = \  \prod_{\lambda\in \Hom(A,C)} a_\lambda^{m_\lambda}\ \in \ X(A,-V)\ , \]
where $m_\lambda$ is the multiplicity of the character $\lambda$ in $V$.
Then $a_V=0$ whenever $V$ has nontrivial $A$-fixed points;
and multiplication by $a_V$ is injective whenever $V^A=0$.
These $a$-classes are multiplicative and natural for restriction homomorphisms, i.e.,
\[ a_V\cdot a_W\ = \ a_{V\oplus W}\text{\quad and\quad }\beta^*(a_V)\ = \ a_{\beta^*(V)}\]
for all homomorphisms $\beta:B\to A$ between elementary abelian 2-groups.\smallskip

In an orientable $\elRO$-algebra, divisibility by the $a$-classes of general representations
can be tested separately on the isotypical summands.
The next proposition makes this precise.

\begin{prop}\label{prop:coprime}
  Let $(X,a)$ be an orientable $\elRO$-algebra, and let $A$ be an elementary abelian 2-group.
  \begin{enumerate}[\em (i)]
  \item 
  Let $\lambda$ be a nontrivial $A$-character with kernel $K$.
  Let $x$ and $y$ be homogeneous elements of $X(A,\star)$.
  Suppose that $\res^A_K(y)$ is a non zero-divisor in $X(K,\star)$.
  Then $x$ is divisible by $a_\lambda^n$ if and only if $x y$ is divisible by $a_\lambda^n$.
\item
  Let $V$ be an $A$-representation with trivial $A$-fixed points.
  Let $z$ be an homogeneous element of  $X(A,\star)$ that is divisible
  by $a_\lambda^{m_\lambda}$ for every nontrivial character $\lambda$ of $A$,
  where $m_\lambda$ is the multiplicity of $\lambda$ in $V$.
  Then $z$ is divisible by $a_V$.
\end{enumerate}
\end{prop}
\begin{proof}
  (i) If $x$ is divisible by $a_\lambda^n$, then clearly so is $x y$.
  We prove the other direction by induction on $n$.
  The statement is vacuous for $n=0$.
  Now let $n\geq 1$. By the induction hypothesis we know that $x$ is divisible by $a^{n-1}_\lambda$.
  Hence there exists an element $x'$ such that $x=x' a^{n-1}_\lambda$.
  Since $x' y a^{n-1}_\lambda=x y$, and $x' y$ is the unique element with this property,
  we know that $x' y$ must be divisible by $a_\lambda$.
  Hence the restriction $\res^A_K(x' y)=\res^A_K(x')\res^A_K(y)$ is trivial.
  Since $\res^A_K(y)$ is a non zero-divisor, this implies that $\res^A_K(x')=0$.
  Hence $x'$ is divisible by $a_\lambda$, and therefore $x$ is divisible by $a_\lambda^n$.
  This finishes the proof.

  (ii) 
  We argue by induction on the number of nontrivial isotypical summands of $V$.
  There is nothing to show if $V=0$.
  Otherwise we let $\chi$ be a nontrivial $A$-character that occurs in $V$.
  Then $V=\chi^{m_\chi}\oplus U$ for an $A$-representation $U$ with
  fewer nontrivial isotypical summands.
  We let $z\in X(A,m)$ be divisible by $a_\lambda^{m_\lambda}$ for all
  nontrivial $A$-characters $\lambda$.
  Then we know by induction that $z$ is divisible by $a_U$,
  say $z=x\cdot a_U$ for some $x\in X(A,m+U)$.
  We let $K$ denote the kernel of $\chi$. Since $\chi$ does not occur in $U$,
  the representation $U$ has trivial $K$-fixed points.
  So $\res^A_K(a_U)=a_{U|_K}$ is a non zero-divisor in $X(K,\star)$.
  Since $x\cdot a_U=z$ is divisible by $a_\chi^{m_\chi}$, part (i)
  shows that $x$ is divisible by $a_\chi^{m_\chi}$. So $z=x\cdot a_U$
  is divisible by $a_\chi^{m_\chi}\cdot a_U=a_V$.
\end{proof}

\begin{defn}[Oriented $\elRO$-algebras]
  An {\em inverse Thom class} of an orientable $\elRO$-algebra $(X,a)$
  is a class $t\in X(C,1-\sigma)$
  such that $\res^C_1(t)=1$, the multiplicative unit in $X(1,0)$.
  An {\em oriented $\elRO$-algebra} is a triple $(X,a,t)$ consisting of an
  orientable $\elRO$-algebra and an inverse Thom class.
\end{defn}

\begin{rk}[The choice of inverse Thom class]\label{rk:choice of inverse Thom class}
  If $(X,a)$ is an orientable $\elRO$-algebra, then in particular the restriction homomorphism
  \[ \res^C_1\ : \ X(C,1 -\sigma)\ \to \ X(1,0) \]
  is surjective. So an inverse Thom class exists.
  Moreover, any two inverse Thom classes differ by an element of the form $a\cdot y$,
  for a unique $y\in X(C,1)$; so the set of inverse Thom classes is a torsor over
  the additive group $X(C,1)$. In some important cases, the group $X(C,1)$ is trivial,
  so that there is a unique inverse Thom class.
  This holds for example for equivariant bordism, and for Bredon homology.
  
  The choice of inverse Thom class is essential,
  and the automorphism group of an orientable $\elRO$-algebra
  need not act transitively on the set of inverse Thom classes.
  An explicit example is the orientable $\elRO$-algebra $(H\sm H)^\sharp$
  arising from the `global dual Steenrod algebra',
  where $H=H\underline\mF_2$ is the global Eilenberg-MacLane spectrum
  of the constant Mackey functor.
  As we plan to discuss elsewhere, $(H\sm H)^\sharp$ has exactly four inverse Thom classes,
  falling into two distinct orbits under the automorphism group of $(H\sm H)^\sharp$.
\end{rk}

Similarly as for the pre-Euler class, we will use the notation
\[ t_\lambda\ = \ \lambda^*(t)\ \in \ X(A,1-\lambda)\]
for an inverse Thom class $t$ and an $A$-character $\lambda$.
The inverse Thom class of the trivial character is the multiplicative unit 1.
And we extend the inverse Thom classes to general $A$-representations by defining
\begin{equation}\label{eq:define_t_V}
  t_V \ = \  \prod_{\lambda\in\Hom(A,C)} t_\lambda^{m_\lambda}\ , 
\end{equation}
where $m_\lambda$ is the multiplicity of the character $\lambda$ in $V$.
These inverse Thom classes are also multiplicative and natural for restriction, i.e.,
\[ t_V\cdot t_W\ = \ t_{V\oplus W}\text{\quad and\quad }\beta^*(t_V)\ = \ t_{\beta^*(V)}\]
for all homomorphisms $\beta:B\to A$ between elementary abelian 2-groups.

\begin{prop}\label{prop:generation}
  Let $(X,a,t)$ be an oriented $\elRO$-algebra. 
  \begin{enumerate}[\em (i)]
  \item For every subgroup $B$ of $A$, and for all $m\in I_A$, the restriction homomorphism
    \[ \res^A_B\ : \ X(A,m)\ \to \ X(B,\res^A_K(m))  \]
    is surjective.
  \item The $I_A$-graded algebra $X(A,\star)$ is generated by its $\mZ$-graded subalgebra $X(A)_*$
  and the classes $a_\lambda$ and $t_\lambda$ for all nontrivial $A$-characters $\lambda$.
  \end{enumerate}
\end{prop}
\begin{proof}
  We prove both statements together by induction over the rank of $A$.
  The induction starts when $A$ is the trivial group, in which case there is nothing to show.
  Now we let $A$ be a nontrivial elementary abelian 2-group, and we assume
  that (i) and (ii) hold for all proper subgroups of $A$.

  To prove (i) we may assume that $B$ is a proper subgroup of $A$.
  We write $m=k-W$ for an $A$-representation $W$ with $W^A=0$.
  By (ii) for $B$, it suffices to show that all classes of the form
  $x\cdot a_U\cdot t_V$ are in the image of the restriction homomorphism,
  whenever $U$ and $V$ are $B$-representations such that $U\oplus V=W|_B$,
  and $x\in X(B)_k$.
  Because $W$ is a sum of 1-dimensional $A$-representations,
  we may choose an $A$-equivariant decomposition $W=\bar U\oplus\bar V$
  such that  $\bar U|_B\iso U$ and $\bar V|_B\iso V$.
  We also choose a homomorphism $r:A\to B$ that is a retraction to the inclusion. Then
  \[ \res^A_B( r^*(x)\cdot a_{\bar U}\cdot t_{\bar V})\ =\ x\cdot a_U\cdot t_V\ , \]
  and we have shown (i) for $A$.

  For (ii) we also fix a representation grading $m=k-W$, with $W^A=0$.
  We prove the claim for all classes in $X(A,k-W)$ by induction over the dimension of $W$.
  There is nothing to show if $W=0$.
  Otherwise we write $W=V\oplus\lambda$
  for an $A$-representation $V$ and a nontrivial $A$-character $\lambda$, with kernel $K$.
  By part (i), the restriction map  $\res^A_K:X(A,k-1-V)\to X(K,k-1-V|_K)$
  is surjective. So given $x\in X(A,k-W)$,
  there is a class $z\in X(A,k-1-V)$ such that $\res^A_K(z)=\res^A_K(x)$. Then
  \[ \res^A_K(x + t_\lambda\cdot z)\ = \ \res^A_K(x) + \res^A_K(z)\ = \ 0 \ .\]
  The fundamental exactness property of an orientable $\elRO$-algebra
  provides a class $y\in X(A,k-V)$
  such that $a_\lambda\cdot y= x + t_\lambda\cdot z$.
  By the inductive hypothesis, the classes $y$ and $z$ are $X(A)_*$-linear combinations of
  products of $a$-classes and $t$-classes.
  Hence the same is true for $x=y\cdot a_\lambda+ z\cdot t_\lambda$,
  and we have shown (ii).
\end{proof}

\section{Oriented global ring spectra}

In this section we explain how orientable global ring spectra
give rise to orientable $\elRO$-algebras,
see Construction \ref{con:elalg} and Theorem \ref{thm:oriented_ring2oriented_el},
leading to a wealth of interesting examples.
For the purpose of this paper, a {\em global ring spectrum}
is a commutative monoid object in the global stable homotopy category of \cite[Section 4.4]{schwede:global},
under the globally derived smash product \cite[(4.2.25)]{schwede:global}.
So informally speaking, a global ring spectrum is a consistent collection
of $G$-equivariant homotopy ring spectra, for every compact Lie group $G$.

In this paper, we mostly care about the underlying $G$-spectra
of a global spectrum when $G$ is an elementary abelian 2-group.
So we could have worked in the $\el_2$-global stable homotopy category,
the Bousfield localization of the global stable homotopy category
that is designed so that it only `sees' the values of a global spectrum at the
global family $\el_2$ of elementary abelian 2-groups.
However, since every $\el_2$-global ring spectrum is underlying a global ring spectrum
(in many different ways),
working with global ring spectra is no loss of generality.
Moreover, some interesting results, specifically Proposition \ref{prop:units_to_1}
and Theorem \ref{thm:oriented consequences} below,
refer to general compact Lie groups.\smallskip

We recall that a {\em global functor} in the sense of \cite[Definition 4.2.2]{schwede:global}
is an additive functor from the global Burnside category of \cite[Construction 4.2.1]{schwede:global}
to the category of abelian groups. In more explicit terms, a global functor
specifies values on all compact Lie groups, restriction homomorphisms along continuous group
homomorphisms, and transfers along inclusions of closed subgroups; these data have to satisfy
a short list of explicit relations that can be found after Theorem 4.2.6 of \cite{schwede:global}.
The primary invariant of a global spectrum $E$ is the homotopy group global functor
$\upi_0(E)=\{\pi_0^G(E)\}_G$ that records the 0-th equivariant homotopy groups for all compact Lie groups,
along with the restriction and transfer maps that relate them,
see \cite[Example 4.2.3]{schwede:global}. One of the transfer homomorphisms plays
a particularly important role for this paper, namely the transfer
\[ \tr_1^C\ : \ \pi_0(E)\ \to \ \pi_0^C(E)\]
from the homotopy group of the underlying non-equivariant spectrum
to the $C$-equivariant homotopy group.

In the following proposition, $\mS$ denotes the global sphere spectrum,
so that $\upi_0(\mS)$ is the Burnside ring global functor,
compare \cite[Example 4.2.7]{schwede:global}.
The value of $\upi_0(\mS)$ at a compact Lie group $G$ is the $G$-equivariant stable
stem $\pi_0^G(\mS)=\colim [S^V,S^V]_*^G$, where the colimit is
over finite-dimensional $G$-subrepresentations of a complete $G$-universe.
When $G$ is finite, $\pi_0^G(\mS)$ is isomorphic to the Grothendieck ring
of isomorphism classes of finite $G$-sets.

For every closed subgroup $H$ of $G$, the ring homomorphism
\[ \Phi^H \ : \ \pi_0^G(\mS)\ \to \ \mZ \]
sends the class of a based continuous $G$-map $f:S^V\to S^V$ to the degree
of the restriction $f^H:S^{V^H}\to S^{V^H}$ to $H$-fixed points.
These homomorphisms are jointly injective, i.e., classes in $\pi_0^G(\mS)$
are detected by the collection of fixed point degrees, see for example \cite[Theorem 3.3.15]{schwede:global}.

\begin{prop}\label{prop:units_to_1}
  Let $\td{\tr_1^C}$ denote the global subfunctor of $\upi_0(\mS)$ generated by the transfer
  $\tr_1^C(1)$ in $\pi_0^C(\mS)$. Let $G$ be a compact Lie group.
  \begin{enumerate}[\em (i)]
  \item A class $x\in \pi_0^G(\mS)$
  belongs to $\td{\tr_1^C}(G)$ if and only if for every closed subgroup $H$ of $G$ with finite Weyl group,
  the integer $\Phi^H(x)$ is even.
\item For every unit $u\in \pi_0^G(\mS)$,  the class $u-1$ belongs to $\td{\tr_1^C}(G)$.
\end{enumerate}
\end{prop}
\begin{proof}
  (i) As we show in the proof of \cite[Proposition 6.1.45]{schwede:global},
  the value $\td{\tr_1^C}(G)$ is the subgroup of $\pi_0^G(\mS)$ generated by $2\cdot \pi_0^G(\mS)$
  and by the transfers $t_H^G=\tr_H^G(1_H)$ for all closed subgroups $H$ of $G$ whose Weyl group
  is finite and of even order.
  The degree homomorphism $\Phi^K:\pi_0^G(\mS)\to\mZ$ is additive, so it takes even values on
  all classes in $2\cdot \pi_0^G(\mS)$. If the Weyl group of $H$ in $G$ is finite of even order,
  then there is a closed subgroup $L$ of $G$ that contains $H$ as a normal subgroup of index 2.
  Then
  \[ \Phi^K(t_H^G)\ = \ \Phi^K(\res^G_K(\tr_L^G(\tr_H^L(1_H)))) \ .\]
  The Weyl group $W_G H$ acts freely from the right on the coset space $G/H$, and hence
  it acts freely and smoothly on the smooth closed manifold $(G/H)^K$.
  By \cite[Proposition 3.4.2 (ii)]{schwede:global}, the integer
  $\Phi^K(t_H^G)$ is divisible by the order of $W_G H$, and hence even.
  We have thus shown that every class in $\td{\tr_1^C}(G)$ has even fixed point degrees.

  For the converse we consider a class with even fixed point degrees. By the previous paragraph,
  any such class is congruent modulo $\td{\tr_1^C}(G)$ to a class of the form
  \[  t_{K_1}^G + \dots + t_{K_n}^G \]
  for pairwise non-conjugate closed subgroups $K_1,\dots,K_n$ of $G$ whose Weyl groups are finite and of odd order.
  If the above sum had at least one summand, we assume without loss of generality that
  $K_1$ is maximal with respect to subconjugacy among the groups that occur, i.e.,
  $K_1$ is not subconjugate to $K_i$ for $i=2,\dots,n$.
  Then for $i=2,\dots,n$, the fixed point set $(G/K_i)^{K_1}$ is empty, and thus $\Phi^{K_1}(t_{K_i}^G)=0$.
  Hence
  \[    \Phi^{K_1}(t_{K_1}^G + \dots + t_{K_n}^G )\ = \  \Phi^{K_1}(t_{K_1}^G)\ = \ |W_G K_1|\ .\]
  This contradicts the assumption that all fixed point degrees are even.
  So the above sum must be empty, i.e., the given class belongs to $\td{\tr_1^C}(G)$.
  
  (ii) The degree homomorphism $\Phi^H:\pi_0^G(\mS)\to\mZ$ is multiplicative, so it sends
  all units of $\pi_0^G(\mS)$ to $\{\pm 1\}$. Hence $\Phi^H(u-1)$ is even,
  and part (i) proves the claim.
\end{proof}

Now we introduce orientable global ring spectra, our main source of examples
for orientable $\elRO$-algebras.

\begin{defn}
  A global ring spectrum $E$ is {\em orientable} if the class
  $\tr_1^C(1)$ in the group $E_0^C=\pi_0^C(E)$ is trivial.
\end{defn}

Some important examples of orientable global ring spectra
are the mod 2 global Eilenberg-MacLane spectrum $H\underline\mF_2$,
see Example \ref{eg:HF}, 
and the global Thom spectrum $\bmO$ defined in \cite[Example 6.1.24]{schwede:global}.
These are specific global equivariant forms of the mod 2 Eilenberg-MacLane spectrum
and the classical real bordism spectrum.
The global ring spectrum  $H\underline\mF_2$ represents Bredon homology with constant coefficients,
and the global ring spectrum $\bmO$ realizes the orientable
$\elRO$-algebra $\cN$ of equivariant bordism.
We return to these examples in much detail in Sections \ref{sec:bredon} and \ref{sec:bordism},
respectively.
Further examples that we discuss below are the global Thom ring spectrum  
$\bMO$ defined in \cite[Example 6.1.7]{schwede:global}
that represents stable equivariant bordism, see Corollary \ref{cor:stable bordism}
and the remarks immediately thereafter;
and the global Borel ring spectrum associated to the mod 2 Eilenberg-MacLane spectrum,
see Theorem \ref{thm:Borel initial}.
For these four examples, we establish universal characterizations
of the associated oriented $\elRO$-algebra, see
Theorem \ref{thm:Bredon initial}, Theorem \ref{thm:main},
Corollary \ref{cor:stable bordism} and Theorem \ref{thm:Borel initial}, respectively.
The global smash product of an orientable global ring spectrum with any other global
ring spectrum is again orientable, so this yields many further examples.

The next theorem shows that orientability of a global ring spectrum 
has rather strong consequences for the homotopy group global functor.

\begin{theorem}\label{thm:oriented consequences}
  Let $E$ be an orientable global ring spectrum, and let $G$ be a compact Lie group.
  \begin{enumerate}[\em (i)]
  \item All $G$-equivariant homotopy groups of $E$ are $\mF_2$-vector spaces.
  \item The ring homomorphism
    $\eta_*:\pi_0^G(\mS)\to \pi_0^G(E)=E_0^G$ induced by the unit morphism $\eta:\mS\to E$
    sends all units of $\pi_0^G(\mS)$ to 1.
  \item For every $G$-representation $V$ 
    and every $G$-equivariant based homotopy self-equivalence $\psi:S^V\to S^V$,
    the automorphism $E_*^G(S^\psi)$ of $E_*^G(S^V)$ is the identity.
  \end{enumerate}
\end{theorem}
\begin{proof}
  (i) The vanishing of the class $\tr_1^C(1)$ yields
  \[ 2 \ = \ \res^C_1(\tr_1^C(1))\ = \ 0 \]
  in $\pi_0(E)$. Since all equivariant homotopy groups of $E$
  are modules over $\pi_0(E)$ via inflation, they are all $\mF_2$-vector spaces.

  (ii)
  As $G$ varies over all compact Lie groups, the homomorphisms $\eta_*:\pi_0^G(\mS)\to\pi_0^G(E)$
  form a morphism of global functors $\eta_*:\upi_0(\mS)\to\upi_0(E)$.
  Because $\tr_1^C(1)=0$ in $\pi_0^C(E)$, the global subfunctor $\td{\tr_1^C}$ of
  $\upi_0(\mS)$ maps to 0 under the morphism $\eta_*$.
  Hence by Proposition \ref{prop:units_to_1} (ii), the homomorphism  $\eta_*$
  sends all units in $\pi_0^G(\mS)$ to 1. 

  (iii) The automorphism $E_*^G(S^\psi)$ of $E_*^G(S^V)$ coincides with multiplication
  by the image of $[\psi]\in \pi_0^G(\mS)$ under the ring homomorphism
  $\eta_*:\pi_0^G(\mS)\to\pi_0^G(E)$.
  Since $\psi$ is a homotopy self-equivalence, the class $[\psi]$ is a unit in $\pi_0^G(\mS)$,
  so part (ii) proves the claim.
\end{proof}

\begin{construction}[The $\elRO$-algebra of an orientable global ring spectrum]\label{con:elalg}
  We let $E$ be an orientable global ring spectrum.
  We will now define an associated $\elRO$-algebra $E^\sharp$.
  We write any given $m\in I_A$ as $m=k-V$ for an integer $k$
  and an $A$-representation $V$ with trivial fixed points.
  Then $k$ is unique, and $V$ is unique up to isomorphism. We set
  \[ E^\sharp(A,m)\ = \ E^\sharp(A,k-V)\ = \   E_k^A(S^V)\ .\]
  A key point is that this assignment is independent up to preferred isomorphism
  of the choice of $V$. Indeed, an isomorphism of $A$-representations
  $\psi:V\to W$ induces an equivariant homeomorphism $S^\psi:S^V\to S^W$,
  and hence an isomorphism of equivariant homotopy groups
  \[ E_k^A(S^\psi) \ : \ E_k^A(S^V)\ \to\ E_k^A(S^W)\ .  \]
  Theorem \ref{thm:oriented consequences} (iii) guarantees that this isomorphism
  does not depend on the choice of $\psi$.

  The restriction map
  along a morphism $\alpha:(B,\alpha^*(m))\to (A,m)$ in $\elRO$ is defined as follows.
  We choose an $A$-representation $V$ and a $B$-representation $W$ without fixed points
  so that $m=k-V$ and $\alpha^*(m)=l-W$ for integers $k\geq l$.
  There is then an isomorphism of $A$-representations
  \[ \psi\ :\ \alpha^*(V)\ \xra{\iso} \ W\oplus\mR^{k-l}\ . \]
  We define
    \[ \alpha^* \ : \ E^\sharp(A,m)\ \to \  E^\sharp(B,\alpha^*(m)) \]
  as the composite
  \begin{align*}
    E_k^A(S^V)\
    &\xra{\alpha^*} \  E_k^B(S^{\alpha^*(V)}) \ \xra{E_k^B(S^{\psi})} \
     E_k^B(S^{W\oplus \mR^{k-l}})  \ \xla[\iso]{-\sm S^{k-l}} \  E_l^B(S^W) \ .
  \end{align*}
  The first map is the restriction homomorphism of equivariant homotopy groups
  \cite[Construction 3.1.15]{schwede:global}, and the
  final isomorphism pointing backwards is the suspension isomorphism.
  Theorem \ref{thm:oriented consequences} (iii) guarantees that $\alpha^*$ 
  does not depend on the choice of $\psi$.
  The multiplication pairings are defined as the composite
  \begin{align*}
    E^\sharp(A,m)\times E^\sharp(B,n)\
    &= \  E_k^A(S^V)\times E_l^B(S^W)\\
    &\xra{\ \sm \ } \ \pi_{k+l}^{A\times B}(E\sm S^V\sm E\sm S^W)\\
    &\xra{\ \mu_*\ } \
      \pi_{k+l}^{A\times B}(E\sm S^{V\oplus W})\  =\  E^\sharp(A\times B,m\oplus n)\ .
  \end{align*}
  Here we have exploited that $m+n=(k+l)-(V\oplus W)$ in $I_{A\times B}$.
\end{construction}

The fixed point inclusion $S^0\to S^\sigma$ defines an element $a\in \pi_0^C(\mS\sm S^\sigma)$
that is sometimes called the {\em (pre-)Euler class} of the sign representation;
we abuse notation and use the same symbol for the image
\[ a\in E_0^C(S^\sigma)\ = \ E^\sharp(C,-\sigma)  \]
under the unit morphism $\mS\to E$.

\begin{theorem}\label{thm:oriented_ring2oriented_el}
  For every orientable global ring spectrum $E$,
  the pair $(E^\sharp,a)$ is an orientable $\elRO$-algebra.
\end{theorem}
\begin{proof}
  We exploit that the sequence of equivariant homotopy groups 
  \[  E_1^C(S^\sigma)\ \xra{\res^C_e} \ E_1(S^1)\ \xra{\ \tr_1^C\ }\  E_1^C(S^1)  \]
  is exact.
  Because $E$ is orientable, we have $\tr^C_1(1\sm S^1)=\tr^C_1(1)\sm S^1=0$,
  and we can choose a class $t\in E_1^C(S^\sigma)$ that restricts to $1\sm S^1$.
  We use $t$ to define classes $t_V\in E_d^A(S^V)$ as in \eqref{eq:define_t_V},
  for all representations $V$ of elementary abelian 2-groups $A$, where $d=\dim(V)$.
  We claim that the following three properties hold for all $A$-representations $W$:
  \begin{enumerate}[(i)]
  \item For every subgroup $B$ of $A$, the restriction homomorphism
    \[ \res^A_B\ : \ E_k^A(S^W)\ \to \ E_k^B(S^{W|_B}) \]
    is surjective.
  \item 
    Suppose that $W=V\oplus\lambda$ for
    a nontrivial $A$-character $\lambda$ with kernel $K$.
    Then the following sequence is exact: 
    \[ 0 \ \to \ E_k^A(S^V)\ \xra{\ \cdot a_\lambda} \ E_k^A(S^{V\oplus\lambda})\
      \xra{\res^A_K} \ E_{k-1}^K(S^{V|_K})\ \to\ 0  \]
  \item The graded $\mF_2$-vector space $E_*^A(S^W)$ is generated
    as a graded $E_*^A$-module by the classes $a_U\cdot t_V$
    for all $A$-representations $U$ and $V$ such that $U\oplus V= W$.
  \end{enumerate}
  Part (ii) is the exactness property that shows that
  $(E^\sharp,a)$ is an orientable $\elRO$-algebra.
  
  The argument is similar as in the proof of Proposition \ref{prop:generation}:
  we prove all three statements together by induction over the rank of $A$.
  The induction starts when $A$ is the trivial group, in which case there is nothing to show.
  Now we let $A$ be a nontrivial elementary abelian 2-group, and we assume
  that parts (i)--(iii) hold for all proper subgroups of $A$.

  We start by proving (i), where we may assume that $B$ is a proper subgroup of $A$.
  By part (iii) for $B$, it suffices to show that all classes of the form
  $x\cdot a_U\cdot t_V$ are in the image of the restriction homomorphism,
  whenever $U$ and $V$ are $B$-representations such that $U\oplus V=W|_B$,
  and $x$ is a homogeneous element of $E_*^B$.
  Because $W$ is a sum of 1-dimensional $A$-representations,
  we may choose an $A$-equivariant decomposition $W=\bar U\oplus\bar V$
  such that  $\res^A_B(\bar U)\iso U$ and $\res^A_B(\bar V)\iso V$.
  We also choose a homomorphism $r:A\to B$ that is a retraction to the inclusion. Then
  \[ \res^A_B( r^*(x)\cdot a_{\bar U}\cdot t_{\bar V})\ =\ x\cdot a_U\cdot t_V\ , \]
  and we have shown (i) for $A$.
  To prove (ii) we smash the cofiber sequence of based $A$-spaces
  \[ A/K_+ \ \xra{\quad} \ S^0 \ \xra{\quad }\ S^\lambda \ \to\ A/K_+\sm S^1 \]
  with $S^V$ and apply $A$-equivariant $E$-homology, yielding a long exact sequence: 
  \begin{align*}
     \dots\  \to \ E_k^A(S^V) \
    \xra{\ \cdot a_\lambda}\ E_k^A(S^{V\oplus\lambda})\
     \xra{\ \partial\ }\ E_{k-1}^A(S^V\sm A/K_+)\ \to\  \dots  
  \end{align*}
  The Wirthm{\"u}ller isomorphism identifies the group
  $E_{k-1}^A(S^V\sm A/K_+)$ with $E_{k-1}^K(S^{V|_K})$.
  Under this identification, the boundary map $\partial$ becomes the
  restriction homomorphism $\res^A_K:E_k^A(S^{V\oplus\lambda})\to E_{k-1}^K(S^{V|_K})$,
  which is surjective by (i).
  So the long exact sequence decomposes into short exact sequences, showing (ii).

  We prove (iii) by induction on $\dim(W)-\dim(W^A)$.
  If $A$ acts trivially on $W$, the $E_*^A$-module $E_*^A(S^W)$
  is free of rank 1 with basis the class $t_W$, so the claim holds.
  If $W$ acts nontrivially on $W$, we write $W=V\oplus\lambda$
  for an $A$-representation $V$ and a nontrivial $A$-character $\lambda$, with kernel $K$.
  By part (i), the restriction map  $\res^A_K:E_{m-1}^A(S^V)\to E_{m-1}^K(S^{V|_K})$
  is surjective. So given $x\in E_k^A(S^W)$,
  there is a class $z\in E_{m-1}^A(S^V)$ such that $\res^A_K(z)=\res^A_K(x)$. Then
  \[ \res^A_K(x + z\cdot t_\lambda)\ = \ \res^A_K(x) + \res^A_K(z)\ = \ 0 \ .\]
  Part (ii) provides a class $y\in E_k^A(S^V)$
  such that  $y\cdot a_\lambda= x + z\cdot t_\lambda$.
  By the inductive hypothesis, the classes $y$ and $z$ are $E_*^A$-linear combinations of
  products of $a$-classes and $t$-classes.
  Hence the same is true for $x=y\cdot a_\lambda+ z\cdot t_\lambda$,
  and we have shown (iii). 
\end{proof}

\section{The universal property of Bredon homology} \label{sec:bredon}

In this section we discuss {\em Bredon $\elRO$-algebras},
i.e., orientable $\elRO$-algebras whose $\mZ$-graded parts consist only of a copy
of $\mF_2$ in dimension zero, see Definition \ref{def:Bredon}.
As the name suggests, Bredon homology of representation spheres with constant mod 2 coefficients
is an example of a Bredon $\elRO$-algebra.
Our main result, Theorem \ref{thm:Bredon initial}, exhibits a universal property:
every Bredon $\elRO$-algebra is initial among {\em additive} oriented $\elRO$-algebras.
It follows that Bredon $\elRO$-algebras are unique up to unique isomorphism,
a fact that is not entirely obvious, at least to the authors, from the definition.

\begin{defn}\label{def:Bredon}
  A {\em Bredon $\elRO$-algebra} is an orientable $\elRO$-algebra $H$
  such that for every elementary abelian 2-group $A$ and all $k\in\mZ$ we have
  \[ H(A)_k\ = \ H(A,k)\ = \
    \begin{cases}
      \mF_2 & \text{ for $k=0$, and}\\
      \ 0 & \text{ for $k\ne0$.}
    \end{cases}
  \]
\end{defn}

As the following example shows, Bredon $\elRO$-algebras exist.
Theorem \ref{thm:Bredon initial} below proves that
Bredon $\elRO$-algebras are unique up to unique isomorphism.

For every $\elRO$-algebra $X$, any two choices of pre-Euler classes
differ by multiplication with a unit in $X(C,0)$.
Furthermore, as we discussed in Remark \ref{rk:choice of inverse Thom class},
the difference between two choices of inverse Thom classes is an element of the form $a\cdot y$,
for $y\in X(C,1)$.
If $H$ is a Bredon $\elRO$-algebra, then $H(C,0)^\times=\mF_2^\times=\{1\}$ and $H(C,1)=0$.
Therefore, the pre-Euler class and the inverse Thom class of a Bredon $\elRO$-algebra are unique.

\begin{eg}[The global Eilenberg-MacLane spectrum]\label{eg:HF}
  We let $H\underline\mF_2$ denote the classical commutative orthogonal ring spectrum
  model for the Eilenberg-MacLane spectrum for the field $\mF_2$,
  as defined for example in \cite[Construction 5.3.8]{schwede:global}.
  By \cite[Proposition 5.3.9]{schwede:global}, its restriction to finite groups is
  an Eilenberg-MacLane ring spectrum for the constant global functor
  with value $\mF_2$.
  So $H\underline\mF_2$ represents Bredon homology and cohomology with coefficients
  in the constant Mackey functor $\mF_2$.
  
  Because the transfer $\tr_1^C(1)$ is trivial in
  Bredon homology with constant mod 2 coefficients,
  $H\underline\mF_2$ is an orientable global ring spectrum.
  By Theorem \ref{thm:oriented_ring2oriented_el}, the pair
  $((H\underline\mF_2)^\sharp,a)$ is an orientable $\elRO$-algebra, and hence a
  Bredon $\elRO$-algebra in the sense of Definition \ref{def:Bredon}.
\end{eg}

\begin{defn}[Additive oriented $\elRO$-algebras]\label{def:additive}
  An oriented $\elRO$-algebra $(X,a,t)$ is {\em additive} if the class  
  \[ a_1 t_\mu t_2\ +\ t_1 a_\mu t_2\ +\ t_1 t_\mu a_2 \]
  in $X(C\times C,2-(p_1\oplus\mu\oplus p_2))$ is trivial. Here $p_1,p_2,\mu:C\times C\to C$
  are the three nontrivial characters of the group $C\times C$,
  and we abbreviate $a_{p_i}$ to $a_i$, and $t_{p_i}$ to $t_i$.
\end{defn}

The term `additive' in the previous definition is motivated by the fact
that the 2-torsion formal group law associated to an additive $\elRO$-algebra 
is the additive formal group law, see Example \ref{eg:additive is additive} below.

\begin{eg}[Bredon $\elRO$-algebras are additive]
  We let $(X,a,t)$ be an oriented $\elRO$-algebra.
  The class  $ a_1 t_\mu t_2 + t_1 a_\mu t_2 + t_1 t_\mu a_2$
  in $X(C\times C,2-(p_1\oplus\mu\oplus p_2))$
  is in the kernel of all three restriction maps from $C\times C$ to $C$.
  By the fundamental exactness property,
  this class is divisible by each of  $a_1$, $a_\mu$ and $a_2$.
  Proposition \ref{prop:coprime} (ii) then shows that the class
  is divisible by the product $a_1 a_\mu a_2$.
  So there is a class $z\in X(C\times C)_2$  such that
  \[  z\cdot a_1 a_\mu a_2 \ =  \ a_1 t_\mu t_2\ +\ t_1 a_\mu t_2\ +\ t_1 t_\mu a_2 \ .     \]
  If the group $X(C\times C)_2$ is trivial,
  then the class $z$ necessarily vanishes, and so $(X,a,t)$ is additive. 
  This is the case in particular for all Bredon $\elRO$-algebras.
\end{eg}

The following main result of this section is the fact that
every Bredon $\elRO$-algebra is initial among additive $\elRO$-algebras.
In the proof we will start using the notation
\[ A^\circ \ = \ \{ \lambda\colon A\to C\ |\ \text{$\lambda$ is surjective} \} \]
for the set of nontrivial characters of an elementary abelian 2-group $A$.

\begin{theorem}[Universal property of Bredon homology]\label{thm:Bredon initial}
  Let $H$ be a Bredon $\elRO$-algebra, and let $X$ be an orientable $\elRO$-algebra.
  \begin{enumerate}[\em (i)]
  \item  Evaluation at the unique inverse Thom class of $H$ 
    is a bijection between the set of morphisms of orientable $\elRO$-algebras
    from $H$ to $X$, and the set of those inverse Thom classes $t$ of $X$ such that
    $(X,a,t)$ is additive.
  \item Any two Bredon $\elRO$-algebras are uniquely isomorphic.
  \end{enumerate}
\end{theorem}
\begin{proof}
  (i) The unique orientation and inverse Thom class make
  $(H,a,t)$ additive, so every morphism of orientable $\elRO$-algebras sends
$t$ to an inverse Thom class of an additive orientation.
  
  Now we prove the uniqueness.
  Proposition \ref{prop:generation} (ii)
  shows that $H(A,\star)$ is generated as an $I_A$-graded $\mF_2$-algebra
  by the classes $a_\lambda$ and $t_\lambda$ for all nontrivial $A$-characters,
  where $t$ is the unique inverse Thom class;
  so every morphism $f:H\to X$ of oriented $\elRO$-algebras
  is determined by its effect on the pre-Euler class $a$ and the inverse Thom class $t$.

  Now we show the existence of the morphism
  for the specific example $H^\sharp=(H\underline\mF_2)^\sharp$,
  the Bredon $\elRO$-algebra given by the Bredon homology of representation spheres,
  see Example \ref{eg:HF}.
  We will exploit the presentation of the rings $H^\sharp(A,\star)$  obtained in
  \cite[Theorem 2.5]{hausmann-schwede:Bredon}.
  We let $(\bar a,\bar t)$ be an additive orientation of $X$.
  We let $T\subset A^\circ$ be any set of nontrivial $A$-characters whose product
  $\prod_{\lambda\in T}\lambda$ is the trivial character.
  We claim that the class
  \[   r(T)\ = \ 
    \sum_{\lambda\in T} \bar a_{\lambda}\cdot  \big( \prod_{\mu\in T\setminus\{\lambda\}} \bar t_{\mu}\big) \]
  in $X(A,|T|-1-\bigoplus_{\lambda\in T}\lambda)$ equals $0$.
  We prove this by induction over the cardinality of $T$.
  A non-empty subset of $A^\circ$ whose elements multiply to 1 must have at least 3 elements,
  so we start the induction for $T=\{\alpha,\beta,\gamma\}$ with $\alpha\cdot\beta\cdot\gamma=1$.
  Then
  \begin{align*}
    r(\{\alpha,\beta,\gamma\})\
    & = \ \bar a_\alpha \bar t_\beta \bar t_\gamma + \bar t_\alpha \bar a_\beta \bar t_\gamma + \bar t_\alpha \bar t_\beta \bar a_\gamma\\
    &= \ (\alpha,\gamma)^*(\bar a_1 \bar t_\mu \bar t_2 + \bar t_1 \bar a_\mu \bar t_2 + \bar t_1 \bar t_\mu \bar a_2)\ = \ 0
  \end{align*}
  by the additivity hypothesis, where $(\alpha,\gamma)\colon A\to C\times C$.
  
  Now we suppose that the set $T$ has at least 4 elements.
  We pick two distinct elements $\alpha\ne \beta$ from $T$,
  and we set $\gamma=\alpha\cdot\beta$.
  If $\gamma\in T$, then the sets $T\setminus\{\alpha,\beta,\gamma\}$
  and $\{\alpha,\beta,\gamma\}$ both have fewer elements than $T$,
  and for both the product of their elements is the trivial character.
  So $r(T\setminus\{\alpha,\beta,\gamma\})=0$ and $r(\{\alpha,\beta,\gamma\})=0$ by induction.
  Thus
  \[  r(T)\ = \
    \bar t_\alpha \bar t_\beta\bar t_\gamma\cdot r(T\setminus\{\alpha,\beta,\gamma\})\ + \
    r(\{\alpha,\beta,\gamma\})\cdot \! \prod_{\mu\in T\setminus\{\alpha,\beta,\gamma\}} \bar t_{\mu} \ = \ 0\ .\]
  If $\gamma\not\in T$, then the sets $(T\setminus\{\alpha,\beta\})\cup\{\gamma\}$
  and $\{\alpha,\beta,\gamma\}$ both have fewer elements than $T$,
  and for both the product of their elements is the trivial character.
  So $r((T\setminus\{\alpha,\beta\})\cup\{\gamma\})=0$
  and $r(\{\alpha,\beta,\gamma\})=0$ by induction.
  Thus
  \[ \bar a_{\gamma}\cdot r(T)\ = \
    (\bar a_\alpha \bar t_\beta+\bar t_\alpha \bar a_\beta)\cdot r((T\setminus\{\alpha,\beta\})\cup\{\gamma\})\ + \
    r(\{\alpha,\beta,\gamma\})\cdot r(T\setminus\{\alpha,\beta\})\ = \ 0\ . \]
  Since multiplication by $\bar a_\gamma$ is injective,
  we conclude that $r(T)=0$. This completes the inductive step, and hence the proof of the claim.

  By \cite[Theorem 2.5]{hausmann-schwede:Bredon},
  the $I_A$-graded ring $H^\sharp(A,\star)$
  is generated by the classes $a_\lambda$ and $t_\lambda$
  for all nontrivial $A$-characters $\lambda$,
  and the ideal of relations between these classes is generated by the polynomials $r(T)$
  as $T$ runs through all minimally dependent subsets of $A^\circ$.
  Minimally dependent sets of nontrivial characters
  in particular have the property that their elements multiply to the trivial character.
  By the previous claim, all generating relations map to 0 in $X(A,\star)$.
  So there is a unique morphism of $I_A$-graded $\mF_2$-algebras
  \[ f(A,\star)\ : \ H^\sharp(A,\star)\ \to \ X(A,\star)\]
  such that $f(A,-\lambda)(a_\lambda)=\bar a_\lambda$ and $f(A,1-\lambda)(t_\lambda)=\bar t_\lambda$
  for all $\lambda\in A^\circ$.
  For varying $A$, these homomorphisms are compatible with restriction
  along group homomorphisms, simply because this holds for $a$- and $t$-classes
  and these generate $H^\sharp(A,\star)$ as an $\mF_2$-algebra.
  We have thus constructed the desired morphism $f:H^\sharp\to X$.
  
  (ii) We let $H$ be any Bredon $\elRO$-algebra. Then $H$ has a unique orientation,
  which is moreover additive. So part (i) provides a unique morphism
  of orientable $\elRO$-algebras $f\colon H^\sharp\to H$ that matches the orientations.
  This morphism is necessarily an isomorphism in integer degrees,
  and hence an isomorphism of orientable $\elRO$-algebras,
 see Proposition \ref{prop:Z-graded iso suffices}.
\end{proof}

\section{The associated 2-torsion formal group law}

As we shall now explain, the underlying $\mF_2$-algebra of every oriented $\elRO$-algebra
 $(X,a,t)$
naturally comes with a 2-torsion formal group law, obtained roughly speaking by
expanding the Euler class $e_\mu$ of the multiplication $\mu:C\times C\to C$ in terms of
the Euler classes of the two projections.
As we will see later, a lot of the structure of $(X,a,t)$ can be recovered
from this underlying 2-torsion formal group law
and the geometric fixed point rings of $X$.

\begin{construction}[Localization away from $t$]\label{con:t-localization}
  We define the localization away from $t$ of an oriented $\elRO$-algebra $(X,a,t)$
  by inverting all inverse Thom classes.
  For every elementary abelian 2-group $A$, we let $(t^{-1}X)(A)$
  be the integer graded part of the localization
  of the $I_A$-graded ring $X(A,\star)$ obtained by inverting all classes
  of the form $t_\lambda$ for all nontrivial $A$-characters $\lambda$.
  Homogeneous elements of degree $k$ of $(t^{-1} X)(A)$ are thus fractions
  of the form $x/t_V$ for some $A$-representation $V$, and some
  $x\in X(A,k+|V|-V)$.
  These fractions then satisfy
  \[ x/t_V \ = \  (x\cdot t_W)/t_{V\oplus W} \text{\qquad and\qquad}
    x/t_V \cdot y/t_W \ = \ (x\cdot y)/t_{V\oplus W}\ . \]
  The localizations away from $t$ come with restriction homomorphisms 
  $\alpha^*:(t^{-1}X)(A)\to (t^{-1}X)(B)$ that are contravariantly functorial
  for group homomorphisms $\alpha:B\to A$, and defined by the formula
  \[  \alpha^*(x/t_V) \ = \ \alpha^*(x)/t_{\alpha^*(V)}\ .     \]
\end{construction}

\begin{defn}[Euler classes]
  Let $(X,a,t)$ be an oriented $\elRO$-algebra. The {\em Euler class} of 
  a nontrivial character $\lambda$ of an elementary abelian 2-group $A$
  is the class
    \[ e_\lambda\ =\ a_\lambda/t_\lambda\ = \ \lambda^*(a/t) \ \in \ (t^{-1}X)(A)_{-1} \]
    of degree $-1$ in the localization away from $t$ of $X$.
\end{defn}

For $i=1,\dots, n$, we let $p_i:C^n\to C$ be the projection to the $i$-th factor.
We write
\[ e_i\ = \ e_{p_i}\ = \ p_i^*(a/t)\ \in \ (t^{-1}X)(C^n)_{-1} \]
for the Euler class of $p_i$.

\begin{prop} 
  Let $(X,a,t)$ be an oriented $\elRO$-algebra and $n\geq 0$.
  The Euler classes $e_1,\dots,e_n$ form a regular sequence
  in the graded ring $(t^{-1}X)(C^n)$
  that generates the kernel of the restriction homomorphism
  \[ \res^{C^n}_1\ : \ (t^{-1}X)(C^n)\ \to \ X(1)\ . \]
\end{prop}
\begin{proof}
  We argue by induction over $n$, beginning with $n=0$ where there is nothing to show.
  Now we assume $n\geq 1$. Because localization is exact, the sequence
  \[ 0 \ \to \ (t^{-1} X)(C^n)_{*+1}\  \xra{e_1\cdot -} \
    (t^{-1} X)(C^n)_*\  \xra{\ j^*\ }\   (t^{-1}X)(C^{n-1})_*\ \to \ 0  \]
  is exact, where $j=(1,-):C^{n-1}\to C^n$ is the embedding as the last $n-1$ factors.
  So $e_1$ is a non zero-divisor, and the restriction homomorphism $j^*$
  identifies $(t^{-1} X)(C^n)/(e_1)$ with $(t^{-1}X)(C^{n-1})$.
  By induction hypothesis, the images of the classes $e_1,\dots,e_{n-1}$ are a regular
  sequence in  $(t^{-1}X)(C^{n-1})$ that generates the augmentation ideal.
  Since $j^*(e_i)=e_{i-1}$ for $2\leq i\leq n$, this completes the inductive step.
\end{proof}

A key tool in our arguments will involve expansions of elements of certain rings as power series
in a distinguished regular element.
We will now set up some theory around this in a systematic fashion.
In the following we will use the term {\em graded ring} as short hand for
`commutative $\mZ$-graded ring'. The next proposition is well-known;
since it is central for various of our arguments, we take the time to spell it out
as a point of reference.

\begin{prop}\label{prop:power expansion}
  Let $(e_1,\dots,e_n)$ be a regular sequence of homogeneous elements of a graded ring $R$.
  Let $s:S\to R$ be a graded ring homomorphism whose composite with the projection $R\to R/(e_1,\dots,e_n)$
  is an isomorphism.
  \begin{enumerate}[\em (i)]
  \item The completion of $R$ at the ideal $(e_1,\dots,e_n)$ is a power series algebra over $S$
    in the elements $e_1,\dots,e_n$.
  \item  Let $T$ be a graded ring that is complete with respect to a homogeneous ideal $J$.
  Then the map
  \begin{align*}
    \{\rho\in\text{\em grRing}(R,T)\colon \rho(e_1,\dots,e_n)\subset J \}\ &\to \ \text{\em grRing}(S,T)\times J^n \\
    \rho \qquad &\longmapsto \ (\rho s,\rho(e_1),\dots,\rho(e_n))
  \end{align*}
  is bijective.
  \end{enumerate}
\end{prop}
\begin{proof}
  We write $I=(e_1,\dots,e_n)$ for the ideal generated by the regular sequence.
  It is a classical fact that because of the regularity, the unique $R/I$-algebra homomorphism
  \[ (R/I)[x_1,\dots,x_n]\ \to \ \bigoplus_{k\geq 0} I^k/I^{k+1} \]
  that sends $x_i$ to $e_i+I^2$ is an isomorphism,
  see for example \cite[\S 9.7, Th{\'e}or{\`e}me 1]{bourbaki:algebre}.
  In particular, $I^k/I^{k+1}$ is a free  $R/I$-module
  on the residue classes of the monomials of exact degree $k$ in $e_1,\dots,e_n$.
  Thus by induction on $k$, the underlying $S$-module of $R/I^{k+1}$
  is free on the residue classes of the monomials of degree at most $k$.
  So $R/I^{k+1}$ is a truncated polynomial algebra over $S$, and the completion of $R$ at $I$
  is a power series algebra over $S$ on the classes $e_1,\dots,e_n$, giving (i).
  This power series algebra is free when it comes to morphisms with complete targets, yielding (ii).
\end{proof}

\begin{defn}\label{def:series_expansion}
  Let $(e_1,\dots,e_n)$ be a regular sequence of homogeneous elements of a commutative graded ring $R$,
  and let $s:S\to R$ be a ring homomorphism whose composite with the projection $R\to R/(e_1,\dots,e_n)$
  is an isomorphism.
  The {\em power series expansion} of $R$ at $e_1,\dots,e_n$
  is the unique morphism $d:R\to S\dbr{e_1,\dots,e_n}$ of graded $S$-algebras that sends $e_i$ to $e_i$
  for all $1\leq i\leq n$.
\end{defn}

\begin{construction}\label{con:d^n}
  We let $(X,a,t)$ be an oriented $\elRO$-algebra and $n\geq 0$.
  The inflation homomorphism $p^*:X(1)\to (t^{-1}X)(C^n)$
  is a section to the restriction homomorphism $\res^{C^n}_1$.
  Since the kernel of $\res^{C^n}_1$ is generated by the regular sequence of Euler classes
  $e_1,\dots,e_n$, Proposition \ref{prop:power expansion} provides a
 unique graded $X(1)$-algebra homomorphism
  \[ d^n\ : \ (t^{-1}X)(C^n)\ \to\ X(1)\dbr{e_1,\dots,e_n} \]
  that sends $e_i$ to $e_i$ for all $1\leq i\leq n$, the {\em power series expansion}
  in the Euler classes $e_1,\dots,e_n$.
\end{construction}

We recall that a {\em 2-torsion formal group law} over a graded commutative ring $R$
is a power series $F(x,y)\in R\dbr{x,y}$ that is homogeneous of degree $-1$ with respect
to the grading $\deg(x)=\deg(y)=-1$,
and that satisfies
\[ F(x,0)=x,\quad F(x,y)=F(y,x),\quad F(x,F(y,z))=F(F(x,y),z)\text{\quad and\quad} F(x,x)=0\ . \]
In the next theorem and in what follows, we shall write $\mu:C\times C\to C$
for the multiplication homomorphism.

\begin{theorem}\label{thm:fgl}
  Let $(X,a,t)$ be an oriented $\elRO$-algebra. 
  The image of the Euler class $e_\mu$ under the power series expansion
  \[ d^2\ : \ (t^{-1}X)(C\times C)\ \to \ X(1)\dbr{e_1,e_2} \]
  is a 2-torsion formal group law over the graded ring $X(1)$,
  the {\em formal group law} of $(X,a,t)$.
\end{theorem}
\begin{proof}
  We let $\alpha:C^n\to C^2$ be a group homomorphism, for some $n\geq 1$.
  We claim that the following diagram of graded $X(1)$-algebras commutes:
  \[ \xymatrix@C=12mm{
   (t^{-1}X)(C^2)\ar[d]_{d^2} \ar[r]^-{\alpha^*}&
      (t^{-1}X)(C^n)\ar[d]^{d^n}   \\
   X(1)\dbr{e_1,e_2}   \ar[r]_-{\alpha^\natural} & X(1)\dbr{e_1,\dots,e_n}  } \]
  Here $\alpha^\natural$ is the unique graded $X(1)$-algebra homomorphism
  that satisfies $\alpha^\natural(e_i)=d^n(e_{p_i\alpha})$ for $i=1,2$.
  Indeed, both composites have the same effect on the regular sequence $(e_1,e_2)$ that generates
  the augmentation ideal of $(t^{-1}X)(C\times C)$.
  So they coincide by Proposition \ref{prop:power expansion} (ii).
  The commutativity of the diagram then yields
  \[ \alpha^\natural(F(e_1,e_2))\ = \ \alpha^\natural(d^2(e_\mu)) \ = \ d^n(\alpha^*(e_\mu))\ .\]
  We derive the four relations of a 2-torsion formal group law by specializing
  for five different choices of $\alpha$.
  When $\alpha$ is the inclusion of the first factor $i_1:C\to C\times C$, 
  then $i_1^\natural:  X(1)\dbr{e_1,e_2}\to X(1)\dbr{e_1}$ sets $e_2=0$,
  and $i_1^*(e_\mu)=e_1$, so we get
  \[ F(e_1,0)\ = \  i_1^\natural(F(e_1,e_2))\ = \ \ d^1(e_1)\ = \ e_1\ .\]
  When $\alpha$ is the diagonal $\Delta:C\to C\times C$,
  then $\delta^\natural:  X(1)\dbr{e_1,e_2}\to X(1)\dbr{e_1}$ sends both $e_1$ and $e_2$ to $e_1$,
  and $\delta^*(e_\mu)=0$, so we obtain
  \[ F(e_1,e_1)\ = \ \Delta^\natural(F(e_1,e_2))\ = \ \ d^1(0)\ = \ 0 \ .\]
  When $\alpha$ is the involution $\tau:C\times C\to C\times C$
  that interchanges the two factors,
  then $\tau^\natural:  X(1)\dbr{e_1,e_2}\to X(1)\dbr{e_1,e_2}$ interchanges $e_1$ and $e_2$,
  and $\tau(e_\mu)=e_\mu$; so we get the commutativity relation
  \[ F(e_2,e_1)\ = \ \tau^\natural(F(e_1,e_2))\ = \ \ d^2(e_\mu)\ = \ F(e_1,e_2) \ .\]
  For $\alpha=\mu\times C:C^3\to C^2$, we have
  $p_1(\mu\times C)=\mu\circ p_{1,2}$, for $p_{1,2}:C^3\to C^2$ the projection to
  the first two coordinates,
  and $p_2(\mu\times C)=p_3$. So the homomorphism
  $(\mu\times C)^\natural:X(1)\dbr{e_1,e_2}\to X(1)\dbr{e_1,e_2,e_3}$
  sends $e_1$ to $d^3(p_{1,2}^*(e_\mu))=d^2(e_\mu)=F(e_1,e_2)$, and it sends $e_2$ to $e_3$.
  So we obtain
  \[ F(F(e_1,e_2),e_3)\ = \ (\mu\times C)^\natural(F(e_1,e_2))\ = \ \ d^3((\mu\times C)^*(e_{\mu})) \ .\]
  The analogous argument for $\alpha=C\times \mu$ yields
  \[ F(e_1,F(e_2,e_3))\ = \  d^3((C\times\mu)^*(e_{\mu})) \ .\]
  Because $(\mu\times C)^*(e_\mu)=(C\times\mu)^*(e_\mu)$
  (the multiplication of $C$ is associative), 
  this establishes the associativity relation $ F(F(e_1,e_2),e_3)= F(e_1,F(e_2,e_3))$.
\end{proof}

\begin{eg}[Additive orientations yield additive formal group laws]\label{eg:additive is additive}
  Let $(X,a,t)$ be an additive $\elRO$-algebra in the sense of Definition \ref{def:additive},
  i.e.,
  \[  a_1 t_\mu t_2\ +\ t_1 a_\mu t_2\ +\ t_1 t_\mu a_2 \ = \ 0 \]
  in $X(C\times C,2-(p_1\oplus\mu\oplus p_2))$. Then
  \[   e_1 + e_\mu + e_2    \ = \
 (a_1 t_\mu t_2 + t_1 a_\mu t_2 + t_1 t_\mu a_2)/t_1 t_\mu t_2  \ = \ 0   \ . \]
  So 
  \[ F(e_1,e_2)\ = \ d^2(e_\mu)\ = \ d^2(e_1+e_2)\ = \ e_1 + e_2\ ,  \]
  i.e., the formal group law is the additive formal group law.
\end{eg}

\section{Geometric fixed points}

To establish the universal property of equivariant bordism,
we shall make use of the {\em geometric fixed points}
of an orientable $\elRO$-algebra, obtained by inverting all the $a$-classes
of all nontrivial characters.
For orientable $\elRO$-algebras that arise from global ring spectra,
this construction generalizes the geometric fixed point construction
for equivariant spectra, see Remark \ref{rk:reconcile} below, thence the name.

We also discuss various ways in which the power series expansions of the previous section
can be extended to Laurent series expansions on geometric fixed points.
The main result is the integrality criterion
of Theorem \ref{thm:detection} that allows to detect when a class
in the geometric fixed point ring $\Phi^A_*X $ arises from a class in $X(A)_*$.
This criterion will be used in Theorem \ref{thm:fix_criterion}
to show how morphisms of oriented $\elRO$-algebras
can be recovered from the induced morphisms of geometric fixed point rings.

\begin{construction}[Geometric fixed points]\label{con:geometric fix}
  We let $(X,a)$ be an orientable $\elRO$-algebra.
  For every elementary abelian 2-group $A$,
  we define the ring of {\em $A$-geometric fixed points} $\Phi^A_* X$
  as the integer graded part of the localization of the $I_A$-graded ring $X(A,\star)$
  obtained by inverting all classes
  of the form $a_\lambda$ for all nontrivial $A$-characters $\lambda$.
  Homogeneous elements in $\Phi^A_k X$ are thus fractions of the form $x/a_V$
  for some $A$-representation $V$ with trivial fixed points and some $x\in X(A,k-V)$.
  These elements satisfy the relations
  \[ x/a_V \ = \  (x\cdot a_W)/a_{V\oplus W}  \text{\qquad and\qquad}
    x/a_V \cdot y/a_W \ = \ (x\cdot y)/a_{V\oplus W} \]
  for all $A$-representations $W$ with $W^A=0$.
  In an orientable $\elRO$-algebra, 
  multiplication by the classes $a_\lambda$ is injective;
  hence also the maps $-/a_V\colon X(A,k-V)\to \Phi^A_k X$ are injective.

  The geometric fixed point rings enjoy {\em inflation homomorphisms},
  i.e., they are contravariantly functorial for epimorphisms $\beta:B\to A$.
  Indeed, $\beta^*(a_\lambda)=a_{\lambda\beta}$,
  so the inflation homomorphism $\beta^*:X(A,\star)\to X(B,\star)$
  induces a ring homomorphism $\beta^*:\Phi^A_* X\to\Phi^B_* X$ of localizations.
  In terms of the above presentation as fractions, we have
  \begin{equation}\label{eq:inflation_geometric}
    \beta^*(x/a_V) \ = \ \beta^*(x)/a_{\beta^*(V)}\ .    
  \end{equation}
\end{construction}

We emphasize a crucial difference between $\Phi^A_* X$,
the localization away from $a$, and the localization $(t^{-1}X)(A)$ away from $t$
discussed in Construction \ref{con:t-localization}:
the latter admit restriction homomorphisms for {\em all} group homomorphisms,
the former only admit inflations, i.e., restriction along epimorphisms.

\begin{rk}[Relation to geometric fixed points of equivariant spectra]\label{rk:reconcile}
  As already mentioned above, the term `geometric fixed points'
  for the localization of an orientable $\elRO$-algebra
  away from the $a$-classes is motivated by the topological examples.
  Indeed, for every compact Lie group $G$ and every genuine $G$-spectrum $R$,
  a common definition of the geometric fixed point homotopy groups
  is as the reduced $G$-equivariant $R$-homology
  of a specific based $G$-space $\tilde E\mathcal P$, see for example
  \cite[(3.18)]{greenlees-may:eqstho}.
  Here $\tilde E\mathcal P$ is the unreduced suspension of a classifying $G$-space
  for the family of proper closed subgroups of $G$.
  A natural homomorphism $R_*^G\to R_*^G(\tilde E\mathcal P)$
  from the equivariant homotopy groups to the geometric fixed point homotopy groups
  is defined as the effect on $R_*^G(-)$
  of the inclusion of cone points $S^0\to \tilde E\mathcal P$.
  If $R$ is a ring spectrum in the
  homotopy category of genuine $G$-spectra, this map is a morphism of graded rings,
  and it factors through an isomorphism from the localization
  of the $R O(G)$-graded homotopy ring of $R$ at the pre-Euler classes of all real
  $G$-representations $V$ with $V^G=0$,
  see \cite[Proposition 3.20]{greenlees-may:eqstho}.

  If $R$ is a global ring spectrum and $\lambda$ a nontrivial character
  of an elementary abelian 2-group $A$,
  then $a_\lambda\in R^\sharp(A,-\lambda)=R_0^A(S^\lambda)$
  is precisely the pre-Euler class of the $A$-representation $\lambda$,
  so the geometric fixed point ring $\Phi^A_* R^\sharp$,
  in the sense of Construction \ref{con:geometric fix},
  of the orientable $\elRO$-algebra $(R^\sharp,a)$
  is the localization of the $R O(A)$-graded homotopy ring of $R$ at the pre-Euler classes.
  Hence the two geometric fixed point rings $\Phi_*^A R^\sharp$ and $R_*^A(\tilde E\mathcal P)$
  of an orientable global ring spectrum are naturally isomorphic.
\end{rk}

For a $\mZ$-graded ring $R$, we now start using the notation
\[ R\lr{\theta}\ = \ R\dbr{\theta}[\theta^{-1}] \]
for the $\mZ$-graded Laurent power series in a variable $\theta$ of degree $-1$,
the localization of the power series ring away from the element $\theta$.
So elements of $R\lr{\theta}$ of degree $k$ are power series $\sum_{i\in\mZ} x_i \theta^i$
such that $x_i\in R_{k+i}$, and $x_i=0$ for all almost all negative values of $i$. \smallskip

Let $(X,a,t)$ be an oriented $\elRO$-algebra, and let $K$ be an elementary abelian 2-group.
We will now define a Laurent power series expansion
on the geometric fixed point ring 
\[ d_K\ :\  \Phi^{K\times C}_* X\ \to \ (\Phi^K_* X)\lr{\theta} \ .\]
These homomorphisms will be instrumental in detecting
{\em effective} classes in the geometric fixed point rings, see Theorem \ref{thm:detection},
and in reconstructing morphisms of oriented $\elRO$-algebra
from their effect on geometric fixed point rings, see Theorem \ref{thm:fix_criterion}.

\begin{construction}\label{con:d_K}
  We let $(X,a,t)$ be an oriented $\elRO$-algebra, and we let $K$
  be an elementary abelian 2-group.
 The construction of the Laurent power series expansion requires certain intermediate
  constructions that we summarize in the following commutative diagram:
\begin{equation}\begin{aligned}\label{eq:d_K diagram}
  \xymatrix@C=12mm{
    &      \Psi^K X\ar[d]^{\delta_K}\ar[r] \ar@/_1pc/[dl]_-{i_1^*} &
    (\Psi^K X)[e_2^{-1}]\ar[d]_{\bar\delta_K} &\Phi^{K\times C}X \ar[l]\ar@/^1pc/[dl]^-{d_K} \\
\Phi^K X &(\Phi^K X)\dbr{\theta}\ar[r]\ar[l]^{\theta=0}&   (\Phi^KX)\lr{\theta} 
  }     
\end{aligned}\end{equation}
  We introduce the graded ring
  \begin{equation}\label{eq:define_Psi}
    \Psi^K X \ = \ X(K\times C,\star)[t_2^{-1},a_V^{-1}\colon V^K=0] \ ,    
  \end{equation}
  the integer graded part of the localization of the $I_{K\times C}$-graded ring $X(K\times C,\star)$
  at all $a$-classes of $(K\times C)$-representations with trivial $K$-fixed points,
  and the class $t_2=p_2^*(t)$ of the projection to the second factor.
  Inverting the classes $a_V$ for all $V^K=0$ is the same as inverting the classes
  $a_\nu$ for all non-zero characters of $K\times C$ {\em except}
  the projection to the second factor.
  For $\nu\ne p_2$, the restriction of $\nu$ to $K$ is nontrivial.
  The restriction homomorphism 
  \[ i_1^*\ : \  X(K\times C,\star) \ \to \ X(K,\star) \]
  along the embedding $i_1:K\to K\times C$ as the first factor
  sends the class $t_2$ to 1, and it sends the class $a_V$
  to the pre-Euler class of the restriction of $V$ to $K$.
  So the restriction extends uniquely to localizations to a morphism of graded rings
  \[ i_1^* \ : \ \Psi^K X \ \to \ \Phi^K X \]
  that we denote by the same symbol. The inflation homomorphism
  \[ p_1^*\ : \  X(K,\star) \ \to \ X(K\times C,\star) \]
  along the projection $p_1:K\times C\to K$ extends uniquely to a homomorphism on the localization
  \[ p_1^*\ : \ \Phi^K X \ \to \ \Psi^K X \]
  that splits $i_1^*$.
  We write  $a_2=p_2^*(a)$ for the pre-Euler class of the projection to the second factor, and we set
  \[  e_2\ =\ a_2/t_2\ \in \ \Psi^K_{-1} X \ .\]
 Because localization is exact, the sequence
  \[ 0\ \to\ \Psi^K_{*+1} X\ \xra{e_2\cdot -}\ \Psi^K_* X\ \xra{\ i_1^*\ }\ \Phi^K_* X\ \to\  0\]
  is exact. Hence $e_2$ is regular in $\Psi^K X$, and the composite of
  $p_1^*:\Phi^K_* X\to\Psi^K_* X$ and the projection $\Psi^K_* X\to (\Psi^K_* X)/(e_2)$
  is an isomorphism. We let
  \[  \delta_K \ : \ \Psi^K X\ \to \ (\Phi^K X)\dbr{\theta} \]
  denote the associated power series expansion
  in the sense of Definition \ref{def:series_expansion}.
  So $\delta_K$ is the unique morphism of graded $\Phi^K X$-algebras
  such that $\delta_K(e_2)=\theta$.
  The homomorphism extends uniquely to the localizations
  to a homomorphism
  \[ \bar\delta_K\ : \ (\Psi^K X)[e_2^{-1}] \ \to\  (\Phi^K X)\lr{\theta}\ .   \]
  The ring $(\Psi^K X)[e_2^{-1}]$ is the localization of $X(K\times C,\star)$ obtained by inverting
  all $a$-classes and the class $t_2$.
  Since $\Phi^{K\times C}X$ is a less drastic localization -- formed by inverting  all $a$-classes --
  it affords a localization morphism
  $\Phi^{K\times C}X\to (\Psi^K X)[e_2^{-1}]$ that is localization away from $t_2/a_2$.
  The composite with $\bar\delta_K$ is the homomorphism
  \[ d_K\ : \ \Phi^{K\times C}X\ \to\  (\Phi^K X)\lr{\theta}\ ;  \]
  it is then given by
  \[ d_K(x/(a_V a_2^n))\ = \ \delta_K(x/(a_V t_2^n)) \cdot \theta^{-n}\ , \]
  for $x\in X(K\times C,k-(V\oplus n p_2))$ with $V^K=0$.
  We have now introduced all ring homomorphisms in the diagram \eqref{eq:d_K diagram}.
\end{construction}

We make the following definition:

\begin{defn} \label{def:effective}
  Let $X$ be an orientable $\elRO$-algebra and $A$ an elementary abelian 2-group.
  A class in $\Phi_k^A X$ is {\em effective} if it lies in the image of the homomorphism
  $-/1:X(A)_k\to \Phi_k^A X$.
\end{defn}

For every orientable $\elRO$-algebra, the map $-/1:X(A)_k\to \Phi_k^A X$ is injective.
So if a class is effective, then its preimage in the integer graded subring is unique.
Said differently, the maps $-/1:X(A)_\ast\to \Phi_*^A X$ form an isomorphism
from the $\mZ$-graded part of $X(A,\star)$ to the subring of effective classes.

The next theorem establishes an effectivity criterion for geometric fixed point classes
in terms of the homomorphisms $d_K:\Phi^{K\times C}X\to (\Phi^K X)\lr{\theta}$.
Effective classes in $\Phi^{K\times C}X$ in particular also lie in the
  ring $\Psi^K X$, and hence their Laurent power series expansion is 
  contained in the subring $(\Phi^K X)\dbr{\theta}$,
  compare the diagram \eqref{eq:d_K diagram}.
  Moreover, $d_K$ is designed so that for $u\in X(K\times C)_*$, the constant coefficient
  is given by
  \[ d_K(u/1)(0)\ = \ i_1^*(u/1) \ = \ i_1^*(u)/1 \ .\]

\begin{theorem} \label{thm:detection} 
  Let $(X,a,t)$ be an oriented $\elRO$-algebra,
  and let $A$ be an elementary abelian 2-group.
  For every $y\in \Phi^A X$, the following two conditions are equivalent.
  \begin{enumerate}[\em (a)]
  \item 
    The class $y$ is effective.
  \item
    For every isomorphism $\alpha:K\times C\iso A$,
    the class $d_K(\alpha^*(y))$ in $(\Phi^K X)\lr{\theta}$ is integral,
    i.e., contained in the subring  $(\Phi^K X)\dbr{\theta}$.
  \end{enumerate}
\end{theorem}
\begin{proof}
  (a)$\Longrightarrow$(b) For every $u\in X(A)_k$, the class
  $\alpha^*(u/1)=\alpha^*(u)/1$ lies in  $\Psi^K X$, so $d_K(\alpha^*(u/1))$
  lies in $(\Phi^K X)\dbr{\theta}$,
  compare the commutative diagram \eqref{eq:d_K diagram}.

  (b)$\Longrightarrow$(a)
  We write $y=x/a_V$ for some $A$-representation $V$ with $V^A=0$, some $k\in\mZ$,
  and some $x\in X(A,k-V)$.
  We let $\lambda$ be a nontrivial $A$-character, with kernel $K$.
  We choose an isomorphism $\alpha:K\times C\iso A$ such that $\lambda\circ\alpha=p_2$.
  We decompose $\alpha^*(V)=U\oplus n p_2$ for some $(K\times C)$-representation
  $U$ with $U^K=0$, and some $n\geq 0$. Then
  \[ d_K(\alpha^*(y))\ = \ d_K(\alpha^*(x)/a_{\alpha^*(V)})\ = \
     d_K(\alpha^*(x)/a_U a_2^n)\ = \
    \delta_K(\alpha^*(x)/a_U t_2^n)\cdot \theta^{-n}\ .\]
  This class is integral by hypothesis,
  and $\delta_K(\alpha^*(x)/a_U t_2^n)$ is integral by definition.
  So we conclude that
  \[ \delta_K(\alpha^*(x)/a_U t_2^n)\ \in \ \theta^n\cdot (\Phi^K X)\dbr{\theta}\ .\]
  In other words, the first $n$ coefficients of the power series
  $\delta_K( \alpha^*(x)/a_U t_2^n)$ vanish.
  Hence the class $\alpha^*(x)/a_U t_2^n$ is divisible by $e_2^n$ in the ring $\Psi^K X$.
    
  Since all the $a$-classes are non zero-divisors, we deduce that for some $m\geq 0$
  the element $\alpha^*(x)\cdot t_2^m$ is divisible by $a_2^n$ in $X(K\times C,\star)$.
  Thus the element $x\cdot(\alpha^{-1})^*(t_2^m)=x\cdot t_\lambda^m$
  is divisible by $(\alpha^{-1})^*(a_2^n)=a_\lambda^n$ in $X(A,\star)$.
  Because $\res^A_K(t_\lambda^m)=1$,
  Proposition \ref{prop:coprime} (i) shows that 
  already $x$ is divisible by $a_\lambda^n$.
  Since $n$ is the multiplicity of $p_2$ in $\alpha^*(V)$,
  it is also the multiplicity of $\lambda=(\alpha^{-1})^*(p_2)$ in $V$.
  We have thus shown that for every nontrivial $A$-character $\lambda$,
  the class $x$ is divisible by $a_{V_\lambda}$.
  Proposition \ref{prop:coprime} (ii) then shows that $x$ is divisible by $a_V$,
  say $x=z\cdot a_V$. Hence $y=x/a_V{}=z/1$ is effective.
\end{proof}

\begin{rk}[Relation to Boardman's work in equivariant bordism]
  For the two-element group $C$, the statement of Theorem \ref{thm:detection} 
  simplifies: an element $y\in \Phi^C X$ 
  is effective if and only if the Laurent power series $d_1(y)$ is integral.
  For the oriented $\elRO$-algebra of equivariant bordism,
  Boardman defines a `stable bordism $J$-homomorphism' $J:\Phi_*^C\cN\to \cN_*\lr{\theta}$ 
  in \cite[\S 6]{boardman:revisited}, using different notation and geometric arguments.
  In this incarnation, the special case $A=C$ of Theorem \ref{thm:detection} 
  was proved by Boardman in \cite[Theorem 11]{boardman:revisited}.
  In this sense, our  Theorem \ref{thm:detection} is a vast generalization of
  Boardman's result to elementary abelian 2-groups of arbitrary rank
  and general oriented $\elRO$-algebras.
\end{rk}

For every oriented $\elRO$-algebra, the class $t_\mu/a_\mu\in \Phi^{C\times C}X$
also lies in the ring $\Psi^C X$, so that $d_C(t_\mu/a_\mu)=\delta_C(t_\mu/a_\mu)$,
and this class is a power series (as opposed to a Laurent power series) in $\theta$
with coefficients in $\Phi^C X$.

\begin{defn}\label{def:define_beta}
  Let $(X,a,t)$  be an oriented $\elRO$-algebra.
  For $n\geq 0$, we define the classes $\beta_n\in \Phi^C_{n+1}X$ by
  \[ d_C(t_\mu/a_\mu)\ = \ \sum_{n\geq 0} \beta_n\cdot \theta^n\ . \]
\end{defn}

For example, $\beta_0=i_1^*(t_\mu/a_\mu)=t/a$.
The next theorem might seem unmotivated right now,
but it will enter crucially in the proof of our main result, Theorem \ref{thm:main}.
In part (ii), what matters is not the precise formula,
but the fact that the classes $\delta_C(\mu^*(\beta_n))$
can be described purely in terms of the $2$-torsion formal group law
and the classes $\beta_i$ themselves.

\begin{theorem}\label{thm:determine_betas}
  Let $(X,a,t)$ be an oriented $\elRO$-algebra with 2-torsion formal group law $F$.
  \begin{enumerate}[\em (i)]
  \item 
    The images of the classes $\beta_n$ under the power series expansion
    $d_1:\Phi^C X\to X(1)\lr{\theta}$ satisfy the relation
    \[  F(\theta,\xi)\cdot \sum_{n\geq 0}\,  d_1(\beta_n)\cdot \xi^n \ = \ 1 \]
    in the ring $X(1)\lr{\theta}\dbr{\xi}$.
  \item 
    The classes $\delta_C(\mu^*(\beta_n))\in (\Phi^C X)\dbr{\theta}$ satisfy the relation
    \[ \sum_{n\geq 0}\delta_C(\mu^*(\beta_n))\cdot \xi^n \ = \ \sum_{n\geq 0}\beta_n\cdot F(\theta,\xi)^n \]
    in the ring $(\Phi^C X)\dbr{\theta,\xi}$.
  \end{enumerate}
\end{theorem}
\begin{proof}
  (i)
  We specialize Construction \ref{con:d_K} to $K=\{1\}$.
  Then $\Psi^1 X=X(C,\star)[t^{-1}]=(t^{-1}X)(C)$,
  with the potentially confusing caveat
  that the class that is generically called $e_2$ in $\Phi^1 K$
  in Construction \ref{con:d_K} is called $e_1$ in $(t^{-1}X)(C)$.
  Under these identifications,
  the homomorphism $\delta_1:\Psi^1 X\to X(1)\dbr{\theta}$ coincides with 
  $d^1:(t^{-1}X)(C)\to X(1)\dbr{e_1}$ up to renaming of the variables.
  So for all $x\in X(C,*-k\sigma)$, the relation
  \begin{equation}\label{eq:d_1's}
      d^1(x/t^k)\cdot e_1^{-k}    \ = \     \delta_1(x/t^k)\td{e_1}\cdot e_1^{-k}
    \ = \ (\delta_1(x/t^k)\cdot \theta^{-k})\td{e_1}\ = \ d_1(x/a^k)\td{e_1}
  \end{equation}
  holds in $X(1)\lr{e_1}$. Here, and in what follows, we  write $\td{e_1}$
  to indicate a renaming of the variable from $\theta$ to $e_1$.

  The classes $d^2(e_1)=e_1$ and $d^2(e_\mu)=F(e_1,e_2)$ both become invertible
  in the ring $X(1)\lr{e_1}\dbr{e_2}$, so the power series expansion
  $d^2: (t^{-1}X)(C^2)\to  X(1)\dbr{e_1,e_2}$
  from Construction \ref{con:d^n} extends
  uniquely to a morphism of graded $X(1)$-algebras
  \[ d^2\ :\  (t^{-1}X)(C^2)[e_1^{-1},e_\mu^{-1}]\ \to \ X(1)\lr{e_1}\dbr{e_2}\ ,  \]
  for which we use the same name.

  The ring $(t^{-1}X)(C^2)[e_1^{-1},e_\mu^{-1}]$
  is the localization of $X(C^2,\star)$ obtained by inverting
  all $t$-classes and the classes $a_1$ and $a_\mu$.
  Since $\Psi^C X$ is a less drastic localization
  -- formed by inverting $t_2$, $a_1$ and $a_\mu$ --
  it affords a localization morphism
  $\ell:\Psi^C X\to (t^{-1}X)(C^2)[e_1^{-1},e_\mu^{-1}]$
  that is localization away from $t_1/a_1$ and $t_\mu/a_\mu$.
  This morphism is explicitly given by
  \[ \ell(x/(a_1^k a_\mu^l t_2^m)) \ = \ (x/(t_1^k t_\mu^l t_2^m)) \cdot e_1^{-k} e_\mu^{-l} \ . \]
  We claim that the following diagram of graded rings commutes:
  \[\xymatrix@C=10mm{
      \Psi^C X\ar[rr]^-\ell\ar[d]_{\delta_C}&&
      (t^{-1}X)(C^2)[e_1^{-1},e_\mu^{-1}] \ar[d]^{d^2}\\
      (\Phi^C X)\dbr{\theta}\ar[r]_-{(-)\td{e_2}}^-\iso &  (\Phi^C X)\dbr{e_2}\ar[r]_-{d_1\dbr{e_2}}&
      X(1)\lr{e_1}\dbr{e_2}
    }\]
  Indeed, on the one hand, the rectangle commutes after precomposition with
  $p_1^*:\Phi^C X\to\Psi^C X$: the two composites send $p_1^*(x/a^k)$,
  for $x/a^k\in \Phi^C X$, to $d_1(x/a^k)\td{e_1}$ and  $d^1( x/t^k)\cdot e_1^{-k}$,
  respectively, which agree by \eqref{eq:d_1's}. 
  On the other hand, both composites in the rectangle send the regular generator $e_2$
  of the kernel of $i_1^*:\Psi^C X\to\Phi^C X$ to the class with the same name.
  So the rectangle commutes by  Proposition \ref{prop:power expansion} (ii).  
  Now we can prove the claim: 
  \begin{align*}
    F(e_1,e_2)\cdot \sum_{n\geq 0}\,  d_1(\beta_n)(e_1)\cdot e_2^n \
    &= \  d^2(e_\mu)\cdot d_1\dbr{e_2}(\delta_C(t_\mu/a_\mu)\td{e_2})\\
    &= \  d^2(e_\mu)\cdot  d^2(\ell(t_\mu/a_\mu))\ = \
    d^2(e_\mu\cdot e_\mu^{-1})\ = \ 1 
  \end{align*}
  This is the desired relation, up to renaming $e_1$ and $e_2$ to $\theta$ and $\xi$.
  
  (ii)
  We consider the $\mZ$-graded ring
  \[ \Xi(X)\ = \ X(C^3,\star)[a_1^{-1},a_{12}^{-1},a_{\bar\mu}^{-1},t_2^{-1},t_{23}^{-1},t_3^{-1}] \ ,\]
  the integer graded subring of the localization
  of the $I_{C^3}$-graded ring $X(C^3,\star)$ by inverting

 \begin{itemize}
  \item 
  the $a$-classes for the projection to the first factor, the sum of the first two factors,
  and the total multiplication $\bar\mu:C^3\to C$,
\item
  and the $t$-classes for the sum of the last two factors
  and the projections to the second and to the third factor.
\end{itemize}
  The restriction $i_{12}^*:X(C^3,\star)\to X(C^2,\star)$
  to the first two factors satisfies
  \[ i_{12}^*(a_1)\ = \ a_1\ , \quad
    i_{12}^*(a_{12})\ = \ i_{12}^*(a_{\bar\mu})\ = \ a_\mu\ , \quad
    i_{12}^*(t_2)\ = \  i_{12}^*(t_{23})\ = \ t_2\text{\qquad and\qquad}
    i_{12}^*(t_3)\ = \ 1\ ;  \]
  so it descends to a homomorphism of localizations
  \[ i_{12}^*\ :\ \Xi(X)\ \to \ \Psi^C X\ , \]
  where $\Psi^C X$ was defined in \eqref{eq:define_Psi}.
  The morphisms
  \[
    p_{12}^*\ , \ (C\times \mu)^*\ , \ (\mu\times C)^*\ :\ X(C^2,\star)\ \to\  X(C^3,\star) 
  \]
  satisfy
  \begin{align*}
    p_{12}^*(a_1)\ &= \ (C\times \mu)^*(a_1)\ = \ a_1\ ,\quad
                     p_{12}^*(a_\mu)\ = \ (\mu\times C)^*(a_1)\ = \
                     a_{12}\ ,\quad
    p_{12}^*(t_2)\ = \ t_2 \\
        (C\times \mu)^*&(a_\mu)\ = \ a_{\bar\mu}\text{\qquad and\qquad}
    (C\times \mu)^*(t_2)\ = \ t_{23}\\  
    (\mu\times C)^*&(a_\mu)\ = \ a_{\bar\mu}\text{\qquad and\qquad}
                     (\mu\times C)^*(t_2)\ = \ t_3\ .
  \end{align*}
  So all three extend to the localizations and yield homomorphisms
  \[
    p_{12}^*\ , \ (C\times \mu)^* \ , \ (\mu\times C)^* \  :\ \Psi^C X\ \to\  \Xi(X)\ .
  \]
  Because $(C\times\mu)\circ i_{12}=p_{12}\circ i_{12}=\Id$,
  the first two of these homomorphisms are sections to $i_{12}^*$.
  Because localization is exact, the class $e_3=a_3/t_3$ is a regular element of $\Xi(X)$
  that generates the kernel of $i_{12}^*$.
  We let 
  \[ \partial \ : \ \Xi(X)\ \to \  (\Psi^{C}X)\dbr{\xi} \]
  be the power series expansion, in the sense of Definition \ref{def:series_expansion},
  of $\Xi(X)$ with respect to the class $e_3$ and the section $p_{12}^*$.
  We will show that the three subdiagrams of the following
  large diagram of graded rings commute:
\[ 
 \xymatrix@C=12mm
    {   
    &   \Psi^C X \ar[rr]^-{\delta_C} \ar[d]^{(C\times \mu)^*} &&
   (\Phi^C X)\dbr{\theta} \ar@/^1pc/[dr]^-{F_*}&\\
  \Phi^C X\ar@/^1pc/[ur]^{\mu^*}\ar@/_1pc/[dr]_{\mu^*}  &\Xi(X) \ar[rr]^-{\partial} &&
      (\Psi^C X)\dbr{\xi} \ar[r]^-{\delta_C\dbr{\xi}} &
         (\Phi^C X)\dbr{\theta,\xi}\\ 
         &  \Psi^C X\ar[u]_{(\mu\times C)^*} \ar[r]_-{\delta_C}
         & (\Phi^C X)\dbr{\theta} \ar[r]^-\iso_-{\td{\xi}}&(\Phi^C X)\dbr{\xi}\ar[u]_{\mu^*\dbr{\xi}}
  }
 \]
 Here $F_*$ is the graded $\Phi^C X$-algebra map that
 sends $\theta$ to the formal group law $F(\theta,\xi)$.
 And we continue to use pointy brackets $\td{-}$ to indicate isomorphisms that
 rename variables.
 The left subdiagram with target $\Xi(X)$
 commutes by the relation $\mu\circ(C\times \mu)=\mu\circ(\mu\times C)$.

  Now we show the commutativity of the upper subdiagram featuring $F_*$.
  We appeal to the uniqueness statement in  Proposition \ref{prop:power expansion} (ii).
  After precomposition with the homomorphism $p_1^*:\Phi^C X\to \Psi^C X$,
  both composites in the upper part become the homomorphism
  $-\cdot 1:\Phi^C X\to (\Phi^C X)\dbr{\theta,\xi}$, the inclusion of the coefficient ring.
  We claim that both composites in the upper subdiagram also agree on $e_2$.
  The inflation homomorphism
  \[  p_{23}^*\ : \ X(C^2,\star)\ \to \ X(C^3,\star) \]
  along the projection to the last two factors satisfies
  \[     p_{23}^*(t_1)\ = \ t_2\ ,\quad
    p_{23}^*(t_\mu)\ = \ t_{23}\ ,\quad
    p_{23}^*(t_2)\ = \ t_3 \ ,
  \]
  and thus passes to a homomorphism from the localization away from $t$ of the source to $\Xi(X)$.
  The composite 
  \begin{equation}\label{eq:tinv2e_2e_3}
    (t^{-1}X)(C^2)\
    \xra{\ p_{23}^*\ }\  \Xi(X) \ \xra{\ \partial\ }\  (\Psi^C X)\dbr{\xi} 
    \ \xra{\delta_C\dbr{\xi}} \
      (\Phi^C X)\dbr{\theta,\xi}    
  \end{equation}
  is a morphism of $X(1)_*$-algebras, and it sends the classes $e_1$ and $e_2$
  to $\theta$ and $\xi$, respectively.
  Since the regular sequence $(e_1,e_2)$ generates the augmentation ideal of $(t^{-1}X)(C\times C)$,
  the uniqueness result from Proposition \ref{prop:power expansion} (ii)
  shows that the composite \eqref{eq:tinv2e_2e_3}  agrees with
  \[      (t^{-1}X)(C^2)\ \xra{\ d^2\ }\ 
    X(1)\dbr{e_1,e_2}\  \xra[\iso]{\td{\theta,\xi}}\ 
    X(1)\dbr{\theta,\xi}\  \xra{\ p^*\dbr{\theta,\xi}\ }\ (\Phi^C X)\dbr{\theta,\xi} \ . \]
  Here $p^*:X(1)\to \Phi^C X$ is the inflation homomorphism.
  So these two homomorphisms also agree on the class $e_\mu$.
  Because
  \[ p_2\circ (C\times\mu)\ = \  \mu\circ p_{23}\ : \ C^3 \ \to \ C\ ,\]
  we have $(C\times\mu)^*(e_2)=p_{23}^*(e_\mu)$ in $\Xi(X)$.
  Thus
  \begin{align*}
    \delta_C\dbr{\xi}(\partial((C\times\mu)^*(e_2)))\
    &= \ \delta_C\dbr{\xi}(\partial(p_{23}^*(e_\mu)))\ = \
      p^*\dbr{\theta,\xi}( d^2(e_\mu)\td{\theta,\xi})\\
    &= \ p^*\dbr{\theta,\xi}(F(\theta,\xi))\ = \ F_*(\theta)\ = \ F_*(d_C(e_2))\ . 
  \end{align*}
  Since $e_2$ is a regular element that generates the kernel of $i_1^*:\Psi^C X\to\Phi^C X$,
  the uniqueness result from  Proposition \ref{prop:power expansion} (ii)
  shows that the upper part of the diagram commutes.

  It remains to show the commutativity of the lower subdiagram,
  and one more time we appeal to Proposition \ref{prop:power expansion} (ii).
  This time both composites with
  the homomorphism $p_1^*:\Phi^C X\to \Psi^C X$ become the composite
  \[ \Phi^C X\ \xra{\ \mu^* \ }\  \Psi^C X\ \xra{-\cdot 1}\  (\Psi^C X)\dbr{\xi}\ ,  \]
  and both composites send the regular element $e_2$ to $\xi$.
  So the lower part of the big diagram commutes, too.

  We can now prove the desired relation:
  \begin{align*}
    \sum_{n\geq 0} d_C(\mu^*(\beta_n))\cdot \xi^n\
       &= \ (d_C\circ\mu^*)\dbr{\xi}(\sum_{n\geq 0}\beta_n\cdot\xi^n)\
    = \ (\delta_C\circ\mu^*)\dbr{\xi}( \delta_C(t_\mu/a_\mu)\td{\xi} )\\
     & = \ F_*(\delta_C(t_\mu/a_\mu))\
    = \ F_*(\sum_{n\geq 0} \beta_n\cdot \theta^n)\
     = \ \sum_{n\geq 0} \beta_n\cdot F(\theta,\xi)^n \ .
  \end{align*}
  The third equality is the commutativity of the big diagram, applied to the element $t/a$.
\end{proof}

We conclude this section with a criterion for when a system of morphisms
of geometric fixed point rings arises from a morphism of oriented $\elRO$-algebras.
This will be used in the proof of our main result, Theorem \ref{thm:main},
showing that equivariant bordism is an initial  oriented $\elRO$-algebra.

We let $\el_2^{\epi}$ denote the category whose objects are all elementary abelian 2-groups,
and whose morphisms are surjective group homomorphisms.
An {\em $\el_2^{\epi}$-algebra} is a contravariant functor from $\el_2^{\epi}$ to
the category of commutative $\mZ$-graded $\mF_2$-algebras.
For every oriented $\elRO$-algebra $X$, the collection of geometric fixed point
rings $\Phi^A X$ supports inflation homomorphisms \eqref{eq:inflation_geometric}
along morphisms in $\el^{\epi}_2$, so they naturally form an $\el_2^{\epi}$-algebra
$\Phi^\bullet X$.
Geometric fixed points are often easier to compute than the $\elRO$-algebras themselves.
Therefore the following theorem will be useful
to construct morphisms of orientable $\elRO$-algebras through their geometric fixed points.

\begin{theorem}\label{thm:fix_criterion}
  Let $(X,a,t)$ and $(Y,\bar a,\bar t)$ be oriented $\elRO$-algebras.
  Let $f^\bullet=\{f^A\}_A:\Phi^\bullet X\to\Phi^\bullet Y$ be a morphism of $\el_2^{\epi}$-algebras.
  Suppose that the following two conditions are satisfied:
  \begin{itemize}
    \item[(a)] $f^C(t/a)=\bar t/\bar a$
    \item[(b)] For every elementary abelian 2-group $K$, the following diagram commutes:
    \[ \xymatrix@C=15mm{
        \Phi^{K\times C} X\ar[r]^-{f^{K\times C}}\ar[d]_{d_K} & \Phi^{K\times C}Y\ar[d]^{d_K} \\
        (\Phi^K X)\lr{\theta}\ar[r]_-{f^K\lr{\theta}} & (\Phi^K Y)\lr{\theta}
      } \]
    \end{itemize}
    Then there exists a morphism $f\colon X\to Y$
    of oriented $\elRO$-algebras such that $\Phi^A f=f^A$
    for all elementary abelian $2$-groups $A$.
  \end{theorem}
  \begin{proof}
  We start by observing 
  \[ f^A(t_\lambda/a_\lambda)\ = \ f^A(\lambda^*(t/a))\ =_{\textrm{(b)}} \
    \lambda^*(f^C(t/a))\ =_{\textrm{(a)}} \ \lambda^*(\bar t/\bar a)\ = \ \bar t_\lambda/\bar a_\lambda  \]
  for all nontrivial characters $\lambda$ of an elementary abelian 2-group $A$.
  
  We recall from Definition \ref{def:effective} that a class $y\in \Phi_k^A X$
  is effective if it lies in the image of the homomorphism $-/1:X(A)_k\to \Phi_k^A X$,
  and that the maps $-/1:X(A)_*\to \Phi_*^A X$
  form an isomorphism from the $\mZ$-graded part of $X(A)$ to the subring of effective classes.
  \smallskip

  Claim 1: for every elementary abelian 2-group $A$, the morphism $f^A:\Phi^A X\to\Phi^A Y$
  takes effective classes to effective classes.
  
  To prove this claim we let $y\in \Phi^A X$ be an effective class,
  and we let $\alpha:K\times C\iso A$ be an isomorphism.
  Because the $f$-maps commute with inflation along $\alpha$ , we have
  \[  d_K(\alpha^*(f^A(y))) \ = \  d_K(f^{K\times C}(\alpha^*(y))) \
    =_{\textrm{(b)}} \    f^K\lr{\theta}(d_K(\alpha^*(y))) \ . \]
  By Theorem \ref{thm:detection}, the class $d_K(\alpha^*(y))$
  in $(\Phi^K Y)\lr{\theta}$ is integral. 
  Since $f^K\lr{\theta}:(\Phi^K X)\lr{\theta}\to (\Phi^K Y)\lr{\theta}$
  preserves integrality, the class $d_K(\alpha^*(f^A(y)))$ in $(\Phi^K Y)\lr{\theta}$ is integral. 
  Again by Theorem \ref{thm:detection}, the class $f^A(y)$ is effective.\smallskip
  
  Because the ring homomorphism $f^A$ takes effective classes to effective classes,
  there are unique maps $f(A)\colon X(A)_*\to Y(A)_*$ that make the diagram
  \[ \xymatrix@C=12mm{ X(A)_* \ar[d]_{-/1}\ar[r]^{f(A)} & Y(A)_* \ar[d]^{-/1} \\
      \Phi_*^A X \ar[r]_{f^A} & \Phi_*^A Y}
  \]
  commute.
  These maps are homomorphisms of graded rings and commute with inflations
  because this is the case for the other three
  morphisms in the commutative diagram, and because the vertical maps are injective.
  \smallskip

  Claim 2: For every monomorphism $i:K\to A$ between elementary abelian 2-groups, the
  following square commutes:
  \[\xymatrix@C=12mm{X(A)_*\ar[r]^-{f(A)}\ar[d]_{i^*}& Y(A)_*\ar[d]^{i^*} \\
      X(K)_*\ar[r]_-{f(K)} & Y(K)_*   }\]
  It suffices to show this when $i(K)$ has index 2 in $A$, and then we may assume
  without loss of generality that $A=K\times C$ and $i=i_1:K\to K\times C$ is
  the embedding as the first summand.
  For $u\in X(K\times C)_*$ the relation
  \[ 
    f^K\lr{\theta}(d_K(u/1))\ =_{\textrm{(b)}} \  d_K(f^{K\times C}(u/1)) \ = \ d_K(f(K\times C)(u)/1)
  \]
  holds in the ring $(\Phi^K_* Y)\lr{\theta}$.
  Since both sides of this equation are integral classes by Theorem \ref{thm:detection},
  the relation already holds in the ring $(\Phi^K_* Y)\dbr{\theta}$;
  in particular, the constant terms of both sides agree, and we obtain
    \begin{align*}
    ( f(K)(i_1^*(u))/1 \
    &= \     f^K(i_1^*(u)/1) \ = \  f^K(d_K(u/1)(0)) \\
    &= \  f^K\lr{\theta}(d_K(u/1))(0) \\
      _{\textrm{(b)}}&= \ d_K(f(K\times C)(u)/1)(0)\ = \ i_1^*(f(K\times C)(u))/1 \ .
  \end{align*}
  Because the map $-/1$ is injective, this concludes the proof of Claim 2.\smallskip

  We can now define the desired morphism of $\elRO$-algebras $f:X\to Y$.  
  We let $W$ be a representation of an elementary abelian 2-group $A$ with $W^A=0$.
  By Proposition \ref{prop:generation},
  every class $x\in X(A,k-W)$ is a homogeneous polynomial with coefficients in
  $X(A)_*$ of degree $k-W$ in the classes $a_\lambda$ and $t_\lambda$,
  for all nontrivial $A$-characters $\lambda$.
  In other words, there is a homogeneous polynomial $g\in X(A)_*[\alpha_\lambda,\tau_\lambda]$
  in variables $\alpha_\lambda$ and $\tau_\lambda$
  of degree $-\lambda$ and $1-\lambda$, respectively, such that
  $x=g(a_\lambda,t_\lambda)$, i.e., substituting the variables 
  $\alpha_\lambda$ and $\tau_\lambda$ by the classes $a_\lambda$ and $t_\lambda$ yields $x$.
  We write $f(A)_\flat(g)\in Y(A)_*[\alpha_\lambda,\tau_\lambda]$
  for the effect of applying the graded ring homomorphism
  $f(A):X(A)_*\to Y(A)_*$ to the coefficients of the polynomial $g$.
  The classes $\bar a_\lambda$ and $\bar t_\lambda$ of the $\elRO$-algebra $Y$
  can then be substituted for the variables in $f(A)_\flat(g)$,
  yielding a class in $Y(A,k-W)$.
  There is no reason why the polynomial $g$ that presents the class $x$ should be unique,
  so we shall need:\smallskip
  
  Claim 3: The class $f(A)_\flat(g)(\bar a_\lambda,\bar t_\lambda)$ in $Y(A,k-W)$
  is independent of how $x=g(a_\lambda,t_\lambda)$ is presented as a polynomial
  in the classes $a_\lambda$ and $t_\lambda$.

  For every homogeneous polynomial $h \in Y(A)_*[\alpha_\lambda,\tau_\lambda]$
  of the appropriate degree we have
  \[    h(\bar a_\lambda,\bar t_\lambda)/\bar a_W
    \ = \ h(1,\bar t_\lambda/\bar a_\lambda)\]
  in $\Phi^A_k Y$, i.e., the class  $h(\bar a_\lambda,\bar t_\lambda)/\bar a_W$
  is obtained by substituting the variables $\alpha_\lambda$ and $\tau_\lambda$ by
  the multiplicative unit $1$ and by $\bar t_\lambda/\bar a_\lambda$, respectively.
  For $h=f(A)_\flat(g)$, this yields
  \begin{align*}
    f(A)_\flat(g)(\bar a_\lambda,,\bar t_\lambda)/\bar a_W\
    &= \    f(A)_\flat(g)(1,\bar t_\lambda/\bar a_\lambda)\
      = \ f(A)_\flat(g)(1,f^A(t_\lambda/a_\lambda)) \\
    &= \     f^A(g(1,t_\lambda/a_\lambda)) \
     \    f^A(g(a_\lambda ,t_\lambda)/a_W) \ = \   f^A(x/a_W) \ .   
  \end{align*}
   This geometric fixed point class
   is thus independent of how $x$ is presented 
   as a polynomial in the classes $a_\lambda$ and $t_\lambda$.
   Since the map  $-/\bar a_W:Y(A,k-W)\to\Phi^A_k Y$ is injective,
   this proves Claim 3.\smallskip

   With Claim  3 established, we can now define a map
   $f(A,k-W):X(A,k-W)\to Y(A,k-W)$ by 
   \[ f(A,k-W)(x)\ = \ f(A)_\flat(g)(\bar a_\lambda,\bar t_\lambda) \ ,\]
   where $g\in X(A)_*[\alpha_\lambda,\tau_\lambda]$ is any homogeneous polynomial
   of degree $k-W$ such that $x=g(a_\lambda,t_\lambda)$.
   With the independence of $g$ at hand, it is then straightforward to 
   shows that these assignments define an $I_A$-graded ring homomorphism.
   And clearly, the map $f(A,-\lambda):X(A,-\lambda)\to Y(A,-\lambda)$
   sends $a_\lambda$ to $\bar a_\lambda$,
   and the map $f(A,1-\lambda):X(A,1-\lambda)\to Y(A,1-\lambda)$
   sends $t_\lambda$ to $\bar t_\lambda$.
   
   Now we check that for varying $A$, these homomorphisms of $I_A$-graded rings
   are compatible with restriction along homomorphisms $\beta:B\to A$
   of elementary abelian 2-groups.
   We already argued above that the relation $\beta^*\circ f(A,\star)=f(B,\star)\circ \beta^*$
   holds on the integer graded subrings.
   The $I_A$-graded ring $X(A,\star)$ is generated by its integer graded
   subring and the classes  $a_\lambda$ and $t_\lambda$ for all nontrivial $A$-characters $\lambda$,
   by Proposition \ref{prop:generation} (ii).
   So it suffices to show that the ring homomorphisms
   $\beta^*\circ f(A,\star)$ and $f(B,\star)\circ \beta^*$
   also agree on the classes $a_\lambda$ and $t_\lambda$.
   We distinguish two cases. If the $B$-character $\lambda\beta$ is nontrivial,
   then $\beta^*(a_\lambda)=a_{\lambda\beta}$,
   and we obtain
   \[ \beta^*(f(A,-\lambda)(a_\lambda))\
     =\ \beta^*(\bar a_\lambda) \ = \ \bar a_{\lambda\beta} \
     =\   f(B,-(\lambda\beta))(a_{\lambda\beta})\ = \
     f(B,-(\lambda\beta))(\beta^*(a_\lambda))\ . \]
   The argument for the $t$-classes is the same, mutatis mutandis.
   If the $B$-character $\lambda\beta$ is trivial, we need to argue differently.
   In this case $\beta^*(a_\lambda)=0$ in $X$, and
   $\beta^*(\bar a_\lambda)=0$ in $Y$, and the relation 
   $\beta^*(f(A,-\lambda)(a_\lambda))= f(B,-1)(\beta^*(a_\lambda))$
   holds as both sides are zero.
   Also in this case, $\beta^*(t_\lambda)=1$ in $X$
   and $\beta^*(\bar t_{\lambda})=1$ in $Y$, because $t$ and $\bar t$ are inverse Thom classes.
   Hence  $\beta^*(f(A,1-\lambda)(t_\lambda))$ and
   $f(B,0)(\beta^*(t_\lambda))$ are both the multiplicative unit 1.
   This finishes the proof.
\end{proof}

\section{Equivariant bordism}\label{sec:bordism}

In this section we introduce the oriented $\elRO$-algebra $\cN$ of {\em equivariant bordism},
and we collect various basic properties.
The main result of this paper is the fact that $\cN$ is an initial 
oriented $\elRO$-algebra, see Theorem \ref{thm:main} below.

\begin{construction}
  Given a finite group $G$ and a based $G$-space $X$,
  we write $\widetilde{\cN}^G_k(X)$ for the $k$-th reduced $G$-equivariant bordism group of $X$.
  We refer to \cite[\S 2]{stong:unoriented_finite} or \cite[Section 6.2]{schwede:global}
  for a detailed background on equivariant bordism.

  For an elementary abelian 2-group $A$, we write a given grading $m\in I_A$
  as $m=k-V$ for some integer $k$ and an $A$-representation $V$ with trivial fixed points.
  Then we set
  \[ \cN(A,m) \ = \ \cN(A,k-V) \ = \ \widetilde \cN_k^A(S^V)\ ,\]
  the reduced $A$-equivariant bordism group of dimension $k$
  of the onepoint compactification $S^V$. 
  A key point is that this definition is independent up to preferred isomorphism
  of the choice of $V$.
  Indeed,  an isomorphism of $A$-representations
  $\psi:V\to W$ induces an equivariant homeomorphism $S^\psi:S^V\to S^W$,
  and hence an isomorphism of bordism groups
  \[ \widetilde \cN_k^A(S^\psi) \ : \ \widetilde \cN_k^A(S^V)\ \to\
    \widetilde \cN_k^A(S^W)\ .  \]
  Theorem \ref{thm:oriented consequences} (iii)
  and the fact that equivariant bordism is represented
  by an orientable global ring spectrum (Example \ref{eg:bordism_versus_mO})
  guarantee that the isomorphism does not depend on the choice of $\psi$.
  Equivariant bordism is concentrated in non-negative integer degrees,
  so the group $\cN(A,k-V)$ is trivial whenever $k$ is negative.

  To define the functoriality we recall the suspension isomorphism in equivariant bordism.
  A specific class $d\in \widetilde{\cN}_1^A(S^1)$ is represented by
  the identity of $S^1$.
  The suspension homomorphism
  \[ -\sm d \ : \ \widetilde{\cN}^A_k(X) \ \to \ \widetilde{\cN}^A_{k+1}(X\sm S^1) \]
  is given by reduced product with the class $d$; it is an isomorphism,
  for example by \cite[Proposition 6.2.11]{schwede:global}.
  
  Now we let $\alpha:(B,\alpha^*(m))\to (A,m)$ be a morphism in the category $\elRO$.
  We choose an $A$-representation $V$ and a $B$-representation $W$
  so that $m=k-V$ and $\alpha^*(m)=k-\alpha^*(V)=l-W$ for integers $k\geq l$.
  There is then an isomorphism of $B$-representations
  \[ \psi\ :\     \alpha^*(V)\ \xra{\iso} \ W\oplus \mR^{k-l} \ . \]
  We define
  \[ \alpha^* \ : \  \cN(A,m) \ \to \ \cN(B,\alpha^*(m)) \]
  as the composite
  \[ \widetilde \cN_k^A(S^V) \ \xra{\ \alpha^* \ } \
    \widetilde \cN_k^B(S^{\alpha^*(V)})\ \xra[\iso]{\widetilde{\cN}_k^B(S^\psi)}\
    \widetilde \cN_k^B(S^{W\oplus \mR^{k-l}}) \ \xla[\iso]{-\sm d^{k-l}}\
        \widetilde \cN_l^B(S^W)\ .  \]
  The first homomorphism is restriction of the action on manifolds and bordisms
  along the homomorphism $\alpha$. The homomorphism decorated $-\sm d^{k-l}$
  is an iterated suspension isomorphism.
  Again by Theorem \ref{thm:oriented consequences} (iii),
  the definition of $\alpha^*$ does not depend on the choice of $\psi$;
  this independence is also needed in verifying that the assignment is compatible with composition.

  The pre-Euler class 
  \[ a\ \in\  \cN(C,-\sigma) \ = \ \widetilde{\cN}^C_0(S^\sigma) \]
  is represented by the inclusion  of the isolated fixed point $\{0\}\to S^\sigma$.
  The inverse Thom class
  \[ t\ \in\  \cN(C,1-\sigma) \ = \ \widetilde{\cN}^C_1(S^\sigma) \]
  is represented by the identity of $S^\sigma$,
  with equivariant smooth structure through its identification with $S(\mR\oplus\sigma)$
  by the stereographic projection.
\end{construction}

\begin{eg}[Equivariant bordism versus the global Thom spectrum $\bmO$]\label{eg:bordism_versus_mO}
  The global Thom spectrum $\bmO$ is defined in \cite[Example 6.1.24]{schwede:global};
  the underlying $G$-equivariant homotopy type of $\bmO$ has been considered before,
  often in different language and notation, for example in \cite[Section 5]{tomdieck:orbittypenII}.
  The spectrum $\bmO$ is the orthogonal spectrum whose $V$-th term is
  \[ \bmO(V)\ = \ T h(G r_{|V|}(V\oplus\mR^\infty)) \ ,\]
  the Thom space of the tautological vector bundle
  of $|V|$-dimensional subspaces of $V\oplus\mR^\infty$.
  We refer to \cite[Example 6.1.24]{schwede:global}
  for the definition of the structure maps and the $E_\infty$-multiplication on $\bmO$.
  
  For every compact Lie group $G$,
  the equivariant Pontryagin-Thom construction provides a
  natural transformation of $G$-equivariant homology theories from $G$-equivariant bordism
  to the $G$-homology theory represented by the global Thom spectrum $\bmO$,
  see \cite[Construction 6.2.27]{schwede:global}.
  If the group $G$ is a product of a finite group and a torus, then
  the equivariant Pontryagin-Thom construction is an isomorphism.
  This results goes back to Wasserman \cite{wasserman},
  a homotopy theoretic proof for finite groups was given by
  tom Dieck \cite[Satz 5]{tomdieck:orbittypenII},
  and a proof in the general case can be found in \cite[Theorem 6.2.33]{schwede:global}.
  The reader should beware that the equivariant Pontryagin-Thom construction is
  provably {\em not} an isomorphism beyond this class of groups;
  for $G=S U(2)$, the homotopy theoretic transfer from the maximal torus normalizer to $S U(2)$
  yields a class in $\pi_0^{S U(2)}(\bmO)$
  that is not in the image, see \cite[Remark 6.2.34]{schwede:global}.

  In any case, the Thom-Pontryagin construction is in particular an isomorphism
  \[ \Theta^A(X) \ : \ \widetilde{\cN}_*^A(X)\ \xra{\ \iso \ }\ \bmO_*^A(X)\]
  for every elementary abelian 2-group $A$ and every based $A$-CW-complex $X$.
  These isomorphisms are furthermore compatible with products and restriction along
  group homomorphisms, compare \cite[Theorem 6.2.31]{schwede:global}.
  Specializing to the case where $X=S^V$ is the onepoint compactification of an $A$-representation
  shows that the Thom-Pontryagin construction is an isomorphism
  of $\elRO$-algebras
  \[ \Theta\ : \ \cN \ \xra{\ \iso \ }\ \bmO^\sharp\ .\]
  Both in $\cN$ and in $\bmO^\sharp$, the pre-Euler class $a$ is
  obtained by applying the map of based $C$-spaces $S^0\to S^\sigma$
  to the multiplicative unit in the $C$-equivariant coefficients.
  Since the Thom-Pontryagin construction is natural in based equivariant maps,
  the pre-Euler classes of $\cN$ and $\bmO^\sharp$ match up.
\end{eg}

Several times in this paper we will need the identification
of the geometric fixed points of equivariant bordism in terms of fixed point data.
This technique, like several others in the subject, goes back to
the pioneering work of Conner and Floyd.

\begin{construction}[Geometric fixed points of equivariant bordism]
  We recall the description of the geometric fixed point rings of
  equivariant bordism in terms of non-equivariant bordism groups.
  We will need it for elementary abelian 2-groups $A$,
  but this fixed point description works in the generality of compact Lie groups,
  see for example \cite[Proposition 6.2.16]{schwede:global}.
  The result takes the form of an isomorphism of graded rings:
  \begin{equation} \label{eq:geometric fix iso}
    \varphi_A \ :\   \Phi^A_n\cN \ \xra{\ \iso \ } \ {\bigoplus}_{j\geq 0}\ \cN_{n-j}(G r_j^{A,\perp} )
  \end{equation}
  Here $G r_j^{A,\perp}$ is the moduli space of $j$-dimensional $A$-representations
  with trivial fixed points, and $\cN_{n-j}(G r_j^{A,\perp} )$ is its non-equivariant bordism group.
  More formally, $G r_j^{A,\perp}$ is the $A$-fixed point space
  of the Grassmannian of $j$-planes
  in the orthogonal complement of the $A$-fixed subspace of a complete $A$-universe.
  The multiplication on the target arises from the maps
  \[ G r_i^{A,\perp} \times G r_j^{A,\perp}\ \to \ G r_{i+j}^{A,\perp}  \]
  that classify the direct sum of $A$-representations with trivial fixed points.
  
  Roughly speaking, the isomorphism \eqref{eq:geometric fix iso} is defined as follows.
  Elements of $\Phi^A_n \cN$ are fractions $x/a_V$, where $x\in \widetilde{\cN}_n^A(S^V)$
  for some $n\geq 0$, and some $A$-representation $V$ with trivial fixed points.
  We let $(M,h)$ represent the class $x$,
  so that $M$ is an $n$-dimensional smooth closed $A$-manifold,
  and $h\colon M\to S^V$ is a continuous $A$-map; such pairs are taken up to bordism and modulo
  the subgroup of pairs where $h$ is constant at the basepoint.
  Because the $A$-action is smooth, the fixed set $M^A$ is a disjoint union of regularly embedded
  smooth submanifolds of $M$ of varying dimensions.
  The $A$-space $S^V$ has only two fixed points 0 and $\infty$,
  and only the fixed points over the non-basepoint 0 will play a role.
  The Conner--Floyd map \eqref{eq:geometric fix iso}
  sends the bordism class $[M,h]$ to $\sum_{j\geq 0} [M^{(j)},\nu_j]$,
  where $M^{(j)}$ is the union of the $(n-j)$-dimensional components of $M^A\cap h^{-1}(0)$,
  and where $\nu_j:M^{(j)}\to Gr_j^{A,\perp}$ classifies the $A$-equivariant
  normal bundle of $M^{(j)}$ inside $M$.
  We refer to \cite[Proposition 6.2.16]{schwede:global}
  for more details on the construction and for a proof that the Conner--Floyd map
  is an isomorphism.
\end{construction}  

The next theorem collects various key properties of equivariant bordism
that are relevant for this paper.

\begin{theorem}\label{thm:bordism_properties}\ 
  \begin{enumerate}[\em (i)]
  \item The pair $(\cN,a)$ is an orientable $\elRO$-algebra with a unique inverse Thom class $t$.
  \item  The 2-torsion formal group law of the oriented $\elRO$-algebra
    $(\cN,a,t)$ is an initial 2-torsion formal group law.
  \item For every elementary abelian 2-group $A$, the morphism of $\cN_*$-algebras
    \[ \bigotimes_{\cN_*}^{\lambda\in A^\circ}\  \Phi^C_*\cN \ \to \  \Phi^A_*\cN \]
    that multiplies the inflation homomorphisms for all nontrivial $A$-characters is an isomorphism.
  \end{enumerate}
\end{theorem}
\begin{proof}
  (i) As we detailed in Example \ref{eg:bordism_versus_mO},
  equivariant bordism is represented by the orientable global ring spectrum $\bmO$.
  Part (i) is thus a special case of Theorem \ref{thm:oriented_ring2oriented_el}.
  By the general theory, the set of inverse Thom classes of $(\cN,a)$
  is a torsor over the bordism group $\cN^C_1$; since this group is trivial,
  the inverse Thom class is unique.

  (ii)
  We reduce the claim to Quillen's theorem \cite{quillen:formal_groups},
  saying that the formal group law associated to the
  preferred, classical real orientation of the non-equivariant Thom spectrum $M O$
  is initial. To this end we recall Quillen's result in a form tailored to our purpose.

  We let $B O$ denote the classifying space of the infinite orthogonal group.
  The Thom spectrum construction assigns to a continuous map $f:X\to B O$
  from a CW-complex a Thom spectrum $T h(f)$, making it a functor
  \[ T h\ : \ Ho(Top/B O)\ \to \ \mathcal{SH} \]
  from the homotopy category of spaces over $B O$ to the stable homotopy category of spectra.
  The functoriality equips the Thom spectrum with a morphism
  of spectra $u(f):T h(f)\to T h(\Id_{B O})=MO$
  that is a tautological Thom class in $MO$-cohomology,
  in the sense that the group of stable maps $[T h(f),MO]=\mathcal{SH}(T h(f),MO)$
  is free of rank 1 as a module, under the Thom diagonal,
  over the ring $MO^0(X)=[\Sigma^\infty_+ X,MO]$,
  with $u(f)$ as generator.
  
  In this language, the preferred real orientation of the spectrum $M O$
  arises as the Euler class of the continuous map $\beta(\sigma-\mR):B C\to B O$ 
  that classifies the group homomorphism that sends the generator of the group $C$ to the
  infinite diagonal matrix with entries  $(-1,1,1,\dots)$.
  Some readers might prefer to think of continuous maps to $B O$
  as classifying virtual vector bundles of rank 0; our name `$\beta(\sigma-\mR)$' for the
  above map reflects that it classifies the virtual vector bundle over $B C$
  associated to the virtual $C$-representation $\sigma-\mR$.
  The Thom spectrum of $\beta(\sigma-\mR)$ is then equivalent to
  the desuspension of the reduced suspension spectrum of
  the space $S^\sigma\sm_C E C_+$ (which happens to be homotopy equivalent to $B C$),
  and its Thom class $u(\beta(\sigma-\mR))$ lies in the group  
  \[ [T h(\beta(\sigma-\mR)), MO]\ = \
       [S^{-1}\sm \Sigma^\infty  S^{\sigma}\sm_C E C_+,MO]\ = \
    \widetilde{M O}^1(S^\sigma\sm_C E C_+) \ .\]
  Its Euler class $e_\tau\in M O^1(B C)$, the restriction of the Thom class
  along the zero section $B C_+\to S^{\sigma}\sm_C E C_+$,
  is the preferred real orientation of $M O$.  
  
  By the theory of real-oriented cohomology theories, the ring
  $MO^*(B C\times B C)$ is then a power series algebra over $MO^*$ in the classes
  $p_1^*(e_\tau)$ and  $p_2^*(e_\tau)$, and expanding the class $(B\mu)^*(e_\tau)$
  in $p_1^*(e_\tau)$ and  $p_2^*(e_\tau)$ yields a 2-variable power series
  with coefficients in $MO^*$ that is a 2-torsion formal group law.
  Quillen's theorem \cite[Theorem 3]{quillen:formal_groups} says
  that this 2-torsion formal group law is initial.

  Now we use the global Thom ring spectrum $\bMO$ defined in \cite[Example 6.1.7]{schwede:global}
  that represents stable equivariant bordism.
  For every elementary abelian 2-group $A$,
  the equivariant homology theory represented by the global ring spectrum $\bMO$
  is the localization of $\bmO^A_*(-)$ at the inverse Thom classes,
  compare \cite[Corollary 6.1.35]{schwede:global}.
  The localization maps $\bmO^A_*(-)\to \bMO^A_*(-)$ at all
  representation spheres define a morphism of oriented $\elRO$-algebras
  $(\bmO^\sharp,a,t)\to(\bMO^\sharp,a,t)$.
  Here we denote the image of the unique inverse Thom class of $\bmO^\sharp$
  by the same letter  $t\in \bMO_1^C(S^\sigma)$;
  we alert the reader that in the orientable $\elRO$-algebra $\bMO^\sharp$,
  the inverse Thom class is far from unique.
  The previous morphism exhibits $\bMO^\sharp$ as
  the localization away from $t$ of $\bmO^\sharp$.
  In particular, this morphism is an isomorphism of underlying
  graded $\mF_2$-algebras, and so it necessarily takes
  the 2-torsion formal group law of $\bmO^\sharp$ to
  the 2-torsion formal group law of $\bMO^\sharp$.
  So it suffices to show that the 2-torsion formal group law of the oriented $\elRO$-algebra
  $(\bMO^\sharp,a,t)$ is an initial 2-torsion formal group law.

  By Theorem 6.1.7 and Remark 6.1.20 of \cite{schwede:global},
  the class $t$ in $\bMO_1^C(S^\sigma)$ is an $R O(C)$-graded unit. 
  The inverse of $t$ is the `shifted Thom class' $u=\bar\sigma_{C,\sigma}\in\bMO^1_C(S^\sigma)$
  discussed in \cite[Remark 6.1.20]{schwede:global}.
  We claim that the bundling homomorphism
  \begin{equation}\label{eq:beta_stable}
    \beta\ : \ \widetilde{\bMO}^1_C(S^\sigma) \ \to \ \widetilde{M O}^1(S^{\sigma}\sm_C E C_+)     
  \end{equation}
  takes the $C$-equivariant Thom class $u$ to the non-equivariant Thom class $u(\beta(\sigma-\mR))$
  discussed above.
  To prove this we refer to the $C$-equivariant Thom spectrum formalism,
  a general reference for which is \cite[Chapter X, \S 3]{lms}.
  We write $\bBO_C$ for the classifying $C$-space for virtual $C$-equivariant vector bundles
  of rank 0. Much like in the non-equivariant situation,
  the equivariant Thom construction
  turns $C$-equivariant maps $X\to\bBO_C$ into genuine $C$-spectra
  endowed with a morphism to $\bMO_C$, the underlying $C$-spectrum
  of the global spectrum $\bMO$.
  An explicit representative for the class $u$ in $\bMO^1_C(S^\sigma)$
  is the $C$-equivariant map
  \begin{align*}
     S^{\sigma\oplus\sigma}\ &\to \ 
    Th(G r_2(\sigma\oplus\mR\oplus\sigma\oplus\mR))\ = \ \bMO(\sigma\oplus\mR)\\
    (x,y)\ &\longmapsto \ ((x,0,y,0),\sigma\oplus 0\oplus \sigma\oplus 0)\nonumber
  \end{align*}
  So $u:S^{\sigma-1}\to\bMO_C$ is the $C$-equivariant Thom class
  of the $C$-equivariant map $*\to\bBO_C$ that classifies
  the virtual $C$-representation $2\sigma-(\sigma\oplus\mR)=\sigma-\mR$.
  The unstable and stable bundling maps take equivariant Thom spectra
  to non-equivariant Thom spectra, see \cite[X Corollary 6.3]{lms}.
  When applied to the $C$-map $\sigma-\mR:*\to\bBO_C$, this identifies the image of
  $u:S^{\sigma-1}=T h(\sigma-\mR)\to\bMO_C$ under the stable bundling map \eqref{eq:beta_stable}
  with the non-equivariant stable map 
  \[  u(\beta(\sigma-\mR))\ :\ S^{\sigma-1}\sm_C E C_+ = T h(\beta(\sigma-\mR))\ \to\  M O\]
  obtained by thomifying $\beta(\sigma-\mR):B C\to B O$.
  This proves the claim that $\beta(u)=u(\beta(\sigma-\mR))$.
  Naturality of the bundling homomorphisms
  for the $C$-equivariant fixed point inclusion $a:S^0\to S^\sigma$
  then shows that the bundling homomorphism
  \[ \beta\ : \ \bMO^C_{-1} = \bMO^1_C(S^0) \ \to \ \widetilde{M O}^1(S^0\sm_C E C_+)= MO^1(B C)
  \]
  takes the $C$-equivariant Euler class $e=a/t=a\cdot u$ of the
  oriented $\elRO$-algebra $(\bMO^\sharp,a,t)$ to the Euler class $e_\tau$
  that defines the preferred classical orientation of $M O$.

  Now we can wrap up.
  The formal group law of $\bMO^\sharp$ is defined by expanding the class $\mu^*(e)=\mu^*(a/t)$
  as a power series in $p_1^*(e)$ and  $p_2^*(e)$.
  And the formal group law in Quillen's work is the expansion of $\mu^*(e_\tau)$
  as a power series in $p_1^*(e_\tau)$ and  $p_2^*(e_\tau)$.
  Since the bundling maps are compatible with restriction along group homomorphisms
  and take $e$ to $e_\tau$, 
  the 2-torsion formal group law of the oriented $\elRO$-algebra $(\bMO^\sharp,a,t)$
  is the preferred 2-torsion formal group law over $MO^*$ that features in
  Quillen's theorem \cite{quillen:formal_groups}. So we have shown claim (ii).

  (iii)   We consider the following commutative square of graded commutative $\cN_*$-algebras:
  \[ \xymatrix{
      \bigotimes_{\cN_*}^{\lambda\in A^\circ}\  \Phi^C_*\cN \ar[r]\ar[d]^{\iso}_{\tensor \varphi_C} &
      \Phi^A_*\cN \ar[d]_\iso^{\varphi_A}\\
      \bigotimes_{\cN_*}^{\lambda\in A^\circ} \left( \bigoplus_{j\geq 0}\ \cN_{*-j}(G r_j^{C,\perp})\right) \ar[r]&
      \bigoplus_{j\geq 0}\ \cN_{n-j}(G r_j^{A,\perp} )
    }  \]
  The left vertical map is a tensor product of copies of the Conner--Floyd fixed point isomorphism
  \eqref{eq:geometric fix iso}. The right vertical map is the
  isomorphism \eqref{eq:geometric fix iso} for the group $A$.
  The lower horizontal map is also induced by inflation along characters.
  More precisely, the restriction to the tensor factor indexed by $(\lambda,j)$
  is induced by the continuous map $G r_j^{C,j}\to G r_j^{A,\perp}$ that identifies
  $\sigma^\infty$ with the $\lambda$-isotypical summand of the complete $A$-universe,
    and then takes Grassmannians of $j$-planes.
  
  Since the vertical maps in the commutative square are isomorphisms,
  we are reduced to showing that the lower horizontal map is an isomorphism.
  This, in turn, is a combination of two facts.
  Firstly, the nontrivial irreducible representations of $A$ are all 1-dimensional,
  and parameterized by the set  $A^\circ$ of nontrivial $A$-characters.
  So the isotypical decomposition of $A$-invariant subspaces becomes a homeomorphism
  \[  
    \coprod_{\sum i_\lambda = j} \left( \prod_{\lambda\in A^\circ} G r_{i_\lambda} \right)
    \ \iso \  G r_j^{A,\perp}  \ .\]
  The coproduct is indexed by $A^\circ$-tuples 
  of natural numbers $i_\lambda$ that add up to $j$.
  Secondly, all non-equivariant bordism groups of spaces
  are free as modules over the non-equivariant bordism ring,
  see for example \cite[Theorem 8.3]{conner-floyd}.
  So the K{\"u}nneth isomorphism identifies the graded bordism groups of a product
  with the tensor product of the bordism groups of the factors.
  If we take the grading shift by $j$ on the $j$-plane summands into account,
  these two facts combine into the statement that
  the  lower horizontal map in the square is an isomorphism.
  This completes the proof.
\end{proof}

\section{The universal property of equivariant bordism}

In this section we prove our main result Theorem \ref{thm:main},
saying that equivariant bordism is an initial oriented $\elRO$-algebra.
We will use Theorem \ref{thm:fix_criterion}, so we need a firm grasp
on the geometric fixed point rings of equivariant bordism.
It is well-known that the $C$-geometric fixed point ring $\Phi^C_*\cN$
of equivariant bordism is a polynomial algebra over the non-equivariant bordism ring
in infinitely many generators, one in every positive degree,
see for example \cite[Lemma 25.1]{conner:differentiable-2nd}.
A key ingredient for the proof of Theorem \ref{thm:main} is the fact that
the specific classes $\beta_n$ which were introduced
purely in terms of the structure as an oriented $\elRO$-algebra
in Definition \ref{def:define_beta} can be taken as such polynomial generators,
see Theorem \ref{thm:C-fix N}.
Unfortunately, we do not know a direct argument to prove this fact, so we introduce
auxiliary classes $\zeta_n\in \cN^C_{n+1}$ represented by explicit $C$-manifolds
in Construction \ref{con:zeta}.
Then we show that the classes $\zeta_n$ for $n\geq 1$, together with $t/a=\beta_0$
form polynomial generators for $\Phi^C\cN$, and that $\beta_n$ is
congruent to $\zeta_n$ modulo $(\beta_0,\dots,\beta_{n-1})$,
see Theorem \ref{thm:C-fix N} and Theorem \ref{thm:C-fix N:beta}.

\begin{construction}
  We define a group $G_n$ as the set $\mZ^n$ endowed with a group structure by
  \begin{equation}\label{eq:G_k acts}
 (x_1,\dots,x_n)\cdot (y_1,\dots,y_n)\ = \
    \left(x_1+  y_1,(-1)^{y_1}\cdot x_2+y_2,\dots,
       (-1)^{y_{n-1}}\cdot x_n+ y_n\right) \ .    
  \end{equation}
  In other words, $G_1=\mZ$, and $G_n=G_{n-1}\ltimes \mZ$,
  the semidirect product formed with respect to the $G_{n-1}$-action on $\mZ$
  via the sign of the last coordinate.
  We define a right action of the group $G_n$ on $\mR^n$
  by the same formula \eqref{eq:G_k acts},
  but now with the elements $x_i$ being real numbers as opposed to integers.
  This action   by affine linear transformations
  is free, properly discontinuous and cocompact.
  So the orbit space $\mR^n/G_n$ is simultaneously a smooth closed $n$-manifold
  and a classifying space for the group $G_n$.
  The mod 2 cohomology of $\mR^n/G_n$,
  and hence the  mod 2 group cohomology of $G_n$,
  is thus a Poincar{\'e} duality algebra with duality class in dimension $n$.
  For example, $\mR/G_1=\mR/\mZ$ is a circle, and $\mR^2/G_2$ is a Klein bottle.
  For $n\geq 2$, the manifold $\mR^n/G_n$ is not orientable.
\end{construction}

The following calculation of the mod 2 cohomology
algebra of the group $G_n$, or equivalently of the manifold $\mR^n/G_n$,
will enter into the proof that the classes $\zeta_n$ are polynomial generators
of $\Phi^C_*\cN$, compare Theorem \ref{thm:C-fix N}.
We expect that this calculation is well-known, but we were unable to find a reference.
The cases $n=1$ and $n=2$ are classical, being
the cohomology algebras of the circle and the Klein bottle.
For $i=1,\dots,n$, we denote by
\[ p_i \ : \ G_n \ \to \ \mF_2 \]
the projection to the $i$-th coordinate of $G_n$, taken modulo 2.
We view these homomorphisms are cohomology classes in $H^1(G_n;\mF_2)=\Hom(G,\mF_2)$.

\begin{prop}\label{prop:cohomology of G_n}
  The cohomology algebra $H^*(G_n;\mF_2)$ has an $\mF_2$-basis consisting of the classes
  \[ p_{i_1}\cdot\ldots\cdot p_{i_m} \]
  for all $1\leq i_1 < i_2 <\dots < i_m\leq n$.
  The multiplicative structure is determined by the relations
  \[ p_i^2\ = \
    \begin{cases}
      \  0 & \text{ for $i=1$, and}\\
      p_{i-1}\cdot p_i & \text{ for $2\leq i\leq n$.}
    \end{cases}
  \]
  In particular, $p_n^n=p_1\cdot p_2\cdot \ldots\cdot p_n$, and this class is the
  generator of $H^n(G_n;\mF_2)$.
\end{prop}
\begin{proof}
  We argue by induction on $n$; for $G_1=\mZ$, the result is well-known.
  Now we suppose that $n\geq 2$.
  We consider the cohomology algebra of $G_n$
  as module over the cohomology algebra of $G_{n-1}$ by restriction along the projection
  \[ q\ : \ G_n = G_{n-1}\ltimes\mZ \ \to \ G_{n-1} \ , \quad q(y_1,\dots,y_n)=(y_1,\dots,y_{n-1}) \ .\]
  The mod 2 cohomology of $\mZ$ is concentrated in dimensions 0 and 1,
  where it is 1-dimensional.
  So the action of $G_{n-1}$ on the mod 2 cohomology of $\mZ$ through $p_{n-1}$
  is necessarily trivial,
  and the Lyndon-Hochschild-Serre spectral sequence degenerates into a short exact sequence of
  graded  $H^*(G_{n-1};\mF_2)$-modules
  \[  0 \ \to \ H^*(G_{n-1};\mF_2) \xra{\ q^* \ }\
    H^*(G_n;\mF_2) \ \xra{\ \delta\ } \ H^{*-1}(G_{n-1};\mF_2)\ \to\ 0\ .\]
  Since $H^1(G_n;\mF_2)=\Hom(G_n,\mF_2)$ is the direct sum of
  $q^*(H^1(G_{n-1};\mF_2))$ and the span of the homomorphism $p_n:G_n\to \mF_2$,
  the connecting homomorphism  must satisfy $\delta(p_n)=1$.
  Hence  $H^*(G_n;\mF_2)$ is free as a $H^*(G_{n-1};\mF_2)$-module
  on the classes $1$ and $p_n$. So the claim about the $\mF_2$-basis
  and those multiplicative relations that do not involve $p_n$ follow by induction. 

  To prove the remaining multiplicative relation $p_n^2=p_{n-1}\cdot p_n$,
  we use naturality for the group homomorphism
  \[ h\ : \ G_n \ \to \ \mF_2\ltimes\mZ/4 \ , \quad
  h(y_1,\dots,y_n)\ = \ (y_{n-1}+ 2\mZ,\ y_n+ 4\mZ)\]
  to the semidirect product of $\mF_2$ acting by sign on $\mZ/4$, also known
  as the dihedral group of order~8.
  The cohomology ring of this dihedral group is well-known,
  see for example \cite[Chapter IV, Theorem 2.7]{adem-milgram};
  we only need the information that the two homomorphisms
  \[ \alpha , \beta\ : \ \mF_2\ltimes\mZ/4 \ \to\ \mF_2  \]
  given by $\alpha(x,y+4\mZ)= x+ y+ 2\mZ$ and  $\beta(x,y+4\mZ)= y+2\mZ$ satisfy 
  the cup product relation $\alpha\cdot\beta = 0$.
  Because $\alpha\circ h=p_{n-1}+p_n$ and $\beta\circ h=p_n$,
  naturality yields the desired multiplicative relation
  \[     p_{n-1} \cdot p_n\ +\  p_n^2 \ =\
    (p_{n-1}+p_n)\cdot p_n\ = \ h^*(\alpha\cdot\beta)\ = \ 0 \ .\qedhere \]
\end{proof}

The next proposition gives a criterion for recognizing $C$-equivariant bordism classes
as polynomial generators of $\Phi^C_*\cN$.
It is very similar to \cite[Lemma 4]{boardman:revisited}
which gives a characteristic number criterion
to detect when elements of $\cN_*(B O)$ can serve as polynomial generators.
This result ought to be well-known to experts,
but we do not know of a reference in this form.

\begin{prop}\label{prop:top SW power}
  Let $M$ be a smooth closed $(n+1)$-dimensional $C$-manifold.
  Let $F$ be the union of the $n$-dimensional components of $M^C$,
  and let $\nu:F\to\mR P^\infty$ classify the normal line bundle of $F$ inside $M$.
  Suppose that $\nu_*[F] \ne  0$ in $H_n(\mR P^\infty,\mF_2)$,
  where $[F]\in H_n(F;\mF_2)$ is the mod 2 fundamental class.
  Then $[M]/1$ is indecomposable in $\Phi^C_*\cN$ as an algebra over $\cN_*$.
\end{prop}
\begin{proof}
  By the Conner--Floyd isomorphism \eqref{eq:geometric fix iso}
  \[  \varphi_C\ : \ \Phi_*^C\cN\ \xra{\ \iso\ } \ \bigoplus_{j\geq 0}\ \cN_{*-j}(G r_j ) \ ,\]
  the class $[M]/1$ is indecomposable in $\Phi^C_*\cN$ if and only if the projection
  of $\varphi_C[M]$ to the linear summand $j=1$ in the direct sum decomposition 
  can serve as a generator in an $\cN_*$-basis of $\cN_*(\mR P^\infty)$.
  By the geometric description of $\varphi_C$, the projection
  to the linear summand is represented by the pair $(F,\nu)$
  consisting of the $n$-dimensional fixed point components
  and the map classifying their normal line bundles inside $M$.
  For every space $X$, the map
  \[  \cN_*(X)\ \to \ H_*(X;\mF_2)\ , \quad  [f:W\to X]\ \longmapsto \  f_*[W]\ ,\]
  passes to an isomorphism
  \[  \cN_*(X)\tensor_{\cN_*}\mF_2 \ \iso \ H_*(X;\mF_2)\ .\]
  So for $X=\mR P^\infty=G r_1$, a pair $(W,f)$
  consisting of a smooth closed $n$-manifold $W$ and a continuous map $f:W\to\mR P^\infty$
  represents a generator in an $\cN_*$-basis of $\cN_*(\mR P^\infty)$
  if and only if $f_*[W]$ is the nonzero element of $H_n(\mR P^\infty;\mF_2)$.
  Altogether, this proves the proposition.
\end{proof}

\begin{construction}[The $\zeta$-classes]\label{con:zeta}
  We define a left action of the group $G_{n-1}$ on the projective space $\mR P^2$ by
\[  (s_1,\dots,s_{n-1})\cdot [x:y:z]\  = \ [x:y:(-1)^{s_{n-1}}\cdot z]\ . \]
By taking balanced product with $\mR^{n-1}$ we obtain a smooth closed $(n+1)$-manifold
\[  W_{n+1}\ = \ \mR^{n-1}\times_{G_{n-1}} \mR P^2 \ . \]
We define a smooth involution $\tau:W_{n+1}\to W_{n+1}$ by
\[  \tau[s_1,\dots,s_{n-1},[x:y:z]]\ = \  [s_1,\dots,s_{n-1},[-x:y:z]]\ .  \]
We write
\[ \zeta_n \ = \ [W_{n+1},\tau]/1 \ \in \ \Phi^C_{n+1}\cN \]
for the bordism class of the $C$-manifold $(W_{n+1},\tau)$, or rather the effective
class in the geometric fixed point ring that it represents.
\end{construction}

\begin{theorem}\label{thm:C-fix N}
  The classes $\zeta_n$ for $n\geq 1$ and the class $t/a$ form a set of polynomial generators
  of the $C$-geometric fixed points of equivariant bordism as an algebra over $\cN_*$.
\end{theorem}
\begin{proof}
  We verify the criterion of Proposition \ref{prop:top SW power}.
  The fixed points of the involution $\tau:W_{n+1}\to W_{n+1}$ are
  \begin{align*}
     (W_{n+1})^C \
     &= \  \mR^{n-1}\times_{G_{n-1}} ( \mR P^2)^C \\
     &= \ \mR^{n-1}\times_{G_{n-1}} (P(0\oplus \mR^2) \cup \{\mR\oplus 0\oplus 0\})\\
     &\iso \ \left(  \mR^{n-1}\times_{G_{n-1}}P(0\oplus \mR^2)\right)\  \cup\ \mR^{n-1}/G_{n-1}\ .
   \end{align*}
   This fixed point manifold has two connected components, one of dimension $n$ and one of dimension $n-1$.
   For our present purposes, we only care about the $n$-dimensional component
   \[ \mR^{n-1}\times_{G_{n-1}}P(0\oplus \mR^2) \ ,  \]
   and we can ignore the other one. To apply Proposition \ref{prop:top SW power},
   we need to identify the normal line bundle of this component.

   We let $p_n:G_n\to\{\pm 1\}$
   denote the group homomorphism given by $p_n(s_1,\dots,s_n)=(-1)^{s_n}$.
   We embed the total space of the line bundle over $\mR^n/G_n$ associated to $p_n$
   into the manifold $W_{n+1}$ by
   \begin{align*}
            \mR^n\times_{G_n} p_n^*(\sigma) \ &\to \quad 
     \mR^{n-1}\times_{G_{n-1}} P(\mR^3) \ = \ W_{n+1} \\
     [s_1,\dots,s_n;u]\ &\longmapsto\ [s_1,\dots,s_{n-1},[u:\sin(s_n\cdot\pi):\cos(s_n\cdot \pi)]]\ .
   \end{align*}
   This map is a smooth open embedding with image 
   $\mR^{n-1}\times_{G_{n-1}}( \mR P^2\setminus \{\mR\oplus 0\oplus 0\})$,
   and it identifies the zero section  $\mR^n/G_n$ in the source with
   the $n$-dimensional fixed point component $\mR^{n-1}\times_{G_{n-1}}P(0\oplus\mR^2)$.
   So this embedding witnesses that the normal line bundle of 
   $\mR^{n-1}\times_{G_{n-1}}P(0\oplus\mR^2)$ inside $W_{n+1}$ is the
   line bundle over $\mR^n/G_n$ associated to the character $p_n$.

   Proposition \ref{prop:top SW power} thus shows that the bordism class of $W_{n+1}$
   can serve as one of the polynomial generators for $\Phi_*^C\cN$,
   provided we show that the map $\mR^n/G_n\to \mR P^\infty$
   that classifies the line bundle associated to $p_n$
   is nonzero in mod 2 homology $H_n(-;\mF_2)$.
   Or equivalently, the classifying map is nonzero in mod 2 cohomology $H^n(-;\mF_2)$.
   Since the $n$-manifold $\mR^n/G_n$ is also a classifying space for the group $G_n$,
   we may show that the group homomorphism $p_n$ induces a nontrivial map
   in mod 2 group cohomology in dimension $n$.
   By Proposition \ref{prop:cohomology of G_n},
   the group cohomology class in $H^1(G_n;\mF_2)$ represented by $p_n$ 
   has the property that its $n$-th cup power
   is the generator of the top cohomology group $H^n(G_n;\mF_2)$.
   This concludes the argument.
 \end{proof}

 Our next aim is to relate the classes $\beta_n$ introduced in
 Definition \ref{def:define_beta} as the coefficients of the power series
 $d_C(t_\mu/a_\mu)$ to the geometrically defined classes $\zeta_n$.
 To this end, we shall make use of the `division operator'.

\begin{construction}[Division operator]
  We let $(X,a,t)$ be an oriented $\elRO$-algebra.
  For an elementary abelian 2-group $K$ and $k\in\mZ$, we define an additive map
  \begin{equation}\label{eq:define_Gamma}
    \Gamma\ :\ X(K\times C)_k\ \to\  X(K\times C)_{k+1}   \ ;
  \end{equation}
  we refer to $\Gamma$ as a {\em division operator}.
  We write $i_1:K\to K\times C$ and $p_1:K\times C\to K$ for the embedding of the first factor,
  and the projection to the first factor, respectively.  
  For $u\in X(K\times C)_k$, we have the relation
  \[ i_1^*(t_2\cdot( u + p_1^*(i_1^*(u)))) \ = \  i_1^*( u) + i_1^*(p_1^*(i_1^*(u)))\ = \ 0 \ . \]
  So there is a unique class $\Gamma(u)\in X(K\times C)_{k+1}$ such that
  \begin{equation}\label{eq:defining_relation_Gamma}
    a_2\cdot \Gamma(u)\ = \ t_2\cdot( u + p_1^*(i_1^*(u)))    
  \end{equation}
  in the group $X(K\times C, (k+1)-p_2)$.
\end{construction}

The following proposition relates the division operator to the
power series expansion $d_K:\Phi^{K\times C}X\to (\Phi^K X)\lr{\theta}$
that we introduced in Construction \ref{con:d_K}.

\begin{prop}\label{prop:d_K consistent}
  Let $(X,a,t)$ be an oriented $\elRO$-algebra.
  For every elementary abelian 2-group $K$ and all homogeneous elements $u\in X(K\times C)$,
  the power series expansion of $u/1$ is given by the formula
  \[  d_K(u/1)\ =\ {\sum}_{n\geq 0}\, i_1^*(\Gamma^n(u))/1 \cdot \theta^n\ . \] 
\end{prop}
\begin{proof}
  The class $u/1$ belongs to $\Psi^K X$,
  so the Laurent power series $d_K(u/1)$ is contained in the integral subring
  $(\Phi^K_*X)\dbr{\theta}$, by construction.
  We will show that the formula for $d_K(u/1)$
  holds in the ring $(\Phi^K_*X)\dbr{\theta}$ modulo $\theta^m$, for all $m\geq 0$.
  We argue by induction on $m$, starting with $m=0$, where there is
  nothing to show.
  Now we suppose that $m\geq 1$.
  By the inductive hypothesis, $d_K(\Gamma(u)/1)$ is congruent modulo $\theta^{m-1}$ to
  \[  {\sum}_{n\geq 0}\, i_1^*(\Gamma^n(\Gamma(u)))/1 \cdot \theta^n\ . \] 
  Hence
  \begin{align*}
    d_K(u/1)\ _{\eqref{eq:defining_relation_Gamma}}
    &= \ d_K( (p_1^*(i_1^*(u)) + e_2\cdot \Gamma(u))/1) \\
    &= \ i_1^*(u)/1\ + \ \theta\cdot d_K(\Gamma(u)/1) \\
    &\equiv \  i^*(u)/1 +  \theta\cdot  {\sum}_{n\geq 0}\, i_1^*(\Gamma^n(\Gamma(u)))/1 \cdot \theta^n    \text{\qquad modulo $\theta^m$}\\
    &= \  {\sum}_{n\geq 0}\, i_1^*(\Gamma^n(u))/1 \cdot \theta^n\
  \end{align*}
  This completes the inductive step, and hence the proof.  
\end{proof}

For equivariant bordism, the division operator 
$\Gamma:\cN^{K\times C}_k\to\cN^{K\times C}_{k+1}$
has a geometric interpretation that we now recall.
Variations and special cases of the following proposition
have been used in the literature on equivariant bordism,
so we make no claim to originality.
For example, the special case $K=1$ is treated in 
\cite[Lemma 25.3]{conner:differentiable-2nd}.

\begin{prop}\label{prop:geometric_Gamma}
  Let $K$ be an elementary abelian 2-group.
  Let $M$ be a $k$-dimensional smooth closed $(K\times C)$-manifold.
 The class $\Gamma[M]$ is represented by the $(k+1)$-manifold
  \[ \mR\times_\mZ M \ = \ (\mR\times M)/(x,m)\sim(x+1, (-1,1)\cdot m)\]
  with $(K\times C)$-action by
  \[ (\kappa,\kappa')\cdot [x,m]\ = \ [\kappa' \cdot x,(\kappa,1)\cdot m]\ . \]
\end{prop}

\begin{proof}
 The class $t_2\cdot ([M] +p_1^*(i_1^*[M]))$
  is represented by the $(K\times C)$-manifold
  $S^{p_2}\times(M\amalg p_1^*(i_1^*M))$ with the first projection as reference map to $S^{p_2}$.
  The class $a_2\cdot [\mR\times_\mZ M]$ is represented by the $(K\times C)$-manifold
  $\mR\times_\mZ M$ with the constant map with value 0 as reference map to $S^{p_2}$.
  Since all $a$-classes are regular, the map
  $\widetilde{\cN}_*^{K\times C}(S^{p_2})\to \Phi^{K\times C}_*\cN$
  is injective. Moreover, the fixed point map \eqref{eq:geometric fix iso}
  \[ \varphi_{K\times C}\ : \ \Phi^{K\times C}_*\cN\ \xra{\ \iso \ } \
    {\bigoplus}_{j\geq 0}\,\cN_{*-j}(G r_j^{K\times C,\perp}) \]
  is bijective.
  So it suffices to show that the two representatives for classes in $\widetilde{\cN}^{K\times C}_{k+1}(S^{p_2})$
  have the same fixed point invariants.
  
  The $(K\times C)$-fixed points of the mapping torus $\mR\times_{\mZ} M$ decompose into two parts:
  \[  (\mR\times_\mZ M)^{K\times C}\ \iso \ M^{K\times C}\cup M^{K\times 1} \]
  The first summand is embedded via
\[ M^{K\times C}\ \to \ \mR\times_\mZ M\ , \quad x \ \longmapsto \ [1/2,x] \ ;\]
  its normal bundle is isomorphic to the Whitney sum of the
  normal bundle of $M^{K\times C}$ inside $M$ and the trivial line bundle associated with
  the representation $p_2$.
  The second summand is embedded via
  \[ M^{K\times 1}\ \to \ \mR\times_\mZ M\ , \quad x \ \longmapsto \ [0,x]\ ;\]
  its normal bundle is isomorphic to the Whitney sum of
  the inflation along $p_1:K\times C\to K$ of
  the normal bundle of $M^{K\times 1}$ inside $M$, 
  and the trivial line bundle associated with the representation $p_2$.
  The $(K\times C)$-fixed points of $S^{p_2}\times(M\amalg p_1^*(i_1^*M))$ over $0\in S^{p_2}$
  are
  \[ \{0\}\times (M\amalg p_1^*(i_1^*M))^{K\times C}\ \iso \ M^{K\times C} \cup M^{K\times 1}\ ;  \]
  these are diffeomorphic fixed points with isomorphic normal bundles. This proves the claim.
\end{proof}

\begin{theorem}\label{thm:C-fix N:beta}\ 
  \begin{enumerate}[\em (i)]
  \item  In the oriented $\elRO$-algebra $(\cN,a,t)$, the classes $\beta_n$
    and $\zeta_n$ are related by the formula
    \[ \zeta_n \ = \ \beta_n + \beta_0 \beta_{n-1}  \]
    in $\Phi^C_{n+1}\cN$, for $n\geq 1$.
  \item
    The classes $\beta_n$ for $n\geq 0$ form a set of polynomial generators
    of the $C$-geometric fixed points of equivariant bordism as an algebra over $\cN_*$.
  \end{enumerate}
\end{theorem}
\begin{proof}
  (i)
  We write $P\in \cN^{C\times C}_2$ for the bordism class of
  the projective space $\mR P^2$ with $(C\times C)$-action by
  \[ (\kappa,\kappa')\cdot [x:y:z]\ = \ [\kappa' x:y:\kappa z]\ . \]
  By Proposition \ref{prop:geometric_Gamma} and induction on $n$,
  the class $\Gamma^n(P)$ in $\cN^{C\times C}_{n+2}$ is represented by
  \[  \mR\times_\mZ(\mR^{n-1}\times_{G_{n-1}} \mR P^2) \ \iso \ \mR^n\times_{G_n} \mR P^2\ ,\]
  where $G_n$ acts on $\mR P^2$ by
  \[  (s_1,\dots,s_n)\cdot [x:y:z] \ = \ [x:y:(-1)^{s_n} z]\    ,\]
  and where $C\times C$ acts on the balanced product by 
  \[ (\kappa,\kappa')\cdot [s_1,\dots,s_n,[x:y:z]]\ = \  [ \kappa s_1,s_2\dots,s_n,[\kappa' x:y:z]]\ .  \]
  Restricting the action along $i_1:C\to C\times C$ shows that 
  $i_1^*(\Gamma^n(P))$ is represented by $(W_{n+2},\tau)$,
  i.e., $i_1^*(\Gamma^n(P))=\zeta_{n+1}$.

  Now we claim that the relation
  \begin{equation} \label{eq:P_is_distinguished}
      P/1\ = \    t_\mu/a_\mu\cdot t_2/a_2 + t_1/a_1\cdot t_2/a_2  + t_1/a_1\cdot t_\mu/a_\mu
  \end{equation}
  holds in the group $\Phi^{C\times C}_2\cN$.
  To prove this, we verify that the two classes have the same image
  under the fixed point isomorphism $\varphi_{C\times C}$, defined in \eqref{eq:geometric fix iso},
  that records the fixed points and their normal data.
  If $\lambda$ is a character of a finite group, and $V$ a $G$-representation
  that does not contain $\lambda$, then $\lambda$ is an isolated fixed point
  in the projective space $P(V\oplus \lambda)$. Moreover, the normal $G$-representation
  at this fixed point is isomorphic to $V\tensor\lambda$.
  The given $(C\times C)$-action on $\mR P^2$ is that of the projective space of
  $p_2\oplus\mR\oplus p_1$. It has three isolated fixed points, with normal representations
  $p_1\oplus\mu$, $p_1\oplus p_2$,  and  $\mu\oplus p_2$, respectively.
  
  For a nontrivial $A$-character $\lambda$,
  the class $t_\lambda\in \widetilde{\cN}^A_1(S^\lambda)$ is represented by
  the identity of $S^\lambda$, which has an isolated fixed point over
  the non-basepoint, with normal representation $\lambda$.
  So the fixed point isomorphism \eqref{eq:geometric fix iso}
  takes the class $t_\lambda/a_\lambda\in\Phi^A_1\cN$ to the element of $\cN_0(G r_1^{A,\perp})$
  in the summand for $j=1$ represented by $\lambda$, considered as a point in $G r_1^{A,\perp}$.
  So the fixed point invariant of $P/1$ is
    \begin{align*}
    \varphi_{C\times C}(P/1)\
    &= \  [\mu] \cdot [p_2] + [p_1]\cdot[p_2] + [p_1]\cdot [\mu]\\
    &=\ \varphi_{C\times C}(t_\mu/a_\mu\cdot t_2/a_2 + t_1/a_1\cdot t_2/a_2 +
      t_1/a_1\cdot t_\mu/a_\mu)  \ . \end{align*}
  This proves the relation \eqref{eq:P_is_distinguished}.

  Now we apply the Laurent power series expansion $d_C:\Phi^{C\times C}\cN\to (\Phi^C\cN)\lr{\theta}$
  to \eqref{eq:P_is_distinguished}.
  The ring  $\Phi^{C\times C}\cN$ is an algebra over $\Phi^C \cN$
  via the inflation $p_1^*$, so $d_C(t_1/a_1)=d_C(p_1^*(t/a))=t/a=\beta_0$.
  Moreover, $d_C(t_2/a_2)=\theta^{-1}$. So
  \begin{align*}
    d_C(P/1)\
    &= \  d_C(t_\mu/a_\mu\cdot t_2/a_2 + t_1/a_1\cdot t_2/a_2 + t_1/a_1\cdot t_\mu/a_\mu)\\
    &= \ ( d_C(t_\mu/a_\mu)+\beta_0) \cdot \theta^{-1} + \beta_0\cdot d_C(t_\mu/a_\mu)  \
    =\ \sum_{n\geq 0} (\beta_{n+1} + \beta_0 \beta_n)\cdot \theta^n\ .
  \end{align*}
  Proposition \ref{prop:d_K consistent} yields the relation
  \[  d_C(P/1)    \ = \  {\sum}_{n\geq 0}\, i_1^*(\Gamma^n(P))/1 \cdot \theta^n
    \ = \  {\sum}_{n\geq 0}\, \zeta_{n+1}\cdot \theta^n\ .\]
  Comparing coefficients of $\theta^n$ yields 
  $\zeta_{n+1}=\beta_{n+1} + \beta_0 \beta_n$ for all $n\geq 0$, as claimed.

  (ii)
  The classes $t/a=\beta_0$ and $\zeta_n/1$ for $n\geq 1$
  are polynomial $\cN_*$-algebra generators of $\Phi^C\cN$ by Theorem \ref{thm:C-fix N}.
  Since $\beta_n$ is congruent to $\zeta_n$ modulo
  $(\beta_0,\dots,\beta_{n-1})$, by part (i), also the classes $\beta_n$ for $n\geq 0$
  form polynomial generators for $\Phi^C_* \cN$ as an $\cN_*$-algebra.
\end{proof}

We are now ready to prove the main theorem of this paper.

\begin{theorem}[Universal property of equivariant bordism]\label{thm:main}
  The oriented $\elRO$-algebra $\cN$ of equivariant bordism is
  an initial oriented $\elRO$-algebra.
\end{theorem}
\begin{proof}
  We need to show that there is a unique morphism
  of oriented $\elRO$-algebras from $(\cN,a,t)$
  to any given oriented $\elRO$-algebra $(Y,\bar a,\bar t)$.
  In accordance with the notation for pre-Euler and inverse Thom classes,
  we shall write $\bar\beta_n$ for the $\beta$-classes for the theory $(Y,\bar a,\bar t)$. 

  If $f:\cN\to Y$ is a morphism of oriented $\elRO$-algebras,
  then $f(1):\cN_*\to Y(1)_*$ must take the formal group law of $(\cN,a,t)$ to that of $(Y,\bar a,\bar t)$.
  The formal group of $(\cN,a,t)$ is initial by Theorem \ref{thm:bordism_properties} (ii),
  so $f(1)$ is uniquely determined by this property.
  For an elementary abelian 2-group $A$, 
  we make the geometric fixed point ring $\Phi^A_* Y$ into an $\cN_*$-algebra via
  the composite of $f(1)\colon \cN_*\to Y(1)_*$ and the inflation homomorphism
  $p_A^*:Y(1)_*\to\Phi^A_* Y$.  The morphism $\Phi^C f:\Phi^C\cN\to\Phi^C Y$ satisfies
  $(\Phi^C f)\circ p_C^*=p_C^*\circ f(1)$, so it is a morphism of graded $\cN_*$-algebras.
  Since the $\beta$-classes are defined intrinsically via the structure of
  oriented $\elRO$-algebras, the relation $(\Phi^C f)(\beta_n)=\bar\beta_n$ holds.
  By Theorem \ref{thm:C-fix N:beta} (ii),
  the classes $\beta_n$ generate $\Phi^C_*\cN$ as an $\cN_*$-algebra,
  so the morphism $\Phi^C f$ is uniquely determined.

  Now we let $A$ be any elementary abelian 2-group.
  Then for every nontrivial $A$-character $\lambda$, the relation
  $(\Phi^A f)\circ\lambda^*=\lambda^*\circ(\Phi^C f)$ holds.
  Since $\Phi^C f$ is uniquely determined and $\Phi^A \cN$ is generated
  as an $\cN_*$-algebra by the images of the homomorphisms $\lambda^*:\Phi^C\cN\to\Phi^A\cN$,
  by Theorem \ref{thm:bordism_properties} (iii), also the morphism $\Phi^A f$
  is uniquely determined.
  Since the maps $-/a_V:Y(A,k-V)\to \Phi^A_k Y$ are injective,
  the morphism $\Phi^A f$ uniquely determines the morphism of $I_A$-graded
  rings $f(A,\star):\cN(A,\star)\to Y(A,\star)$. This concludes the proof of uniqueness.

  To construct a morphism of oriented $\elRO$-algebras from $(\cN,a,t)$
  to $(Y,\bar a,\bar t)$, we turn the uniqueness argument around.  
  Since the $2$-torsion formal group law of the oriented $\elRO$-algebra $(\cN,a,t)$ is initial,
 there is a unique homomorphism of graded $\mF_2$-algebras $f^1\colon \cN_*\to Y(1)_*$
  that classifies the 2-torsion formal group law of $(Y,\bar a,\bar t)$.
  In other words, $f^1$ is the unique morphism that when applied
  to the coefficients of the 2-torsion formal group law of $(\cN,a,t)$
  yields the one of $(Y,\bar a,\bar t)$. 
  
  Again we make the geometric fixed point ring $\Phi^A Y$ into an $\cN$-algebra via
  the composite of $f^1\colon \cN_*\to Y(1)_*$ and the inflation homomorphism
  $p_A^*:Y(1)_*\to\Phi^A Y$. 
  By Theorem \ref{thm:C-fix N}, the geometric fixed points $\Phi^C_*\cN$
  are polynomial over $\cN_*$ on the classes $\beta_n$.
  So there is a unique $\cN_*$-algebra homomorphism
  \[ f^C\ : \ \Phi^C\cN \ \to \ \Phi^C Y\]
  that sends the class $\beta_n$ in $\Phi^C\cN$ to the class $\bar\beta_n$ in $\Phi^C Y$,
  for all $n\geq 0$.
  
  Now we let $\lambda$ be a nontrivial character of an elementary abelian 2-group $A$.
  Because $f^C\circ p_C^*=p_C^*\circ f^1$, the composite of
  \[ \cN_* \ \xra{\ p_C^* \ }\ \Phi^C \cN\ \xra{\ f^C\ } \  \Phi^C Y\ \xra{\ \lambda^*\ } \ \Phi^A Y\ \]
  equals  $p_A^*\circ f^1: \cN_*\to \Phi^A Y$.
  In particular, the composite $\lambda^*\circ f^C\circ p^*_C$ is independent of $\lambda$.
  By Theorem \ref{thm:bordism_properties} (iii), $\Phi^A \cN$ is the coproduct,
  in the category of commutative graded $\cN_*$-algebras,
  of copies of $\Phi^C \cN$, indexed by the set $A^\circ$ of nontrivial characters.
  So there is a unique morphism of graded rings
  \[ f^A \ : \ \Phi^A \cN \ \to \ \Phi^A Y \]
  such that $f^A\circ \lambda^*=\lambda^*\circ f^C$ for all nontrivial $A$-characters.

  We will now show that the collection of morphisms $f^\bullet=\{f^A:\Phi^A \cN\to\Phi^A Y\}_A$
  satisfies the hypotheses of Theorem \ref{thm:fix_criterion}.
  To show that the morphisms are compatible with inflations,
  we let $\alpha:B\to A$ be any epimorphism between elementary abelian 2-groups.
  Then
  \[ f^B\circ \alpha^*\circ \lambda^* \ = \
    f^B\circ (\lambda\alpha)^*\ =\
    (\lambda\alpha)^*\circ f^C \ =\ \alpha^*\circ \lambda^*\circ f^C \ =\
    \alpha^*\circ f^A\circ \lambda^*    \]
  for every nontrivial $A$-character $\lambda$.
  By Theorem \ref{thm:bordism_properties} (iii),
  the images of the homomorphisms $\lambda^*:\Phi^C \cN\to \Phi^A\cN$
  generate $\Phi^A \cN$ as an $\cN_*$-algebra,
  so this proves that $f^B\circ \alpha^*=\alpha^*\circ f^A$.
  
  Condition (a) of Theorem \ref{thm:fix_criterion}
  holds because $f^C(t/a)=f^C(\beta_0)=\bar\beta_0=\bar t/\bar a$.
  Verifying condition (b) is the most involved part of this proof,
  and we establish it through a sequence of claims.

  Claim 1: The following diagram commutes:
    \[ \xymatrix@C=15mm{
        \Phi^C \cN\ar[d]_{f^C}\ar[r]^-{d_1} &
        \cN_*\lr{\theta}\ar[d]^{f^1\lr{\theta}}\\
        \Phi^C Y\ar[r]_-{d_1}   & Y(1)\lr{\theta}  } \]
  Theorem \ref{thm:determine_betas} (i) provides the relation
  \[ F_\cN(\theta,\xi)\cdot\sum_{n\geq 0} d_1(\beta_n)\cdot \xi^n \ = \ 1  \]
  in the ring $\cN\lr{\theta}\dbr{\xi}$.
  We apply the homomorphism $f^1\lr{\theta}\dbr{\xi}$
  to this equation and exploit that
  the pushforward of the formal group law $F_\cN$ along $f^1:\cN_*\to Y(1)_*$ is $F_Y$. This yields
  \begin{align*}
    F_Y(\theta,\xi)\cdot\sum_{n\geq 0} f^1\lr{\theta}(d_1(\beta_n))\cdot \xi^n \
    &= \  1 \
    = \ F_Y(\theta,\xi)\cdot\sum_{n\geq 0} d_1(\bar\beta_n)\cdot \xi^n \
  \end{align*}
  in the ring $Y(1)\lr{\theta}\dbr{\xi}$.
  The second equality is Theorem \ref{thm:determine_betas} (i) for $(Y,\bar a,\bar t)$.
  Since multiplicative inverses are unique, the two infinite sums on both sides
  must be equal. 
  Comparing coefficients of $\xi^n$ yields
  $f^1\lr{\theta}(d_1(\beta_n))=d_1(\bar\beta_n)=d_1(f^C(\beta_n))$.
  Since the classes $\beta_n$ generate $\Phi^C\cN$ as an $\cN_*$-algebra,
  this proves that the diagram of Claim 1 indeed commutes.\smallskip

  Claim 2: The following diagram commutes:
    \[ \xymatrix@C=15mm{
        \Phi^C \cN\ar[d]_{f^C}\ar[r]^-{\mu^*} &
        \Phi^{C\times C} \cN \ar[r]^-{d_C} &
        (\Phi^C \cN)\lr{\theta}\ar[d]^{f^C\lr{\theta}}\\
        \Phi^C Y\ar[r]_-{\mu^*} & \Phi^{C\times C}Y\ar[r]_-{d_C}        
        & (\Phi^C Y)\lr{\theta}
      } \]
    Here, as before, $\mu:C\times C\to C$ is the multiplication.
  Theorem \ref{thm:determine_betas} (ii) provides the relation
  \[ \sum_{n\geq 0} d_C(\mu^*(\beta_n)) \cdot \xi^n\ =\
    \sum_{n\geq 0}\, \beta_n\cdot F_{\cN}(\theta,\xi)^n \]
  in the ring $(\Phi^C\cN)\dbr{\theta,\xi}$.
  We apply the homomorphism $f^C\dbr{\theta,\xi}$
  to this equation and exploit that
  the pushforward of the formal group law $F_\cN$ along
  $f^1:\cN_*\to Y(1)_*$ is $F_Y$. This yields
  \begin{align*}
    \sum_{n\geq 0} f^C\dbr{\theta}(d_C(\mu^*(\beta_n))) \cdot \xi^n\
    &=\ \sum_{n\geq 0}\, f^C(\beta_n)\cdot F_Y(\theta,\xi)^n \\
    &=\ \sum_{n\geq 0}\, \bar\beta_n\cdot F_Y(\theta,\xi)^n
      =\ \sum_{n\geq 0}\, d_C(\mu^*(\bar\beta_n))\cdot \xi^n \ .
  \end{align*}
  The third equality is Theorem \ref{thm:determine_betas} (ii) for $(Y,\bar a,\bar t)$.
  Comparing coefficients of $\xi^n$ yields
  \[ f^C\dbr{\theta}(d_C(\mu^*(\beta_n))) \ = \ d_C(\mu^*(\bar\beta_n))\
    = \ d_C(\mu^*(f^C(\beta_n)))\ .\]
  Because the classes $\beta_n$ generate $\Phi^C_*\cN$ as an $\cN_*$-algebra,
  this proves that the diagram of Claim 2 indeed commutes.\smallskip

  Claim 3:  
  For every elementary abelian 2-group $K$
  and every nontrivial character $\lambda:K\times C\to C$,
  the following diagram commutes:
    \[ \xymatrix@C=15mm{
        \Phi^C \cN\ar[d]_{f^C}\ar[r]^-{\lambda^*} &
        \Phi^{K\times C} \cN \ar[r]^-{d_K} &
        (\Phi^K \cN)\lr{\theta}\ar[d]^{f^K\lr{\theta}}\\
        \Phi^C Y\ar[r]_-{\lambda^*} & \Phi^{K\times C}Y\ar[r]_-{d_K}        
        & (\Phi^K Y)\lr{\theta}
    } \]
  We distinguish three cases, each based on a different kind of naturality argument.
  Firstly, if $\lambda=p_2$ is the projection to the second factor, then
  $d_K\circ \lambda^*=p_K^*\lr{\theta}\circ d_1$ by naturality of the power series
  expansions $d_K$ in the group $K$. So
  \begin{align*}
    f^K\lr{\theta}\circ d_K\circ \lambda^*
    &= \ f^K\lr{\theta}\circ p_K^*\lr{\theta} \circ d_1 \
      = \  p_K^*\lr{\theta}\circ f^1\lr{\theta} \circ d_1 \\
    _{\text{Claim 1}} &= \ p_K^*\lr{\theta}\circ d_1\circ  f^C \
                 = \ d_K\circ\lambda^*\circ f^C\ .
  \end{align*}
  
  Secondly, we suppose that $\lambda=\nu\circ p_1$ for some nontrivial $K$-character
  $\nu:K\to C$.
  By design, the composite $d_K\circ p_1^*:\Phi^K \cN \to \Phi^{K\times C} \cN$
  is the inclusion $\iota:\Phi^K \cN \to(\Phi^K \cN)\lr{\theta}$
  as the constant Laurent power series.
  So 
  \[  d_K\circ \lambda^*\ = \ d_K\circ p_1^*\circ \nu^*\ = \ \iota\circ\nu^*\ ,\]
  and hence
  \[ f^K\lr{\theta}\circ d_K\circ \lambda^*\ = \
    f^K\lr{\theta}\circ \iota\circ \nu^*
    \ = \ \iota\circ f^K \circ \nu^*
    \ = \ \iota\circ \nu^*\circ f^C 
    \ = \ d_K\circ \lambda^*\circ f^C\ .    \]
  
  The third and final case is when $\lambda=\mu\circ(\nu\times C)$
  for a nontrivial $K$-character $\nu$. Then
  \[  d_K\circ \lambda^*\ = \ d_K\circ (\nu\times C)^*\circ\mu^*\ = \
    \nu^*\lr{\theta}\circ d_C\circ \mu^* \]
  by naturality of the power series expansions $d_K$ in the group $K$. So
  \begin{align*}
    f^K\lr{\theta}\circ d_K\circ \lambda^*
      &= \ f^K\lr{\theta}\circ \nu^*\lr{\theta}\circ d_C\circ\mu^*\
     = \  \nu^*\lr{\theta}\circ f^C\lr{\theta}\circ d_C\circ\mu^*\\
    _{\text{Claim 2}}&= \ \nu^*\lr{\theta}\circ d_C\circ\mu^*\circ f^C\
                       = \ d_K\circ\lambda^*\circ f^C\ .
  \end{align*}
    
  Now we can prove condition (b) of Theorem \ref{thm:fix_criterion}.
  For any nontrivial character $\lambda:K\times C\to C$,
  Claim 3 and the fact that $f^{K\times C}\circ\lambda^*=\lambda^*\circ f^C$
  show that the two homomorphisms
  \[ f^K\lr{\theta}\circ  d_K \ ,\ d_K\circ f^{K\times C}\ : \
    \Phi^{K\times C}\cN\ \to\  (\Phi^KY)\lr{\theta}\]
  agree after precomposition with $\lambda^*:\Phi^C \cN\to\Phi^{K\times C}\cN$.
  By Theorem \ref{thm:bordism_properties} (iii),
  the images of these homomorphisms $\lambda^*$
  generate $\Phi^{K\times C} \cN$ as an $\cN_*$-algebra,
  so this proves the desired relation
  $f^K\lr{\theta}\circ  d_K = d_K\circ f^{K\times C}$.
\end{proof}

\section{Invertibly oriented \texorpdfstring{$\elRO$}{el}-algebras and global 2-torsion group laws}
\label{sec:global-2-torsion group}

In this section we explain the relationship between the oriented $\elRO$-algebras discussed
in this paper and the global 2-torsion group laws
in the sense of the first author's paper \cite{hausmann:group_law}.
The bottom line is that global 2-torsion group laws `are'
oriented $\elRO$-algebras whose inverse Thom class is invertible.
Because inverse Thom classes can be universally inverted as described in
Construction \ref{con:t-localization}, this exhibits the category of
global 2-torsion group laws as a reflective subcategory
of the category of oriented $\elRO$-algebras.

\begin{defn}[Invertible inverse Thom classes]\label{def:invertible}
  An inverse Thom class $t$ of an orientable $\elRO$-algebra $X$ is {\em invertible}
  if for every nontrivial character $\lambda$
  of an elementary abelian 2-group $A$ and all $m\in I_A$, the multiplication map
  \[  t_\lambda \cdot - \ : \ X(A,m)\ \to \ X(A,m+1- \lambda) \]
  is an isomorphism.
  An {\em invertibly oriented} $\elRO$-algebra is an oriented $\elRO$-algebra
  whose inverse Thom class is invertible.
\end{defn}

An invertible inverse Thom class essentially collapses the $I_A$-gradings
to integer gradings: for every elementary abelian 2-group $A$,
the $I_A$-graded ring $X(A,\star)$ is a polynomial ring over its $\mZ$-graded
subring in the variables $t_\lambda$ for $\lambda\in A^\circ$:
\[ X(A,\star)\ = \ X(A)_*[t_\lambda\colon \lambda\in A^\circ]\ .\]
In the presence of an invertible inverse Thom class $t$, one can therefore translate
the structure and properties of an orientable $\elRO$-algebra into
structure and properties of the $\el_2$-algebra made up by the integer graded subrings.
For orientable $\elRO$-algebras arising from global ring spectra,
invertibility in the sense of Definition \ref{def:invertible} is
equivalent to invertibility in the $R O(C)$-graded homotopy ring,
see Proposition \ref{prop:RO-invertible_is_invertible}.

We recall from \cite[Definition 4.1]{hausmann:group_law}
that a {\em global 2-torsion group law}
is an $\el_2$-$\mF_2$-algebra $\cG$ equipped with a {\em coordinate},
i.e., a class $e\in \cG(C)$ such that for every nontrivial character
$\lambda$ of an elementary abelian 2-group $A$, the sequence
\[ 0  \ \to\ \cG(A)_{*+1}\ \xra{ e_\lambda\cdot -}\ \cG(A)_*\ \xra{\res^A_K}\ \cG(K)_*\ \to \ 0 \]
is exact, where $e_\lambda=\lambda^*(e)$ and $K$ is the kernel of $\lambda$.

As described in Construction \ref{con:t-localization},
we can turn an oriented $\elRO$-algebra $(X,a,t)$
into an $\el_2$-$\mF_2$-algebra $t^{-1}X$ by inverting all $t$-classes.
When $A=1$ is a trivial group, we have $(t^{-1} X)(1)_*=X(1)_*$.

\begin{theorem}\
  \begin{enumerate}[\em (i)]
  \item 
  For every oriented $\elRO$-algebra $(X,a,t)$,
  the $\el_2$-algebra $t^{-1}X$ is a global 2-torsion group law
  with respect to the coordinate $e=a/t$.  
\item
  The functor
  \begin{align*}
    t^{-1} \ : \ \text{\em (oriented $\elRO$-algebras)}\ &\to \ \text{\em (global 2-torsion group laws)}\\   
    (X,a,t)\ &\longmapsto \ (t^{-1}X,a/t)
  \end{align*}
  has a fully faithful right adjoint.
  The essential image of the right adjoint
    is the subcategory of invertibly oriented $\elRO$-algebras.
  \end{enumerate}
\end{theorem}
\begin{proof}
  (i) We have
  $e_\lambda = \lambda^*(a/t) = a_\lambda/t_\lambda$ in  $(t^{-1}X)(A)_{-1}$.
  Because localization is exact,  the short exact sequences
  \[ 0 \ \to \ X(A,m+1)\ \xra{a_\lambda\cdot -}\ X(A,m+1-\lambda)\
    \xra{\res^A_K}\ X(K,\res^A_K(m))\ \to \ 0 \]
  for varying $m\in I_A$ become a short exact sequence
  \[  0 \ \to \ (t^{-1}X)(A)_{|m|+1}\ \xra{e_\lambda\cdot -}\ (t^{-1}X)(A)_{|m|}\
    \xra{\res^A_K}\  (t^{-1}X)(K)_{|m|}\ \to \ 0 \ .   \]
  This verifies the defining exactness property of a global 2-torsion group law.
  
  (ii) The right adjoint
  \[  \psi \ : \  \text{(global 2-torsion group laws)}\ \to\ \text{(oriented $\elRO$-algebras)} 
 \]
 to $t^{-1}$ is obtained as follows.
 Given any global 2-torsion group law $(\cG,e)$, we define an $\elRO$-algebra $\cG[t]$
  at $A$ by
  \[ \cG[t](A,\star)\ = \ \cG(A)[t_\lambda:\lambda\in A^\circ]\ . \]
  Here each $t_\lambda$ is a polynomial variable in degree $(A,1-\lambda)$,
  for each nontrivial $A$-character $\lambda$.
  If $\alpha:B\to A$ is a homomorphism between elementary abelian 2-groups,
  we define $\beta^*:X(A,\star)\to Y(B,\star)$ as the given restriction homomorphism
  $\beta^*:\cG(A)\to \cG(B)$ on coefficient rings, and by setting
  \[ \beta^*(t_\lambda)\ = \
    \begin{cases}
      t_{\lambda\beta} & \text{ if $\lambda\beta$ is nontrivial, and}\\
     \  1 & \text{ if $\lambda\beta$ is trivial.}
    \end{cases}
  \]
  We omit the straightforward verification that these data define an $\elRO$-algebra.
  By design, the polynomial generator $t$ in $\cG[t](C,\star)= \cG(C)[t]$
  is an invertible inverse Thom class of this $\elRO$-algebra.
  Moreover, the exactness property of the orientation $e$ shows that the class
  \[ a\ =\ e\cdot t\ \in \ \cG(C)[t]_{-\sigma}\ = \ \cG[t](C,-\sigma) \]
  is a pre-Euler class, so $\psi(\cG,e)=(\cG[t],e t,t)$ is an invertibly oriented $\elRO$-algebra.

  The unit of the adjunction $\eta_X:(X,a,t)\to \psi(t^{-1}X,a/t)=((t^{-1}X)[t],(a/t)\cdot t, t)$
  is defined at an elementary abelian 2-group $A$ by
  \[ \eta_X(A,k-V)(x) \ = \ (x/t_V)\cdot t_V \ ,\]
  where $k\in\mZ$, and $V$ is an $A$-representation with trivial fixed points.
  We omit the straightforward verification that this indeed defines
  a morphism of oriented $\elRO$-algebras.
  Moreover, if the inverse Thom class $t$ is invertible, then $\eta_X$ is an isomorphism.

  Since the orientation of $\psi(\cG,e)$ is invertible, localization away from $t$
  has no effect on the underlying rings, and the restriction of $\cG[t]$
  to the integer graded subrings returns the original $\el_2$-algebra $\cG$,
  with its original coordinate $(e t)/t=e$.  
  More formally: the functor $\psi$ is right inverse to the functor $t^{-1}$,
  and a natural isomorphism $\cG\to t^{-1}(\psi(\cG,e))$ is given by
  the identifications of the rings $\cG(A)$ as the integer graded subrings of $\cG(A)[t]$.
  We omit the verification that the inverse isomorphism
  $\epsilon_\cG:t^{-1}(\psi(\cG,e))\to \cG$
  and the natural transformation $\eta:\Id\to \psi\circ t^{-1}$
  satisfy the triangle equalities of an adjunction.
  In any adjunction for which the counit is a natural isomorphism,
  the right adjoint is fully faithful.
  
  The right adjoint $\psi$ takes values in invertibly oriented $\elRO$-algebras.
  And if $(X,a,t)$ is an invertibly oriented $\elRO$-algebra,
  then the adjunction unit $(X,a,t)\to \psi(t^{-1} X,a/t)$ is an isomorphism
  at all integer graded subrings, and hence an isomorphism of $\elRO$-algebras.
  Thus $(X,a,t)$ is in the essential image of the right adjoint.
\end{proof}

\begin{rk}
  As explained by the first author in \cite{hausmann:group_law}, 
  every global 2-torsion group law has a 2-torsion formal group law
  over its underlying ring. Every oriented $\elRO$-algebra also
  comes with a formal group law over its underlying ring,
  as explained in Theorem \ref{thm:fgl}.
  Comparing definitions reveals that for every oriented $\elRO$-algebra $(X,a,t)$,
  the formal group law of $(t^{-1}X, a/t)$ defined in \cite{hausmann:group_law}
  is the one that we introduced in Theorem \ref{thm:fgl}.

  The formal group law construction from a global 2-torsion group
  law also admits a fully faithful right adjoint;
  the right adjoint embeds the category of 2-torsion formal group laws as the full subcategory of
  {\em complete} global 2-torsion group laws, compare \cite[Example 4.7]{hausmann:group_law}.
  Altogether, the category of 2-torsion formal group laws embeds fully faithfully
  into the category of oriented $\elRO$-algebras,
  as the ones that are both invertibly oriented and complete.
  The following diagram of categories and functors summarizes the situation:
  \[ \xymatrix@C=8mm{
      \text{(2-torsion formal group laws)}\ar@<1.2ex>@{^(->}[d] \ar[r]^-\psi_-\iso
      & \text{(complete invertibly oriented $\elRO$-algebras)}\ar@<1.2ex>@{^(->}[d]  \\
      \text{(global 2-torsion group laws)}\ar[r]_-\iso^-\psi\ar@<1.2ex>[u]_{\dashv\ }
      & \text{(invertibly oriented $\elRO$-algebras)} \ar@<1.2ex>@{^(->}[d]\ar@<1.2ex>[u]_{\dashv\ }\\
      &  \text{(oriented $\elRO$-algebras)}\ar@<1.2ex>[u]_{\dashv\ }    } \]
\end{rk}

In Definition \ref{def:invertible} we defined when an inverse Thom class of an 
$\elRO$-algebra is invertible. For global ring spectra, invertibility is
equivalent to invertibility in the $R O(C)$-graded homotopy ring:

\begin{prop}\label{prop:RO-invertible_is_invertible}
  For every oriented global ring spectrum $(E,t)$, the following conditions are equivalent.
  \begin{enumerate}[\em (a)]
  \item The class $t$ is an $R O(C)$-graded unit,
    i.e., there exists a class $s\in \pi_{\sigma-1}^C(E)$ such that $s\cdot t=1$.
  \item The class $t$ is an invertible as an inverse Thom class of
    the orientable $\elRO$-algebra $(E^\sharp,a)$.
  \end{enumerate}
\end{prop}
\begin{proof}
  (a)$\Longrightarrow$(b) 
  We let $\lambda:A\to C$ be any nontrivial character 
  of an elementary abelian 2-group.
  Then $\lambda^*(s)\in E_{\lambda-1}^A=\pi_{\lambda-1}^A(E)$ is a homotopy class such that
  \[ \lambda^*(s)\cdot t_\lambda \ = \ \lambda^*(s\cdot t)\ = \ 1\ . \]
  We write $n\in I_A$ as $n=k-V$ for $k\in \mZ$ and some $A$-representation $V$
  with trivial fixed points. Then the multiplication map
  \[ -\cdot \lambda^*(s) \ : \  E_{k+1}^A(S^{V\oplus\lambda})\ \to \ E_k^A(S^V) \]
  is inverse to
  \[ -\cdot t_\lambda \ : \ E^\sharp(A,n)=E_k^A(S^V)\ \to \
    E_{k+1}^A(S^{V\oplus\lambda}) = E^\sharp(A,n+1- \lambda) \ .\]
  In particular, multiplication by $t_\lambda$ is bijective.

  (b)$\Longrightarrow$(a) We consider the commutative diagram with exact rows:
  \[ \xymatrix{
      \pi_0^C(E)\ar[r]^-{\res^C_1}\ar[d]_{-\cdot t}&
      \pi_0(E)\ar[r]^-{\tr_1^C}\ar@{=}[d]&
      \pi_{\sigma-1}^C(E)\ar[r]^-{a \cdot }\ar[d]_{-\cdot t}&
      \pi_{-1}^C(E)\ar[r]^-{\res^C_1}\ar[d]_{-\cdot t}&
      \pi_{-1}(E)\ar@{=}[d]\\
      \pi_{1-\sigma}^C(E)\ar[r]_-{\res^C_1} &
      \pi_0(E)\ar[r]_-{\tr_1^C} &
      \pi_0^C(E)\ar[r]_-{a\cdot } &
      \pi_{-\sigma}^C(E)\ar[r]_-{\res^C_1} &   \pi_{-1}(E)
    } \]
  Of the three maps labeled as multiplication by $t$, the outer two
  lie in the range graded by the submonoid $I_C$ of $R O(C)$
  that is encoded by the $\elRO$-algebra $E^\sharp$. Hence those are isomorphisms
  by the invertibility hypothesis (a). By the five lemma, the middle map
  is thus an isomorphism, too. So the multiplicative unit $1$ in $\pi_0^C(E)$
  has an inverse, which is the desired $R O(C)$-graded inverse of $t$.
\end{proof}

\begin{construction}[Stable equivariant bordism]\label{con:stable_bordism}
  For a compact Lie group $G$, {\em stable equivariant bordism}
  is a $G$-equivariant homology theory $\mathfrak N^{G:S}_*$
  first considered by Br{\"o}cker and Hook \cite[\S 2]{broecker-hook}.
  It is a specific localization of $G$-equivariant bordism
  at geometrically defined classes
  associated to $G$-representations.
  If we restrict attention to elementary abelian 2-groups $A$,
  then the localization that defines stable equivariant bordism
  is exactly inverting the inverse Thom classes for the
  oriented $\elRO$-algebra $\cN$ of equivariant bordism:
  \[ \mathfrak N^{A:S}_* \ =\ (t^{-1}\cN)(A)_* \]
\end{construction}

Equivariant bordism $\cN$ is an initial oriented $\elRO$-algebra
by Theorem \ref{thm:main};
and the localization morphism $\cN\to t^{-1}\cN$
is initial among morphisms of oriented $\elRO$-algebras that send the unique orientation
of $\cN$ to an invertible inverse Thom class.
Together, these universal properties immediately yield the following corollary
about stable equivariant bordism:

\begin{cor}[Universal property of stable equivariant bordism]\label{cor:stable bordism} 
  Stable equivariant bordism  $(t^{-1}\cN,a,t)$
  is an initial invertibly oriented $\elRO$-algebra.
  The global 2-torsion group law of stable equivariant bordism 
  is an initial global 2-torsion group law.
\end{cor}

Combining Corollary \ref{cor:stable bordism} with the main result
of Br{\"o}cker and Hook \cite{broecker-hook}
yields an independent proof of a recent result of the first author,
namely that the global 2-torsion group law carried by the global Thom spectrum $\bMO$ is initial,
see \cite[Theorem D]{hausmann:group_law}.
Indeed, Theorem 4.1 of \cite{broecker-hook} says that
for every compact Lie group $G$, stable equivariant bordism is isomorphic
to the equivariant homology theory represented by the global Thom spectrum $\bMO$
of \cite[Example 6.1.7]{schwede:global}; see also \cite[Remark 6.2.38]{schwede:global}
for a different proof.
The underlying $G$-spectra of the global object $\bMO$ are the real analogues of
tom Dieck's `homotopical equivariant bordism' spectra \cite{tomDieck-bordism},
and they have been much studied since the 70's.
So via \cite[Theorem 4.1]{broecker-hook}, our Corollary \ref{cor:stable bordism}
becomes the statement that the global 2-torsion group law associated to $\bMO$
is initial; in other words, we recover Theorem D of \cite{hausmann:group_law}.
We want to emphasize that this earlier result of the first author
was a key motivation for the present project,
and some of our present arguments are inspired by arguments from \cite{hausmann:group_law}.
Our proof here differs from the one in \cite{hausmann:group_law}
in that we make no use of equivariant formal groups.

\begin{eg}[Mod 2 global Borel homology]
  The forgetful functor from the global stable homotopy category to the
  non-equivariant stable homotopy category has a lax symmetric monoidal right adjoint,
  the {\em global Borel} functor $b:\SH\to\GH$,
  see Theorem 4.5.1 and Construction 4.5.21 of \cite{schwede:global}.
  For every spectrum $R$ and every compact Lie group $G$,
  the underlying $G$-spectrum of the global Borel spectrum $b R$ represents
  $G$-equivariant Borel $R$-cohomology, i.e., $R$-cohomology of the homotopy orbit
  construction, thence the name. Because the Borel functor is lax symmetric monoidal, it takes
  non-equivariant homotopy ring spectra to global ring spectra.

  In particular, the global Borel spectrum $b(H\mF_2)$
  of the non-equivariant mod 2 Eilenberg-MacLane spectrum
  is a global ring spectrum. The adjunction unit
  $\eta:H\underline\mF_2\to b(H\mF_2)$ from the mod 2 global Eilenberg-MacLane spectrum
  is a morphism of global ring spectra.
  Since $H\underline\mF_2$ is orientable, so is $b(H\mF_2)$.
  The group $\pi_1^C(b(H\mF_2))$ is isomorphic to $H^{-1}(B C,\mF_2)$,
  and hence trivial. So the orientable global ring spectrum $b(H\mF_2)$
  has a unique inverse Thom class, which is necessarily the image of
  the unique inverse Thom class of $H\underline\mF_2$ under
  the morphism $\eta:H\underline\mF_2\to b(H\mF_2)$.
  Since the unique inverse Thom class of $H\underline\mF_2$ is additive,
  the unique inverse Thom class of $b(H\mF_2)$ is additive, too.

  All real vector bundles are orientable in mod 2 cohomology.
  In particular, there is a Thom class in $H^1(S^\sigma\sm_C EC_+;\mF_2)$
  for the line bundle over $B C$ associated to the sign representation.
  By the Thom isomorphism,
  the image of this Thom class under the isomorphism 
  \[ H^1(S^\sigma\sm_C EC_+;\mF_2) \ \iso \   \pi_{\sigma-1}^C( b(H\mF_2)) \]
  is an $R O(C)$-graded inverse to the unique inverse Thom class.
  So the unique inverse Thom class of $b(H\mF_2)$ is invertible,
  by Proposition \ref{prop:RO-invertible_is_invertible}.
\end{eg}

\begin{theorem}[Universal property of mod 2 Borel homology]\label{thm:Borel initial}
  Let $Y$ be an orientable $\elRO$-algebra.
  Evaluation at the unique inverse Thom class of the global Borel spectrum $b(H\mF_2)$ 
  is a bijection between the set of morphisms of orientable $\elRO$-algebras
  from $b(H\mF_2)^\sharp$ to $Y$, and the set of orientations of $Y$
  that are both additive and invertible.
\end{theorem}  
\begin{proof}
  We claim that the morphism of oriented $\elRO$-algebras
  \[  \eta^\sharp \ : \  H = ((H\underline\mF_2)^\sharp,a,t) \ \to \
    (b(H\mF_2)^\sharp,a,t)  \]
  exhibits the target as the localization of
  the source away from its unique inverse Thom class;
  more precisely: the induced morphism of global 2-torsion group laws $t^{-1}(\eta^\sharp)$
  obtained by localization away from $t$ is an isomorphism.
  To prove this, we write $b(H\mF_2)^\natural$ for the $\el_2$-algebra made from
  the integer graded homotopy rings of the global Borel spectrum $b(H\mF_2)$.
  Because $t$ acts invertibly in the global Borel theory,
  we must show that the unique extension
  \[  \tilde\eta\ : \ t^{-1}H\ \to \ b(H\mF_2)^\natural  \]
  of $\eta^\sharp$ is an isomorphism of $\el_2$-algebras.
  To this end, it suffices to show that $\tilde\eta$
  is an isomorphism in integer gradings
  for the elementary abelian 2-groups $C^n$ for all $n\geq 0$.
  The $I_A$-graded ring $H(A,\star)$ is generated
  as an $\mF_2$-algebra by the classes $a_\lambda$
  and $t_\lambda$ for all nontrivial $A$-characters $\lambda$,
  by Proposition \ref{prop:generation}.  
  So the integer graded ring $(t^{-1}H)(A)_*$ is generated by the classes
  $e_\lambda=a_\lambda/t_\lambda$ for all nontrivial $A$-characters $\lambda$.
  If three such characters satisfy a multiplicative relation $\alpha\cdot\beta\cdot\gamma=1$,
  then additivity of the orientation yields the relation
  \[  a_\alpha t_\beta t_\gamma\ +\ t_\alpha a_\beta t_\gamma\ +\ t_\alpha t_\beta a_\gamma \ = \ 0 \ .     \]
  Dividing by $t_\alpha t_\beta t_\gamma$ yields the relation $e_\alpha+e_\beta+e_\gamma=0$
  in $(t^{-1}H)(A)_{-1}$.
  In particular, for all $n\geq 0$, the integer graded ring $(t^{-1}H)(C^n)$
  is generated by the Euler classes $e_1,\dots,e_n$ of the
  projections $p_i:C^n\to C$ to the $n$ factors.

  On the other hand, $b(H\mF_2)^\natural(C^n)$ is the mod 2 cohomology ring of $B(C^n)$,
  which is an $\mF_2$-polynomial algebra on the corresponding Euler classes
  $e_1,\dots,e_n$ of the projections.
  So the map $\tilde\eta(C^n):( t^{-1}H)(C^n)\to b(H\mF_2)^\natural(C^n)$ is
  surjective. Since the Euler classes  $e_1,\dots,e_n$
  do not satisfy any nontrivial polynomial relations  in the target,
  the Euler classes do not satisfy any nontrivial polynomial relations in the source either.
  So both rings are $\mF_2$-polynomial algebras on the Euler classes
  $e_1,\dots,e_n$, and the map is an isomorphism.
  
  Now we can conclude the argument.
  The source of the morphism $\eta^\sharp$ is a Bredon $\elRO$-algebra 
  in the sense of Definition \ref{def:Bredon};
  so by Theorem \ref{thm:Bredon initial}, it is initial among
  oriented $\elRO$-algebras equipped with an additive orientation.
  Since the target of $\eta^\sharp$ is a localization of the source
  away from the inverse Thom class $t$, the target
  is an initial oriented $\elRO$-algebra whose orientation is
  both additive and invertible.
\end{proof}

\section{Epilogue: reproving classical results on equivariant bordism} \label{sec:classical}

In this final section we describe the relationship between the framework we develop in this paper
and previous work on equivariant bordism rings,
specifically that by Conner--Floyd \cite{conner-floyd}, Alexander \cite{alexander},
and Firsching \cite{firsching}.
We show that these classical results can be readily deduced from our formalism,
and in fact generalize to a large class of oriented $\elRO$-algebras.

\subsection{The Conner--Floyd exact sequence}

We start with a generalization of the Conner--Floyd
exact sequence \cite[Theorem 28.1]{conner-floyd} for bordism of involutions.
For this, we let $(X,a,t)$ be an oriented $\elRO$-algebra, and
$A$ an elementary abelian 2-group. As we used many times throughout this paper,
the map $-/1\ \colon \ X(A)_* \to \Phi^A_* X$ is injective.
For $A=C$, it is furthermore straightforward to describe the structure of the cokernel.

\begin{prop} \label{prop:exact_sequence}
  For every oriented $\elRO$-algebra $(X,a,t)$ the cokernel of the monomorphism
  $-/1:X(C)_*\to \Phi^C X$
  is free as an $X(1)_*$-module with basis the classes $(t/a)^n$ for $n\geq 1$.
\end{prop}
  \begin{proof}
  The fundamental exact sequence and induction on $n$ shows that
  $X(C,*-n\sigma)$ is the direct sum of the image of $\cdot a^n:X(C)_*\to X(C,*-n\sigma)$
  and a free graded $X(1)_*$-module spanned by $a^{n-i} t^i$
  for $i=1,\dots, n$. Since $\Phi^C_* X$ is the colimit of these groups
  along multiplication by $a$, the claim follows by passage to colimit in $n$.
\end{proof}

Proposition \ref{prop:exact_sequence} also shows that
if $\Phi^C X$ is free as a graded module over $X(1)_*$,
then $X(C)_*$ is projective as a graded  $X(1)_*$-module.
If, in addition, the $\mF_2$-algebra $X(1)_*$ is graded-connected, then this implies that 
$X(C)_*$ is even free as a graded  $X(1)_*$-module.

\subsection{Alexander's theorem}
It was already observed by Conner and Floyd
as a consequence of their exact sequence \cite[Theorem 28.1]{conner-floyd}
that the $C$-equivariant bordism ring is free as a graded module
over the non-equivariant bordism ring.
Alexander provided an explicit $\cN_*$-basis of $\cN^C_*$ in
\cite[Theorem 1.2]{alexander}.
We generalize Alexander's result to certain kinds of oriented $\elRO$-algebras,
with a different, and arguably simpler, proof.

We recall the division operator $\Gamma:X(C)_k\to X(C)_{k+1}$
from construction \eqref{eq:define_Gamma},
which is characterized by the equation $a\cdot \Gamma(x)= t\cdot x+ \epsilon^*(x)$.
Here $\epsilon:C\to C$ is the trivial homomorphism, so that $\epsilon^*$
is the composition of restriction to the trivial group, followed by inflation.

\begin{prop}\label{prop:compare_bases}
  Let $(X,a,t)$ be an oriented $\elRO$-algebra.
  Suppose that $X(1)_*$ is graded-connected as an $\mF_2$-algebra.
  Let $(b_i)_{i\in I}$ be a family of homogeneous elements of $X(C)_*$
  of positive degrees.
  Then the following conditions are equivalent:
  \begin{enumerate}[\em (a)]
  \item The classes $1$ and $\Gamma^n(b_i)$ for $n\geq 0$ and $i\in I$ form a
    basis of $X(C)_*$ as a graded module over $X(1)_*$.
  \item The classes $1$ and $b_i/1$ for $i\in I$ form a
    basis of $\Phi^C_* X$ as a graded module over the polynomial ring $X(1)_*[t/a]$.
  \end{enumerate}
\end{prop}
\begin{proof}
  Condition (a) and Proposition \ref{prop:exact_sequence} together imply that the classes
  $(t/a)^n$ and $\Gamma^n(b_i)/1$ for $n\geq 0$ and $i\in I$ form a
  basis of $X(C)_*$ as a graded module over $X(1)_*$.
  The defining relation of the division operator yields
  \[ \Gamma(x)/1 \ = \  (t/a)\cdot x/1 +\epsilon(x)\cdot (t/a)\ .  \]
  If the dimension of the class $x$ is positive, this shows that
  $\Gamma(x)/1$ is congruent to $(t/a)\cdot x/1$ modulo the graded ideal of
  positive degree elements in $X(1)$.
  By induction on $n$ we see that $\Gamma^n(x)/1$ is congruent to $(t/a)^n\cdot x/1$
  modulo the ideal of  positive degree elements in $X(1)$.
  In particular, the class $\Gamma^n(b_i)/1$ is congruent to $(t/a)^n\cdot b_i$.
  Since $X(C)_*$ is bounded below and $X(1)_*$ is a graded-connected $\mF_2$-algebra,
  this shows that the classes
  $(t/a)^n$ and $(t/a)^n\cdot b_i/1$ for $n\geq 0$ and $i\in I$ also form a
  basis of $X(C)_*$ as a graded module over $X(1)_*$. Hence the classes
  $1$ and $b_i/1$ for $i\in I$ form a basis of $X(C)_*$
  as a graded module over $X(1)_*[t/a]$, i.e., condition (b) holds.
  The other implication is proved by reversing the argument.
\end{proof}

\begin{cor}\label{cor:Alexander}
  Let $(X,a,t)$ be an oriented $\elRO$-algebra.
  Suppose that $X(1)_*$ is graded-connected as an $\mF_2$-algebra.
  Let $(y_i)_{i\in I}$ be a family of homogeneous elements of $X(C)_*$
  of positive degree such that $\Phi^C_* X$ is a polynomial algebra
  over $X(1)_*$ on the classes $(y_i)_{i\in I}$ and the class $t/a$.
  Then the classes $1$ and the classes $\Gamma^n(y_{i_1}\cdot\ldots\cdot y_{i_r})$
  for all $n\geq 0$, $r\geq 1$ and $i_1,\dots,i_r\in I$ form
  a basis of $X(C)_*$ as a graded module over $X(1)_*$.
\end{cor}
\begin{proof}
  We let $\mathcal B$ be the set of all monomials
  of degree at least one in the classes $y_i/1$ for all $i\in I$.
  The hypothesis that $\Phi^C_* X$ is polynomial over $X(1)_*$
  on the classes $(y_i)_{i\in I}$ and $t/a$
  is equivalent to the property that $\mathcal B$ is a basis of
  $\Phi^C_* X$ as module over $X(1)_*[t/a]$.
  Proposition \ref{prop:compare_bases} (a) then yields the claim.
\end{proof}

For equivariant bordism, we obtain Alexander's $\cN_*$-basis of $\cN^C_*$
by letting $y_i\in \cN_i^C$ be the bordism class of projective space
$P(\sigma \oplus \mR^i)$, for $i=2,3,\dots$.
The $C$-fixed points of $P(\sigma \oplus \mR^i)$ have two components:
an isolated fixed point $P(\sigma)$,
and the space $\mR P^{i-1}$ with normal bundle the tautological line bundle.
So Proposition \ref{prop:top SW power} shows that the classes $y_i/1$ for $i\geq 1$
and $a/t$ form a set of polynomial generators of $\Phi^C_* \cN$, as desired.
So Corollary \ref{cor:Alexander} applies, and we recover \cite[Theorem 1.2]{alexander}.

We can also apply Corollary \ref{cor:Alexander}
to the classes $\zeta_i\in \cN_{i+1}^C$ from Construction \ref{con:zeta}, for $i\geq 1$.
This way we obtain a different $\cN_*$-basis of $\cN^C_*$, given by 1
and the classes $\Gamma^n(\zeta_{i_1}\cdot\ldots\cdot \zeta_{i_r})$
for all $n\geq 0$, $r\geq 1$ and $i_1,\dots,i_r\geq 1$.

\subsection{Firsching's pullback square}
Finally we explain how a generalization of the main results of Firsching's paper \cite{firsching} drops out of our theory. In the following theorem we will write
\[ x_\lambda \ = \ t_\lambda/a_\lambda \ \in \ \Phi^A_1 X\ , \]
where $\lambda$ is any nontrivial $A$-character.
And we write  $(\Phi_*^A X)[x_\lambda^{-1}]$ for the localization
of the geometric fixed point ring away from the classes
$x_\lambda$ for all nontrivial $A$-characters $\lambda$.

\begin{theorem}\label{thm:pullback}
  Let $(X,a,t)$ be an oriented $\elRO$-algebra. Let $A$ be an elementary abelian 2-group
  such that for some (hence any) index two subgroup $K$ of $A$
  and every nontrivial $K$-character $\mu$,
  the class $x_\mu=t_\mu/a_\mu$ is a regular element of the ring $\Phi^K_* X$. Then the
  following square is a pullback of $\mZ$-graded rings:
    \[ \xymatrix@C=25mm{
        X(A)_*\ar[r]^-{-/1}\ar[d] & \Phi^A_* X\ar[d]\\
        (t^{-1} X)(A)_*\ar[r]_-{z/t_V\mapsto z/a_V\cdot x_V^{-1}} & (\Phi_*^A X)[x_{\lambda}^{-1}\colon\lambda\in A^\circ]
      } \]
\end{theorem}
\begin{proof}
  We claim that for every index two subgroup $K$ of $A$ and every $K$-character $\mu$,
  the class $t_\mu$ is a regular element in the $I_K$-graded ring $X(K,\star)$.
  Indeed, suppose that $y\cdot t_\mu=0$ for some $y\in X(K,k-V)$ with $k\in\mZ$ and $V^K=0$.
  Then 
  \[ 0 \ = \ (y\cdot t_\mu) / (a_V\cdot a_\mu) \ = \  (y/a_V)\cdot (t_\mu/a_\mu) \]
  in $\Phi_*^A X$. So $y/a_V=0$ by the regularity hypothesis.
  Because the map $-/a_V$ is injective, this implies $y=0$.
  Because all the classes $t_\mu$ for $\mu\in K^\circ$ are regular,
  we also deduce that for every $A$-representation $U$,
  the class $\res^A_K(t_U)=t_{U|_K}$ is a regular element in $X(K,\star)$.
  
  Now we turn to the proof of the theorem.
  Since the horizontal maps in the square are injective, we only need to show the following claim:
  let $y\in \Phi^A_k X$ be a class whose image in $(\Phi^A_*X)[x_\lambda^{-1}]$ is
  in the image of $(t^{-1}X)(A)_k$.  Then the class $y$ is effective.
  We suppose that $y=y'/a_W$ for some $A$-representation $W$ with trivial fixed points,
  and $y'\in X(A,k-W)$. By assumption there is 
  another $A$-representation $V$ with trivial fixed points and a class $z\in X(A,k-V)$
  such that $(z/a_V)\cdot x_V^{-1}=y=y'/a_W$ in the localization $(\Phi_*^A X)[x_{\lambda}^{-1}]$.
  Hence there is another $A$-representation $U$ with trivial fixed points such that 
  \[ (a_W\cdot z\cdot t_U)/a_{W\oplus V\oplus U}\ = \ 
    (z/a_V)\cdot x_U \ = \  (y'/a_W)\cdot x_V \cdot x_U \ = \
    (y'\cdot t_V\cdot t_U)/a_{W\oplus V\oplus U}   \]
  in the geometric fixed point ring $\Phi_*^A X$.
  Since $-/a_{W\oplus V\oplus U}$ is injective, this forces
  \[ z\cdot t_U\cdot a_W\ = \ y'\cdot t_{V\oplus U}  \]
  in $X(A,k-(W\oplus V\oplus U))$.
  Now we let $\lambda$ be an $A$-character with kernel $K$,
  and we let $m$ be the multiplicity of $\lambda$ in $W$.
  Then the left hand side of the previous equation is divisible by $a_\lambda^m$.
  Because the class $\res^A_K(t_{V\oplus U})$  is regular in $X(K,\star)$, 
  Proposition \ref{prop:coprime} (i) shows that also $y'$ is divisible by  $a_\lambda^m$.
  Proposition \ref{prop:coprime} (ii) then shows $y'$ is divisible by  $a_W$,
  i.e., $y'=y''\cdot a_W$ for a unique $y''\in X(A,k)$.
  Thus the class $y=y'/a_W=y''/1$ is effective.
\end{proof}

For the oriented $\elRO$-algebra $\cN$ of equivariant bordism,
all the geometric fixed point rings $\Phi^K \cN$ are polynomial algebras over
the non-equivariant bordism ring $\cN_*$, and hence they do not contain any zero-divisors.
So Theorem \ref{thm:pullback} shows that for every elementary abelian 2-group $A$,
the followings square is a pullback:
\[ \xymatrix{ \cN_*^A\ar[r] \ar[d] & \Phi^A_*\cN \ar[d]\\
    \mathfrak N_*^{A:S} \ar[r] & (\Phi^A_* \cN)[x_{\lambda}^{-1}\colon\lambda\in A^\circ]} \]
Here $\mathfrak N_*^{A:S} =t^{-1}\cN$ is stable $A$-equivariant bordism, 
see Construction \ref{con:stable_bordism}.
Firsching writes $e_\lambda^{-1}$ for $x_\lambda$,
and he uses the $A$-homology theory represented by the global Thom spectrum $\bMO$
instead of $\mathfrak N_*^{A:S}$; these equivariant homology theories
are isomorphic by \cite[Theorem 4.1]{broecker-hook}.
So modulo notation and the identification of $\Phi^A_*\cN$ with a polynomial $\cN_*$-algebra,
this is the main result Theorem 3.18 of \cite{firsching}.

\end{document}